\documentclass[11pt]{article}
\usepackage{amsmath,amsthm,amssymb}
\usepackage[usenames,dvipsnames]{xcolor}
\usepackage{enumerate}
\usepackage{graphicx}
\usepackage{cite}
\usepackage{comment}
\usepackage{oands}
\usepackage{tikz}
\usepackage{changepage}
\usepackage{bbm}
\usepackage{mathtools}
\usepackage[margin=1in]{geometry}
\usepackage[pagewise,mathlines]{lineno}
\usepackage{appendix}
\usepackage{stmaryrd} 
\usepackage{multicol}
\usepackage{microtype}
\usepackage[colorinlistoftodos]{todonotes}

\usepackage[pdftitle={Confluence of geodesics in Liouville quantum gravity},
  pdfauthor={Ewain Gwynne and Jason Miller},
colorlinks=true,linkcolor=NavyBlue,urlcolor=RoyalBlue,citecolor=PineGreen,bookmarks=true,bookmarksopen=true,bookmarksopenlevel=2,unicode=true,linktocpage]{hyperref}

\theoremstyle{plain}
\newtheorem{thm}{Theorem}[section]

\newtheorem{lem}[thm]{Lemma}
\newtheorem{prop}[thm]{Proposition}

\def\@rst #1 #2other{#1}
\newcommand\MR[1]{\relax\ifhmode\unskip\spacefactor3000 \space\fi
  \MRhref{\expandafter\@rst #1 other}{#1}}
\newcommand{\MRhref}[2]{\href{http://www.ams.org/mathscinet-getitem?mr=#1}{MR#2}}

\theoremstyle{definition}

\newtheorem{remark}[thm]{Remark}

\numberwithin{equation}{section}

\newcommand{\dsb}{\begin{adjustwidth}{2.5em}{0pt}
\begin{footnotesize}}
\newcommand{\dse}{\end{footnotesize}
\end{adjustwidth}}

\newcommand{\ssb}{\begin{adjustwidth}{2.5em}{0pt}}
\newcommand{\sse}{\end{adjustwidth}}

\newcommand{\aryb}{\begin{eqnarray*}}
\newcommand{\arye}{\end{eqnarray*}}
\def\alb#1\ale{\begin{align*}#1\end{align*}}
\def\allb#1\alle{\begin{align}#1\end{align}}
\newcommand{\eqb}{\begin{equation}}
\newcommand{\eqe}{\end{equation}}
\newcommand{\eqbn}{\begin{equation*}}
\newcommand{\eqen}{\end{equation*}}

\newcommand{\BB}{\mathbbm}
\newcommand{\ol}{\overline}

\newcommand{\op}{\operatorname}

\newcommand{\frk}{\mathfrak}

\newcommand{\ep}{\varepsilon}
\newcommand{\rta}{\rightarrow}

\newcommand{\wt}{\widetilde}
\newcommand{\wh}{\widehat} 
\newcommand{\mcl}{\mathcal}

\newcommand{\bdy}{\partial}
\newcommand{\rng}{\mathring}

\let\originalleft\left
\let\originalright\right
\renewcommand{\left}{\mathopen{}\mathclose\bgroup\originalleft}
\renewcommand{\right}{\aftergroup\egroup\originalright}

\title{Confluence of geodesics in\\ Liouville quantum gravity for $\gamma \in (0,2)$}
\date{ }
\author{Ewain Gwynne and Jason Miller \\ {\it University of Cambridge}}

\begin{document}

\maketitle

\begin{abstract}
We prove that for any metric which one can associate with a Liouville quantum gravity (LQG) surface for $\gamma \in (0,2)$ satisfying certain natural axioms, its geodesics exhibit the following confluence property.  For any fixed point $z$, a.s.\ any two $\gamma$-LQG geodesics started from distinct points other than $z$ must merge into each other and subsequently coincide until they reach $z$.  This is analogous to the confluence of geodesics property for the Brownian map proven by Le Gall (2010).  Our results apply for the subsequential limits of Liouville first passage percolation and are an important input in the proof of the existence and uniqueness of the LQG metric for all $\gamma\in (0,2)$.  
\end{abstract}

\tableofcontents

\section{Introduction}
\label{sec-intro}

\subsection{Overview}
\label{sec-overview}

Fix $\gamma \in (0,2)$, let $U\subset\BB C$, and let $h$ be a variant of the Gaussian free field (GFF) on $U$.
The theory of Liouville quantum gravity (LQG) is concerned with the random Riemannian metric 
\eqb \label{eqn-lqg-def}
e^{\gamma h} (dx^2 + dy^2), 
\eqe
on $U$, where $dx^2 + dy^2$ denotes the Euclidean Riemannian metric tensor. The surface parameterized by this Riemannian metric is called a \emph{$\gamma$-LQG surface}.

LQG was first introduced in the physics literate by Polyakov~\cite{polyakov-qg1} in the context of bosonic string theory.  
One reason why LQG surfaces are interesting mathematically is that they arise as the scaling limits of random planar maps: the special case when $\gamma=\sqrt{8/3}$ (called ``pure gravity") corresponds to the scaling limit of uniform random planar maps, and other values of $\gamma$ (sometimes referred to as ``gravity coupled to matter") correspond to random planar maps weighted by the partition function of an appropriate critical statistical mechanics model on the map.

The definition of $\gamma$-LQG given in~\eqref{eqn-lqg-def} does not make literal sense since the GFF is only a distribution, not a function. In particular, the GFF can be integrated against a smooth test function, but it does not have well-defined pointwise values so it cannot be exponentiated. 
Consequently, one needs a regularization procedure to make rigorous sense of this object. 
Previously, this has been accomplished for the associated volume form, i.e., the $\gamma$-LQG area measure.
This is a random measure $\mu_h$ on $U$ which is a limit of regularized versions of $e^{\gamma h} \, d^2 z$, where $d^2 z$ denotes Lebesgue measure~\cite{kahane,shef-kpz,rhodes-vargas-review}.  

It is a long-standing open problem to construct a canonical \emph{metric} associated with a $\gamma$-LQG surface, i.e., a random metric $D_h$ on $U$ which is obtained, in some sense, by exponentiating the GFF~$h$. The random metric space $(U,D_h)$ (for certain special variants of the GFF) should correspond to the scaling limit of random planar maps, equipped with their graph distance, with respect to the Gromov-Hausdorff topology. 
Miller and Sheffield~\cite{lqg-tbm1,lqg-tbm2,lqg-tbm3} constructed the LQG metric in the special case when $\gamma=\sqrt{8/3}$ using various special symmetries which are unique to this case. 
They also showed that for certain special choices of the pair $(U,h)$, the random metric space $(U,D_h)$ agrees in law with a \emph{Brownian surface}. 
Brownian surfaces, such as the Brownian map~\cite{legall-uniqueness,miermont-brownian-map} or the Brownian disk~\cite{bet-mier-disk} are continuum random metric spaces which arise as the scaling limits of uniform random planar maps in the Gromov-Hausdorff topology. 

This paper is part of a series of works which aims to construct the $\gamma$-LQG metric for all $\gamma \in (0,2)$.   
Analogously to the $\gamma$-LQG measure $\mu_h$, this metric will be defined as the limit of metrics induced by continuous approximations of the GFF. 
To describe these approximations, we first need to define the $\gamma$-LQG dimension exponent $d_\gamma > 2$ from~\cite{dg-lqg-dim}.
The value of this exponent is not known explicitly, but it can be defined in terms of various discrete approximations of LQG distances (such as random planar maps or Liouville first passage percolation, as discussed just below). 
Once the $\gamma$-LQG metric is constructed, it can be shown that $d_\gamma$ is its Hausdorff dimension~\cite{gp-kpz}. 
Let
\eqb \label{eqn-xi-def}
\xi := \frac{\gamma}{d_\gamma} .
\eqe

\begin{figure}[ht!]
\begin{center}
\includegraphics[width=0.6\textwidth]{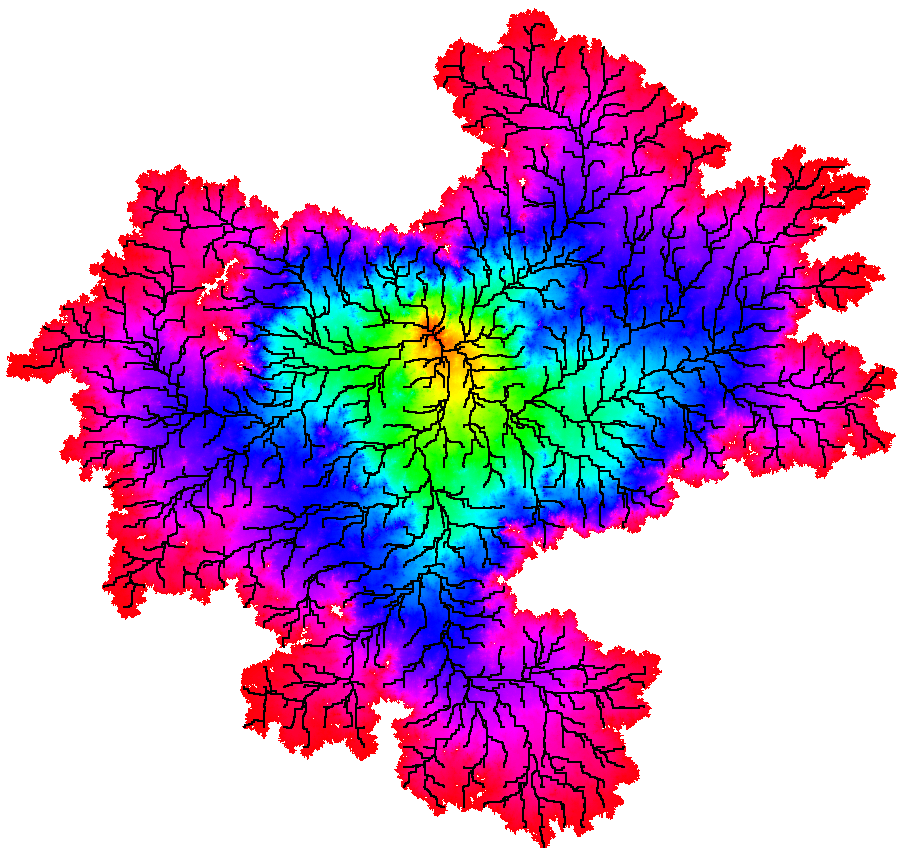}	
\end{center}
\vspace{-0.03\textheight}
\caption{\label{fig:geodesic_sim} A metric ball generated from a discrete GFF on a $1024 \times 1024$ subset of $\BB Z^2$ using the metric where the weight of each path $P$ is given by the sum of $e^{\xi h(x)}$ over those $x$ visited by $P$ where $\xi=1/\sqrt{6} = \gamma/d_\gamma$ for $\gamma=\sqrt{8/3}$ and $d_\gamma = 4$.  It is believed that these approximations fall into the same universality class as for other forms of LFPP, for example~\eqref{eqn-lfpp}.  All of the geodesics from the center of the metric ball to points in the intersection of the metric ball with a grid of spacing $20$ (i.e., $(20 \BB Z)^2$) are shown.} 
\end{figure}

For concreteness, we will primarily focus on the whole-plane case (but see Remark~\ref{remark-other-domains}). 
We say that a random distribution $h$ on $\BB C$ is a \emph{whole plane GFF plus a continuous function} if there exists a coupling of $h$ with a random continuous function $f : \BB C \rta \BB R$ such that the law of $h-f$ is that of a whole-plane GFF.
We similarly define a \emph{whole-plane GFF plus a bounded continuous function}, except that we also require that $f$ is bounded.
Let $p_s(z,w) := \frac{1}{2\pi s} \exp\left( - \frac{|z-w|^2}{2s} \right)$ be the heat kernel, so that $p_s(z,\cdot)$ approximates a point mass at $z$ when $s$ is small.
If $h$ is a whole-plane GFF plus a bounded continuous function, we define the convolution
\eqb \label{eqn-gff-convolve}
h_\ep^*(z) := (h*p_{\ep^2/2})(z) = \int_{\BB C}  h(w) p_{\ep^2/2}^U(z,w) \, d^2w  ,\quad\forall z\in\BB C,\quad\forall \ep > 0 ,
\eqe
where the integral is interpreted as a distributional pairing. 
For $z,w\in\BB C$ and $\ep> 0$, we define the \emph{$\ep$-Liouville first passage percolation (LFPP) metric} by\footnote{The intuitive reason why we look at $e^{\xi h_\ep(z)}$ instead of $e^{\gamma h_\ep(z)}$ to define the metric is as follows. Since $\mu_h$ is obtained by exponentiating $\gamma h$, we can scale LQG areas by a factor of $C>0$ by adding $\gamma^{-1}\log C$ to the field. By~\eqref{eqn-lfpp}, this results in scaling distances by $C^{\xi/\gamma} = C^{1/d_\gamma}$, which is consistent with the fact that the ``dimension" should be the exponent relating the scaling of areas and distances.}
\eqb \label{eqn-lfpp}
D_h^\ep(z,w) := \inf_{P : z\rta w} \int_0^1 e^{\xi h_\ep^*(P(t))} |P'(t)| \,dt 
\eqe
where the infimum is over all piecewise continuously differentiable paths from $z$ to $w$.  It was shown by Ding, Dub\'edat, Dunlap, and Falconet~\cite{dddf-lfpp} that the laws of the LFPP metrics, re-scaled by the median distance across a square, are tight and every subsequential limit induces the Euclidean topology on $\BB C$.  Subsequently, it was shown in~\cite{local-metrics,lqg-metric-estimates} that every subsequential limit can be realized as a measurable function of $h$, so in fact the convergence occurs in probability.  Our eventual aim is to show that the subsequential limit of LFPP is unique, so can be legitimately called \emph{the} $\gamma$-LQG metric.  

The contribution of the present paper is to study the geometry of geodesics for subsequential limits of LFPP. In particular, we will prove that (in contrast to geodesics for smooth metrics) such geodesics satisfy the following property. For any fixed point $z$, a.s.\ any two geodesics from distinct points and to $z$ will merge into each other and coincide for an interval of time before reaching $z$. 
Hence the LQG geodesics towards $z$ form a tree-like structure. 
This property is called \emph{confluence of geodesics}.
The analogous property for geodesics in the Brownian map was proven in~\cite{legall-geodesics}, and played an important role in the proof of the uniqueness of the Brownian map in~\cite{legall-uniqueness,miermont-brownian-map} and in the identification of the $\sqrt{8/3}$-LQG metric with the Brownian map~\cite{lqg-tbm1,lqg-tbm2,lqg-tbm3,tbm-characterization}. Likewise, the results of this paper are an important tool in the proof that the subsequential limit of LFPP is unique in~\cite{gm-uniqueness}. 

Our proofs do \emph{not} use very much external input. 
Indeed, aside from basic properties of the GFF (as can be found, e.g., in~\cite{shef-gff} and the introductory sections of~\cite{ss-contour,ig1,ig4}) we will only use a few lemmas from~\cite{mq-geodesics,local-metrics,lqg-metric-estimates} which can be easily taken as black boxes. 
In fact, even the proofs of these external lemmas do not use much beyond basic facts about the GFF. 
\bigskip

\noindent\textbf{Acknowledgments.} We thank Jian Ding, Julien Dub\'edat, Alex Dunlap, Hugo Falconet, Josh Pfeffer, Scott Sheffield, and Xin Sun for helpful discussions. We thank an anonymous referee for helpful comments on an earlier version of this paper. EG was supported by a Herchel Smith fellowship and a Trinity College junior research fellowship. JM was supported by ERC Starting Grant 804166.

\subsection{Weak LQG metrics}
\label{sec-weak-lqg-metric}

We will not work with LFPP directly in this paper. Instead, we will work in a more general framework involving a random metric which satisfies a list of axioms which are known to be satisfied for every subsequential limit of LFPP (see~\cite[Theorem 1.2]{lqg-metric-estimates}). To state these axioms precisely, we will need some preliminary definitions concerning metric spaces.
Let $(X,D)$ be a metric space.  
\medskip

\noindent
For $r>0$, we write $\mcl B_r(A;D)$ for the open ball consisting of the points $x\in X$ with $D (x,A) < r$.  
If $A = \{y\}$ is a singleton, we write $\mcl B_r(\{y\};d_X) = \mcl B_r(y;d_X)$.
\medskip

\noindent
For a curve $P : [a,b] \rta X$, the \emph{$D$-length} of $P$ is defined by 
\eqbn
\op{len}\left( P ; D  \right) := \sup_{T} \sum_{i=1}^{\# T} D(P(t_i) , P(t_{i-1})) 
\eqen
where the supremum is over all partitions $T : a= t_0 < \dots < t_{\# T} = b$ of $[a,b]$. Note that the $D$-length of a curve may be infinite.
\medskip

\noindent
For $Y\subset X$, the \emph{internal metric of $D$ on $Y$} is defined by
\eqb \label{eqn-internal-def}
D(x,y ; Y)  := \inf_{P \subset Y} \op{len}\left(P ; D \right) ,\quad \forall x,y\in Y 
\eqe 
where the infimum is over all paths $P$ in $Y$ from $x$ to $y$. 
Then $D(\cdot,\cdot ; Y)$ is a metric on $Y$, except that it is allowed to take infinite values.  
\medskip
 
\noindent
We say that $(X,D)$ is a \emph{length space} if for each $x,y\in X$ and each $\ep > 0$, there exists a curve of $D$-length at most $D(x,y) + \ep$ from $x$ to $y$. 
\medskip

\noindent
A \emph{continuous metric} on an open domain $U\subset\BB C$ is a metric $D$ on $U$ which induces the Euclidean topology on $U$, i.e., the identity map $(U,|\cdot|) \rta (U,D)$ is a homeomorphism. 
We equip the space of continuous metrics on $U$ with the local uniform topology for functions from $U\times U$ to $[0,\infty)$ and the associated Borel $\sigma$-algebra.
We allow a continuous metric to satisfy $D(u,v) = \infty$ if $u$ and $v$ are in different connected components of $U$.
In this case, in order to have $D^n\rta D$ w.r.t.\ the local uniform topology we require that for large enough $n$, $D^n(u,v) = \infty$ if and only if $D(u,v)=\infty$.
\medskip

\noindent
We are now ready to state the axioms under which we will work throughout the rest of the paper. 
Let $\mcl D'(\BB C)$ be the space of distributions (generalized functions) on $\BB C$, equipped with the usual weak topology.
For $\gamma \in (0,2)$, a \emph{weak $\gamma$-LQG metric} is a measurable function $h\mapsto D_h$ from $\mcl D'(\BB C)$ to the space of continuous metrics on $\BB C$ such that the following is true whenever $h$ is a whole-plane GFF plus a continuous function.
\begin{enumerate}[I.] 
\item \textbf{Length space.} Almost surely, $(\BB C , D_h)$ is a length space, i.e., the $D_h$-distance between any two points of $\BB C$ is the infimum of the $D_h$-lengths of $D_h$-continuous paths (equivalently, Euclidean continuous paths) between the two points. \label{item-metric-length}
\item \textbf{Locality.} Let $U\subset\BB C$ be a deterministic open set. 
The $D_h$-internal metric $D_h(\cdot,\cdot ; U)$ is a.s.\ determined by $h|_U$.  \label{item-metric-local}. 
\item \textbf{Weyl scaling.} Let $\xi$ be as in~\eqref{eqn-xi-def} and for each continuous function $f : \BB C\rta \BB R$, define \label{item-metric-f}  
\eqb \label{eqn-metric-f}
(e^{\xi f} \cdot D_h) (z,w) := \inf_{P : z\rta w} \int_0^{\op{len}(P ; D_h)} e^{\xi f(P(t))} \,dt , \quad \forall z,w\in \BB C ,
\eqe
where the infimum is over all continuous paths from $z$ to $w$ parameterized by $D_h$-length. Then a.s.\ $ e^{\xi f} \cdot D_h = D_{h+f}$ for every continuous function $f : \BB C\rta\BB R$.  
\item \textbf{Translation invariance.} For each deterministic point $z \in \BB C$, a.s.\ $D_{h(\cdot + z)} = D_h(\cdot+ z , \cdot+z)$.  \label{item-metric-translate}
\item \textbf{Tightness across scales.} Suppose that $h$ is a whole-plane GFF and let $\{h_r(z)\}_{r > 0, z\in\BB C}$ be its circle average process.\footnote{We refer to~\cite[Section 3.1]{shef-kpz} for the basic properties of the circle average process.}
For each $r > 0$, there is a deterministic constant $\frk c_r > 0$ such that the set of laws of the metrics $\frk c_r^{-1} e^{-\xi h_r(0)} D_h (r \cdot , r\cdot)$ for $r > 0$ is tight (w.r.t.\ the local uniform topology). Furthermore, the closure of this set of laws w.r.t.\ the Prokhorov topology on continuous functions $\BB C\times \BB C \rta [0,\infty)$ is contained in the set of laws on continuous metrics on $\BB C$ (i.e., every subsequential limit of the laws of the metrics $\frk c_r^{-1} e^{-\xi h_r(0)} D_h (r \cdot  , r \cdot )$ is supported on metrics which induce the Euclidean topology on $\BB C$). Finally, there exists  \label{item-metric-coord} 
$\Lambda > 1$ such that for each $\delta \in (0,1)$, 
\eqb \label{eqn-scaling-constant}
\Lambda^{-1} \delta^\Lambda \leq \frac{\frk c_{\delta r}}{\frk c_r} \leq \Lambda \delta^{-\Lambda} ,\quad\forall r  > 0.
\eqe
\end{enumerate}

The following theorem is~\cite[Theorem 1.2]{lqg-metric-estimates}, and is proven building on~\cite{dddf-lfpp,local-metrics}.

\begin{thm}[\!\!\cite{dddf-lfpp,local-metrics,lqg-metric-estimates}] \label{thm-lfpp-axioms}
Let $\gamma \in (0,2)$. For every sequence of positive $\ep$'s tending to zero, there is a weak $\gamma$-LQG metric $D$ and a subsequence $\{\ep_k\}_{k\in\BB N}$ for which the following is true. Let $h$ be a whole-plane GFF plus a bounded continuous function. Then the re-scaled LFPP metrics $\frk a_{\ep_k}^{-1} D_h^{\ep_k}$ from~\eqref{eqn-lfpp} converge in probability to $D_h$. 
\end{thm}

In light of Theorem~\ref{thm-lfpp-axioms}, to prove theorems about subsequential limits of LFPP it suffices to prove theorems about weak $\gamma$-LQG metrics.
A particular advantage of this approach is that the Miller-Sheffield $\sqrt{8/3}$-LQG metric~\cite{lqg-tbm1,lqg-tbm2,lqg-tbm3} is a weak $\sqrt{8/3}$-LQG metric (see~\cite[Section 2.5]{gms-poisson-voronoi} for a careful explanation of why this is the case). So, all of our results also apply to this metric. 
 
Let us now briefly comment on the above axioms. 
From the perspective that LQG is the random two-dimensional Riemannian manifold obtained by exponentiating the GFF, it is clear that Axioms~\ref{item-metric-length}, \ref{item-metric-local}, and~\ref{item-metric-f} should be true for any reasonable notion of a metric on LQG. 
It is also not hard to see, at least heuristically, why these axioms should be satisfied for subsequential limits of LFPP. 
 
It is expected that the LQG metric should satisfy a coordinate change formula when we apply a complex affine map which is directly analogous to the coordinate change formula for the LQG measure~\cite[Proposition 2.1]{shef-kpz}. 
In particular, it should be the case that for any $a,b\in\BB C$, $a \not=0$, a.s.\
\eqb \label{eqn-lqg-metric-coord}
 D_h \left( a\cdot + b , a \cdot + b \right) = D_{h(a\cdot + b)  +Q\log |a| }(\cdot,\cdot)  ,\quad \text{for} \quad Q =\frac{2}{\gamma} + \frac{\gamma}{2} .
\eqe
It is not known at this point that subsequential limits of LFPP satisfy~\eqref{eqn-lqg-metric-coord} (this will be proven in~\cite{gm-uniqueness}, using the result of the present paper). Axioms~\ref{item-metric-translate} and~\ref{item-metric-coord} are a substitute for the formula~\eqref{eqn-lqg-metric-coord}, which is why we differentiate between the case of a true (strong) LQG metric and a weak LQG metric.  
These axioms imply the tightness of various functionals of $D_h$. For example, if $U\subset\BB C$ is open and $K\subset\BB C$ is compact, then the laws of 
\eqb
\left( \frk c_r^{-1} e^{-\xi h_r(0)} D_h (r K , r\bdy U) \right)^{-1} \quad \text{and} \quad \frk c_r^{-1} e^{-\xi h_r(0)} \sup_{u,v\in r K} D_h (  u ,  v ; r U )
\eqe
as $r$ varies are tight. 
 

\begin{remark}[Metrics associated with other fields] \label{remark-other-domains}
Suppose that $D$ is a weak $\gamma$-LQG metric. 
Then $D_h$ is defined whenever $h$ is a whole-plane GFF plus a continuous function. 
It is not hard to see that one can also define the metric if $h$ is equal to a whole-plane GFF plus a continuous function plus a finite number of logarithmic singularities of the form $-\alpha\log|\cdot - z|$ for $z\in\BB C$ and $\alpha  < Q$; see~\cite[Theorem 1.10 and Proposition 3.17]{lqg-metric-estimates}. 

We can also define metrics associated with GFF's on proper sub-domains of $\BB C$. To this end, let $U\subset\BB C$ be open and let $h$ be a whole-plane GFF. Due to Axiom~\ref{item-metric-local}, we can define for each open set $U\subset \BB C$ the metric $D_{h|_U} := D_h(\cdot,\cdot ;U)$ as a measurable function of $h|_U$. 
We can write $h|_U = \rng h^U + \frk h^U$, where $\rng h^U$ is a zero-boundary GFF on $U$ and $\frk h^U$ is a random harmonic function on $U$ independent from $\rng h^U$.
In the notation~\eqref{eqn-metric-f}, we define
\eqb
D_{\rng h^U} := e^{-\xi \frk h^U} \cdot D_{h|_U} .
\eqe
It is easily seen from Axioms~\ref{item-metric-local} and~\ref{item-metric-f} that $D_{\rng h^U}$ is a measurable function of $\rng h^U$. 
Indeed, if we are given an open set $V\subset U$ with $\ol V\subset U$, choose a smooth compactly supported bump $f : U\rta [0,1]$ which is identically equal to 1 on $V$. 
Then Axiom~\ref{item-metric-local} applied to the field $h - f \frk h^U$ implies that the internal metric of $D_{\rng h^U}$ on $V$, which equals $D_{h-f\frk h^U}(\cdot,\cdot ; V)$, is a.s.\ determined by $(h-f\frk h^U)|_V = \rng h^U|_V$. Letting $V$ increase to all of $U$ gives the desired measurability of $D_{\rng h^U}$ w.r.t.\ $\rng h^U$. 
This defines the $\gamma$-LQG metric for a zero-boundary GFF. Using Axiom~\ref{item-metric-f}, we can similarly define the metric for a zero-boundary GFF plus a continuous function on $U$.  
\end{remark}

\subsection{Main results}
\label{sec-main-results}

Let $\gamma \in (0,2)$, let $D$ be a weak $\gamma$-LQG metric, and let $h$ be a whole-plane GFF. For concreteness, we can fix the additive constant for $h$ so that its average on the unit circle is equal to $0$ but we note that by Axiom~\ref{item-metric-f} the geodesics associated with $D_h$ do not depend on how the additive constant is fixed.
It is easy to check that $D_h$ is boundedly compact, in the sense that closed, $D_h$-bounded subsets of $\BB C$ are compact~\cite[Lemma 3.8]{lqg-metric-estimates}. By Axiom~\ref{item-metric-length} and a standard result in metric geometry~\cite[Corollary 2.5.20]{bbi-metric-geometry} this implies that for any two points in $\BB C$ can be joined by at least one path of minimal $D_h$-length, which we call a \emph{$D_h$-geodesic}. 

Recall that for $s > 0$ and $z\in\BB C$, we write $\mcl B_s(z;D_h)$ for the $D_h$-metric ball of radius $s$ centered at $z$. 
The simplest form of confluence which we establish in this paper is the following statement. See Figure~\ref{fig-thm-illustration}, left, for an illustration. 
 
\begin{thm}[Confluence of geodesics at a point] \label{thm-clsce}
 Almost surely, for each radius $s > 0$ there exists a radius $t \in (0,s)$ such that any two $D_h$-geodesics from $z$ to points outside of $\mcl B_s(z;D_h)$ coincide on the time interval $[0,t]$.
\end{thm} 

We emphasize that the property of $D_h$-geodesics described in Theorem~\ref{thm-clsce} is very different from the behavior of geodesics w.r.t.\ a smooth Riemannian metric on $\BB C$: indeed, in the latter situation geodesics from two different points targeted at $z$ intersect only at $z$. 

It was shown by Le Gall that the analog of Theorem~\ref{thm-clsce} holds for geodesics in the Brownian map based at a typical point sampled from the volume measure on the Brownian map~\cite[Corollary~7.7]{legall-geodesics}.  
Due to the equivalence of Brownian and $\sqrt{8/3}$-LQG surfaces, this is equivalent to the statement that one a.s.\ has confluence of $\sqrt{8/3}$-LQG geodesics at a typical sampled uniformly from the $\sqrt{8/3}$-LQG area measure.
However, Theorem~\ref{thm-clsce} has new content even in the case of the $\sqrt{8/3}$-LQG metric since it gives confluence at a \emph{Lebesgue} typical point, rather than a quantum typical point.

\begin{figure}[t!]
 \begin{center}
\includegraphics[scale=.85]{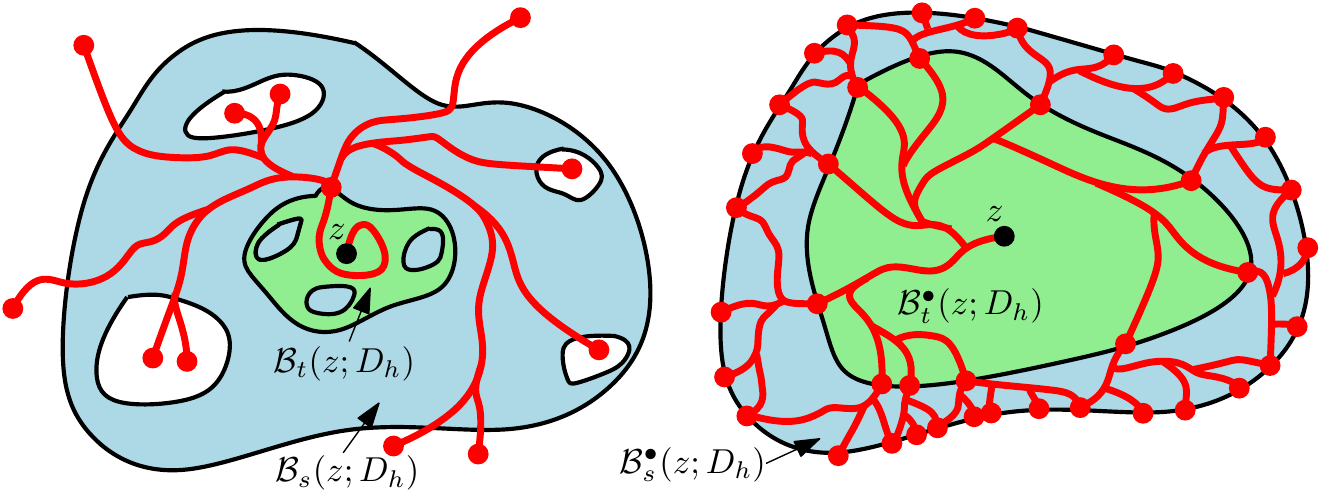}
\vspace{-0.01\textheight}
\caption{\textbf{Left:} Theorem~\ref{thm-clsce} asserts that for a fixed $z\in\BB C$, it is a.s.\ the case that for each $s>0$ and each small enough $t>0$ (depending on $s$), all of the $D_h$-geodesics from $z$ to points outside of $\mcl B_s(z;D_h)$ coincide until they exit $\mcl B_t(z;D_h)$. Several such $D_h$-geodesics are shown in red.
\textbf{Right:} Theorem~\ref{thm-finite-geo0} asserts that a.s.\ for \emph{any} $0<t<s<\infty$ there are only finitely many possibilities for the restriction of a $D_h$-geodesic from 0 to $\bdy\mcl B_s^\bullet(0;D_h)$ to the time interval $[0,t]$. Several such leftmost $D_h$-geodesics are shown in red. 
}\label{fig-thm-illustration}
\end{center}
\vspace{-1em}
\end{figure} 

Another form of confluence concerns geodesics across an annulus between two filled $D_h$-metric balls. 
For $s > 0$ and $z\in\BB C$, we define the \emph{filled metric ball} $\mcl B_s^\bullet(z;D_h)$ to be the union of $\ol{\mcl B_s(z;D_h)}$ and the set of points in $\BB C$ which are disconnected from $\infty$ by $\ol{\mcl B_s (z;D_h)}$. 
Each point $x\in \bdy \mcl B_s^\bullet(z;D_h)$ lies at $D_h$-distance exactly $s$ from $z$, so every $D_h$-geodesic from $z$ to $x$ stays in $\mcl B_s^\bullet(z;D_h)$. For some atypical points $x$ there might be many such $D_h$-geodesics, but there is always a distinguished $D_h$-geodesic from $z$ to $x$, called the \emph{leftmost geodesic}, which lies (weakly) to the left of every other $D_h$-geodesic from 0 to $x$ if we stand at $x$ and look outward from $ \mcl B_s^\bullet(z;D_h) $ (see Lemma~\ref{lem-leftmost-geodesic}). 

\begin{thm}[Confluence of geodesics across a metric annulus] \label{thm-finite-geo0}
Fix $z\in\BB C$. Almost surely, for each $0 < t < s < \infty$ there is a finite set of $D_h$-geodesics from 0 to $\bdy\mcl B_t^\bullet(z;D_h)$ such that every leftmost $D_h$-geodesic from 0 to $\bdy\mcl B_s^\bullet(z;D_h)$ coincides with one of these $D_h$-geodesics on the time interval $[0,t]$. In particular, there are a.s.\ only finitely many points of $\bdy\mcl B_t^\bullet(z;D_h)$ which are hit by leftmost $D_h$-geodesics from 0 to $\bdy\mcl B_s^\bullet(z;D_h)$. 
\end{thm}
 
See Figure~\ref{fig-thm-illustration}, right, for an illustration of the statement of Theorem~\ref{thm-finite-geo0}. 
The proof of Theorem~\ref{thm-finite-geo0}, given in Section~\ref{sec-confluence}, is in some sense the core of the paper. 
See Section~\ref{sec-conf-outline} for an overview of this argument.
Theorem~\ref{thm-clsce} will be deduce from Theorem~\ref{thm-finite-geo0} in Section~\ref{sec-one-geodesic}. 

We will actually prove much more quantitative versions of Theorem~\ref{thm-finite-geo0} which give upper bounds for the number of leftmost $D_h$-geodesics across an annulus between two filled LQG metric balls which are uniform with respect to the Euclidean size of the LQG metric balls. See Theorems~\ref{thm-finite-geo} and~\ref{thm-finite-geo-quant}.
This will be important in~\cite{gm-uniqueness}. 
However, once we know that every weak LQG metric satisfies the scale invariance property~\eqref{eqn-lqg-metric-coord} (which will be established in~\cite{gm-uniqueness}), Theorem~\ref{thm-finite-geo} itself implies an estimate which is uniform w.r.t.\ the Euclidean size of the LQG metric balls. 

We now briefly discuss how the results of this paper are used in~\cite{gm-uniqueness}. Very roughly, the reason why confluence is useful in this setting is that it allows us to establish near-independence for events which depend on small segments of a $D_h$-geodesic which are separated from each other. 
To be more precise, fix distinct points $\BB z,\BB w\in\BB C$ and let $P$ be the (a.s.\ unique; see Lemma~\ref{lem-geo-unique} below) $D_h$-geodesic from $\BB z$ to $\BB w$.
Suppose we are given $0 < t < s < \infty$ and we want to study, for a small $\ep > 0$, the conditional law of $h|_{B_\ep(P(s))}$ given $P|_{[0,t]}$. 
A priori, this seems to be a very intractable conditioning since we do not know anything about the conditional law of $h$ given $P|_{[0,t]}$. 
However, we can get around this problem using Theorem~\ref{thm-finite-geo0}, as follows. 

Let $\delta \in (0,1)$ be a small parameter. If $\delta$ is chosen small enough relative to $t$ and then $\ep$ is chosen small enough relative to $\delta$, then a slightly more quantitative version of Theorem~\ref{thm-finite-geo0} tells us that the following is true. With high probability, there is an arc $I\subset\bdy\mcl B_{s-\delta}^\bullet$ such that (a) all of the leftmost $D_h$-geodesics from $\BB z$ to $I$ agree on the time interval $[0,t]$ and (b) we have $B_\ep(P(s))\subset \BB C\setminus \mcl B_s^\bullet$ and every point of $B_\ep(P(s))$ is much closer to $I$ than to $\bdy\mcl B_{s-\delta}^\bullet\setminus I$ (to keep $P(s)$ away from the endpoints of $I$, one can use Lemma~\ref{lem-geo-kill-pt} below). 
Condition (b) tells us that $P(s-\delta) \in I$ and moreover this is still the case if we change the behavior of the field in $B_\ep(P(s))$ in a ``reasonable" way.
Hence we can change what happens in $B_\ep(P(s))$ without changing $P|_{[0,t]}$.

\begin{remark}[Confluence for more general fields] \label{remark-other-points}
We have stated Theorems~\ref{thm-clsce} and~\ref{thm-finite-geo0} for a whole-plane GFF, but the statements extend immediately to other variants of the GFF via local absolute continuity. 
The proofs of Theorems~\ref{thm-clsce} and~\ref{thm-finite-geo0} also work if we replace $h$ by $h-\alpha\log|\cdot|$ for $\alpha \in (-\infty,Q)$, in which case $D_{h^\alpha}$ makes sense as a continuous metric on $\BB C$.  
In particular, by taking $\alpha=\gamma$, we get that confluence of geodesics holds a.s.\ at a ``quantum typical point". 
However, confluence of geodesics does not hold a.s.\ at all points simultaneously.
For example, if $z$ is a point on a $D_h$-geodesic $P$ which is not one of the endpoints of $P$, then there are two distinct incoming geodesic ``arms" emanating from $P$ which do not coalesce, namely the segments of $P$ before and after it hits $z$. That is, the conclusion of Theorem~\ref{thm-clsce} does not hold for such a choice of $z$. 
\end{remark}

\subsection{Basic notation}
\label{sec-notation}

\noindent
We write $\BB N = \{1,2,3,\dots\}$ and $\BB N_0 = \BB N \cup \{0\}$. 
\medskip

\noindent
For $a < b$, we define the discrete interval $[a,b]_{\BB Z}:= [a,b]\cap\BB Z$. 
\medskip

\noindent
If $f  :(0,\infty) \rta \BB R$ and $g : (0,\infty) \rta (0,\infty)$, we say that $f(\ep) = O_\ep(g(\ep))$ (resp.\ $f(\ep) = o_\ep(g(\ep))$) as $\ep\rta 0$ if $f(\ep)/g(\ep)$ remains bounded (resp.\ tends to zero) as $\ep\rta 0$. We similarly define $O(\cdot)$ and $o(\cdot)$ errors as a parameter goes to infinity. 
\medskip

\noindent
If $f,g : (0,\infty) \rta [0,\infty)$, we say that $f(\ep) \preceq g(\ep)$ if there is a constant $C>0$ (independent from $\ep$ and possibly from other parameters of interest) such that $f(\ep) \leq  C g(\ep)$. We write $f(\ep) \asymp g(\ep)$ if $f(\ep) \preceq g(\ep)$ and $g(\ep) \preceq f(\ep)$. 
\medskip

\noindent
We will often specify any requirements on the dependencies on rates of convergence in $O(\cdot)$ and $o(\cdot)$ errors, implicit constants in $\preceq$, etc., in the statements of lemmas/propositions/theorems, in which case we implicitly require that errors, implicit constants, etc., appearing in the proof satisfy the same dependencies. 
\medskip
  
\noindent
For $z\in\BB C$ and $r>0$, we write $B_r(z)$ for the Euclidean ball of radius $r$ centered at $z$. We also define the open annulus
\eqb \label{eqn-annulus-def}
\BB A_{r_1,r_2}(z) := B_{r_2}(z) \setminus \ol{B_{r_1}(z)} ,\quad\forall 0 < r_r < r_2 < \infty .
\eqe

\section{Preliminaries}
\label{sec-prelim}

In this section, we prove some preliminary results which are needed for the main argument in Section~\ref{sec-confluence}. 
In Section~\ref{sec-geodesic}, we prove some basic monotonicity properties of $D_h$-geodesics and construct the leftmost geodesics appearing in Theorem~\ref{thm-finite-geo0}.
In Section~\ref{sec-fkg}, we prove a version of the FKG inequality for $D_h$, which says that certain non-decreasing functionals of $D_h$ are positively correlated.
In Section~\ref{sec-annulus-iterate} we state a general lemma from~\cite{mq-geodesics,local-metrics} concerning the approximate independence of the restrictions of the GFF to concentric annuli. 
In Section~\ref{sec-disconnect-set}, we prove some deterministic geometric lemmas which will be used to control the geometry of filled $D_h$-metric balls. 

\subsection{Basic properties of LQG metric balls and geodesics}
\label{sec-geodesic}

Let $D$ be a weak $\gamma$-LQG metric for some $\gamma \in (0,2)$ and let $h$ be a whole-plane GFF. 
For $s >0$ and $z\in\BB C$, we define the filled metric ball $\mcl B_s^\bullet(z;D_h)$ as in the discussion just above Theorem~\ref{thm-finite-geo0}. 

We recall the definition of a local set of the GFF from~\cite[Lemma 3.9]{ss-contour}. Suppose $(h,A)$ is a coupling of $h$ with a random set $A$. 
We say that $A$ is a \emph{local set} for $h$ if for any open set $U\subset \BB C$, the event $\{A\cap U \not=\emptyset\}$ is conditionally independent from $h|_{\BB C\setminus U}$ given $h|_U$. If $A$ is determined by $h$ (which will be the case for all of the local sets we consider), this is equivalent to the statement that $A$ is determined by $h|_U$ on the event $\{A\subset U\}$. 
For a local set $A$, we can condition on the pair $(A,h|_A)$: this is by definition the same as conditioning on the $\sigma$-algebra $\bigcap_{\ep > 0} \sigma(A , h|_{B_\ep(A)})$. 
The conditional law of $h|_{\BB C\setminus A}$ given $(A,h|_A)$ is that of a zero-boundary GFF on $\BB C\setminus A$ plus a harmonic function on $\BB C\setminus A$ which is determined by $(A,h|_A)$. 

\begin{lem} \label{lem-ball-local}
Let $z\in\BB C$ be deterministic. 
If $\tau$ is a stopping time for the filtration generated by $(\ol{\mcl B_s(z;D_h)} , h|_{\ol{\mcl B_s(z;D_h)}})$, then $\ol{\mcl B_\tau(z;D_h)}$ is a local set for $h$.
The same is true with $\mcl B_s^\bullet(z;D_h)$ in place of $ \ol{\mcl B_s(z;D_h)}$.  
\end{lem}
\begin{proof}
The lemma in the case of an un-filled metric ball follows from a general fact about local metrics of the Gaussian free field, see~\cite[Lemma 2.2]{local-metrics}.
We will now treat the case of a filled metric ball. We will use the following criterion from~\cite[Lemma 3.9]{ss-contour}: a closed set $A$ coupled with the GFF $h$ is a local set if and only if for each open set $U\subset\BB C$, the event $\{A\subset U\}$ is conditionally independent from $h|_{\BB C\setminus U}$ given $h|_U$.
For a deterministic radius $s$, the event $\{\mcl B_s^\bullet(z;D_h) \subset U\}$ is the same as the event that $\ol{\mcl B_s(z;D_h)} \subset U$ and each bounded connected component of $U\setminus \ol{\mcl B_s(z;D_h)}$ is contained in $U$. 
Hence $\{\mcl B_s^\bullet(z;D_h) \subset U\}$ is determined by the event $\{\ol{\mcl B_s(z;D_h)} \subset U\}$ and the set $\ol{\mcl B_s(z;D_h)}$ on this event.
By~\cite[Lemma 3.9]{ss-contour} and the locality of $\ol{\mcl B_s(z;D_h)}$, it follows that $\{\mcl B_s^\bullet(z;D_h) \subset U\}$ is determined by $h|_U$, so $\mcl B_s^\bullet(z;D_h)$ is a local set.
The case of stopping times which take on only countably many possible values is immediate from the case of deterministic times. The case of general stopping times follows from the standard strong Markov property argument (i.e., look at the times $2^{-n}  \lceil 2^n \tau \rceil$ and send $n\rta\infty$) and the fact that local sets behave well under limits~\cite[Lemma 6.8]{qle}. 
\end{proof}

We abbreviate
\eqb \label{eqn-filled-ball}
\mcl B_s^\bullet := \mcl B_s^\bullet(0;D_h) .
\eqe
Note that each point $x \in \bdy\mcl B_s^\bullet$ lies at $D_h$-distance exactly $s$ from 0. 
Hence there is a $D_h$-geodesic $P$ from 0 to $x$ which stays in $\mcl B_s^\bullet$. 
For each $t < s$, such a geodesic satisfies $P(t) \in \bdy\mcl B_t^\bullet$: indeed, $P(t) \in \bdy\mcl B_t(0;D_h)$ by the definition of a geodesic and $P$ has to pass through $\bdy\mcl B_t^\bullet$ since $\bdy\mcl B_t^\bullet$ disconnects 0 from $x$. 
Conversely, if $t < s$ and $P'$ is a $D_h$-geodesic from $\bdy\mcl B_t^\bullet$ to $\bdy\mcl B_s^\bullet$ (i.e., a path of length $s-t$ between these sets), then the concatenation of any $D_h$-geodesic from 0 to $P'(0)$ with $P'$ is a path of length $s$ from 0 to $P'(s-t) \in \bdy\mcl B_s^\bullet$, so is a $D_h$-geodesic from 0 to $P'(s-t)$.  

The goal of this subsection is to prove various monotonicity statements for $D_h$-geodesics from 0 to $\bdy\mcl B_s^\bullet$ which, roughly speaking, tell us that such geodesics have to stay in cyclic order. There is some subtlety since there can be multiple $D_h$-geodesics from 0 to some points of $\bdy\mcl B_s^\bullet$, and moreover some geodesics might share a non-trivial segment, so one cannot expect to have exact monotonicity.

The starting point of our proofs is the following result from~\cite{mq-geodesics}. 

\begin{lem}[\!\!\cite{mq-geodesics}] \label{lem-geo-unique}
Almost surely, for each $q\in\BB Q^2$ there is only one $D_h$-geodesic from 0 to $q$. 
\end{lem}
\begin{proof}
This follows from the proof of~\cite[Theorem 1.2]{mq-geodesics}. We note that the theorem is stated for a strong LQG metric, but the proof does not use the coordinate change axiom. 
\end{proof}

As a consequence of Lemma~\ref{lem-geo-unique}, we get the following lemma which says that $D_h$-geodesics started from 0 cannot cross $D_h$-geodesics from 0 to rational points.
This will be the main tool in our monotonicity arguments.

\begin{lem} \label{lem-non-cross}
For $q\in\BB Q^2$, let $P_q$ be the a.s.\ unique $D_h$-geodesic from 0 to $q$. The following holds a.s. If $q\in\BB Q^2$, $P'$ is a $D_h$-geodesic started from 0, and $u \in P_q\cap P'$, then there is a time $s \geq 0$ such that $P_q(s) = P'(s) = u$ and $P_q(t) = P'(t)$ for each $t \in [0,s]$.  
\end{lem}

Lemma~\ref{lem-non-cross} says that if a $D_h$-geodesic started from a rational point and an arbitrary $D_h$-geodesic started from 0 meet after time 0, then they have to agree for a non-trivial interval of time. We emphasize, however, that we have not yet proven that there are \emph{any} such pairs of $D_h$-geodesics started from 0 which meet, so Lemma~\ref{lem-non-cross} does not imply any sort of confluence. 

\begin{proof}[Proof of Lemma~\ref{lem-non-cross}]
Suppose $u\in P_q \cap P'$ is as in the statement of the lemma. 
Since each of $P_q$ and $P'$ is a $D_h$-geodesic started from 0, the time when each of $P_q$ and $P'$ hits $u$ is equal to $s := D_h(0,u)$. 
Therefore $P_q(s) = P'(s) =u$. 
The concatenation of $P'|_{[0,s]}$ and $P_q|_{[s,D_h(0,q)]}$ is a $D_h$-geodesic from 0 to $q$. This $D_h$-geodesic must coincide with $P_q$ by the uniqueness of $D_h$-geodesics between rational points, whence $P_q(t) = P'(t)$ for each $t \in [0,s]$.
\end{proof}

For general points $y\in \bdy\mcl B_s^\bullet$, there might be many (even infinitely many) $D_h$-geodesics from 0 to $y$. 
However, there are two distinguished geodesics which are in some sense maximal, in the following sense. 

\begin{lem} \label{lem-leftmost-geodesic} 
Almost surely, for each $s>0$ and each $y\in \bdy \mcl B_s^\bullet$, there exists a unique \emph{leftmost (resp.\ rightmost) geodesic} $P_y^-$ (resp.\ $P_y^+$) from 0 to $y$ such that each $D_h$-geodesic from 0 to $y$ lies (weakly) to the right (resp.\ left) of $P_y^-$ (resp.\ $P_y^+$) if we stand at $y$ and look outward from $\mcl B_{s }^\bullet$. 
Moreover, there are sequences of points $q_n^- , q_n^+ \in \BB Q^2 \setminus\mcl B_s^\bullet$ such that the $D_h$-geodesics from 0 to $q_n^\pm$ satisfy $P_{q_n^\pm} \rta P_y^\pm$ uniformly. 
\end{lem}

\begin{figure}[t!]
 \begin{center}
\includegraphics[scale=.75]{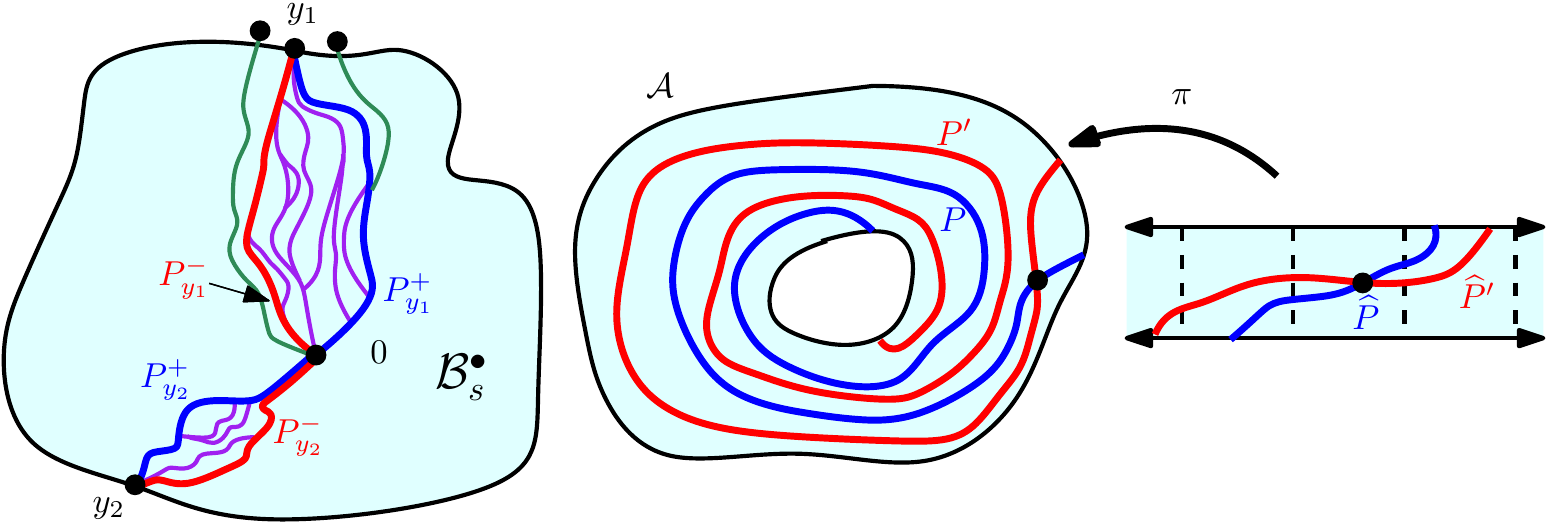}
\vspace{-0.01\textheight}
\caption{\textbf{Left:} Two points $y_1,y_2 \in \bdy \mcl B_s^\bullet$ and their associated leftmost and rightmost $D_h$-geodesics (red and blue). Other $D_h$-geodesics from 0 to $y_1$ and $y_2$ are shown in purple. We have also shown two $D_h$-geodesics from 0 to points of $\BB Q^2\setminus \mcl B_s^\bullet$ (green) which approximate $P_{y_1}^-$ and $P_{y_1}^+$, respectively. 
Note that $P_{y_1}^-$ and $P_{y_1}^+$ intersect only at their endpoints, whereas $P_{y_2}^-$ and $P_{y_2}^+$ coincide for an initial time interval. Theorem~\ref{thm-clsce} implies that a.s.\ the latter situation holds simultaneously for every $y\in \bdy\mcl B_s^\bullet$, but this has not been established yet. 
\textbf{Middle:} Illustration of the statement and proof of Lemma~\ref{lem-winding-intersect}. Two paths $P,P'$ in an annular region $\mcl A$ are shown, with $\op{wind}(P) $ slightly smaller than 2 and $\op{wind}(P')$ slightly larger than 3. 
\textbf{Right:} To show that $P$ and $P'$ intersect, we consider their lifts to the universal cover of $\mcl A$ chosen so that $\wt P(0)$ lies to the right of $\wh P'(0)$. Then $\wh P(1)$ lies to the left of $\wh P'(1)$, which forces the paths to intersect. 
}\label{fig-leftmost}
\end{center}
\vspace{-1em}
\end{figure} 

See Figure~\ref{fig-leftmost}, left, for an illustration of the statement and proof of Lemma~\ref{lem-leftmost-geodesic}. 
Note that if $0 <t < s$, then $P_x^-|_{[0,t]}$ is the leftmost $D_h$-geodesic from 0 to $P_x^-(t)$, since the concatenation of $P_{P_x^-(t)}^-$ and $P_x^-|_{[t,s]}$ is a $D_h$-geodesic which lies (weakly) to the left of $P_x^-$. 
We remark that \cite[Proposition 2.2]{tbm-characterization} gives an alternative definition and construction of leftmost geodesics, but we will give an independent proof since we need the approximation by geodesics to rational points.

\begin{proof}[Proof of Lemma~\ref{lem-leftmost-geodesic}]
We will construct $P_y^-$; the construction of $P_y^+$ is analogous. 
Fix a point $z \in \bdy\mcl B_s^\bullet\setminus \{y\}$. 
We can choose a sequence $y_n^-$ of points in the clockwise arc of $\bdy\mcl B_s^\bullet$ from $z$ to $y$ which converge to $y$ from the left (with ``left" defined as in the lemma statement). 
Since $D_h$ induces the Euclidean topology, for each $n$ we can find $q_n^- \in \BB Q^2\setminus\mcl B_s^\bullet$ such that $D_h(q_n^- , y_n^-)$ is smaller than half of the distance from $q_n^-$ to any point of $\bdy\mcl B_s^\bullet$ in the clockwise arc of $\bdy\mcl B_s^\bullet$ from $y$ to $z$. 
Then the points $P_{q_n^-}(s)$ converge to $y$ from the left.
Since each $P_{q_n^-}$ is a geodesic, the Arz\'ela-Ascoli theorem implies that after possibly passing to a subsequence, we can arrange that the paths $P_{q_n^-}$'s converge uniformly to a continuous path $P_y^-$. The path $P_y^-$ is a $D_h$-geodesic from 0 to $y$.
By Lemma~\ref{lem-non-cross}, no $D_h$-geodesic from 0 to $y$ can cross any of the $P_{q_n^-}$'s. 
It follows that each such geodesic lies to the right of $P_y^-$. 
\end{proof}

We next need to rule out the possibility that some $D_h$-geodesics started from 0 wind around the origin more than others, and thereby cross several other $D_h$-geodesics. It is possible for a $D_h$-geodesic to wind around 0 an arbitrarily large number of times, so instead of bounding the amount of winding we will show that it is approximately the same for all of the $D_h$-geodesics started from 0 (Lemma~\ref{lem-geo-winding}). 
To make this precise, we need to define the winding number of a curve in an annulus. 
Let $\mcl A$ be a topological space homeomorphic to a closed annulus and let $\bdy_{\op{in}} \mcl A$ and $\bdy_{\op{out}}\mcl A$ be its inner and outer boundaries. 
Let $\pi : \BB R\times [0,1] \rta \mcl A$ be a universal covering map normalized so that $\pi(\BB R\times\{0\}) = \bdy_{\op{in}} \mcl A$, $\pi(\BB R\times\{1\}) = \bdy_{\op{out}} \mcl A$, and $\pi(x , y) = \pi(x + n  ,  y)$ for each $n\in\BB Z$ and $(x,y) \in \BB R\times [0,1]$.
For a path $P : [a,b] \rta\mcl A$, let $\wh P : [a,b] \rta \BB R\times [0,1]$ be a lift of $P$ with respect to $\pi$. We define the \emph{winding number} of $P$ by
\eqb
\op{wind}(P) := \left( \text{First coordinate of $\wh P(b)$} \right) - \left( \text{First coordinate of $\wh P(a)$} \right) .
\eqe
Note that $\op{wind}(P)$ does not depend on the choice of lift due to the periodicity of $\pi$. 
We also note that $\op{wind}(P)$ is not required to be an integer since we consider general paths, not loops.

\begin{lem} \label{lem-winding-intersect}
Let $\mcl A$ be as above and let $P, P'$ be paths in $\mcl A$ from $\bdy_{\op{in}}\mcl A$ to $\bdy_{\op{out}}\mcl A$ which satisfy $\op{wind}(P) + 1  < \op{wind}(P') $. 
Then $P$ and $P'$ must cross each other in the sense that there is a point $w \in P\cap P'$ such that the segments of $P$ and $P'$ before hitting $w$ do not coincide. 
\end{lem}
\begin{proof}
The proof is completely elementary but we give the details for the sake of completeness. See Figure~\ref{fig-leftmost}, middle and right, for an illustration. 
Let $\pi : \BB R\times [0,1] \rta \mcl A$ be the universal covering map from above. 
We parameterize our paths $P , P' : [0,1] \rta \mcl A$.  

Let $\wh P : [0,1] \rta \BB R\times [0,1]$ be a lift of $P$ with respect to $\pi$ started from a point in $[0,1) \times \{0\}$ and let $\wh x \in [0,1)$ be the first coordinate of $\wh P(0)$. 
We claim that we can choose a lift $\wh P' :[0,1] \rta \BB R\times [0,1]$ of $ P'$ such that the first coordinate $\wh x'$ of $\wh P'(0)$ lies in $[\wh x  -1 , \wh x)$. Indeed, if the starting point of the lift of $P'$ started from a point in $(0,1] \times\{0\}$ lies to the left of $\wh x$, then we can take $\wh P'$ to be this lift.
Otherwise, we can take $\wh P'$ to be the lift of $P'$ started from a point in $(-1,0]\times \{0\}$, which (by the periodicity of $\pi$) is given by translating the lift started from a point in $(0,1] \times\{0\}$ one unit to the left.

By the definition of the winding number, the first coordinates of $\wh P(1)$ and $\wh P'(1)$ are $\wh y = \wh x + \op{wind}(P) $ and $\wh y' = \wh x' + \op{wind}(  P')$, respectively.
Since $\op{wind}(P) + 1 < \op{wind}(P') $ and $\wh x' \in [\wh x - 1,  \wh x)$, it follows that $\wh y'  > \wh y$.  
Since also $\wh x' < \wh x$, it follows that $P$ and $P'$ must cross. 
\end{proof}

\begin{lem} \label{lem-geo-winding}
Almost surely, for each $0 < s < s' < \infty$, the winding numbers in $\mcl B_{s'}^\bullet\setminus\mcl B_s^\bullet$ of any two $D_h$-geodesics from $\bdy \mcl B_s^\bullet$ to $\bdy\mcl B_{s'}^\bullet$ (i.e., paths of $D_h$-length $s'-s$ between these boundaries) differ by at most 1. 
\end{lem}
\begin{proof} 
For $q\in\BB Q^2\setminus\mcl B_{s'}^\bullet$, let $P_q$ be the (a.s.\ unique) $D_h$-geodesic from 0 to $q$. 
The combination of Lemma~\ref{lem-non-cross} (applied to a $D_h$-geodesic from 0 to $\bdy\mcl B_{s'}^\bullet$ which coincides with the given geodesic on $[s,s']$) and Lemma~\ref{lem-winding-intersect} shows that a.s.\ the winding number of any $D_h$-geodesic from $\bdy \mcl B_s^\bullet$ to $\bdy\mcl B_{s'}^\bullet$ differs from winding number of each of the $P_q|_{[s,s']}$'s by at most 1. 
Hence there is an interval $[a,b]\subset \BB R$ of length at most one such that $\op{wind}(P_q|_{[s,s']}) \in [a,b]$ for every $q \in \BB Q^2\setminus\mcl B_{s'}^\bullet$. 
Since the winding number depends continuously on the path, the approximation statement from Lemma~\ref{lem-leftmost-geodesic} implies that a.s.\ for each $y\in\bdy\mcl B_{s'}^\bullet$, the winding numbers of the leftmost and rightmost geodesics $P_y^-|_{[s,s']}$ and $P_y^+|_{[s,s']}$ lie in $[a,b]$. 
Since every $D_h$-geodesic from $y$ to $\bdy\mcl B_s^\bullet$ lies between $P_y^-$ and $P_y^+$, the winding number of any such geodesic lies between $\op{wind}(P_y^-|_{[s,s']})$ and $\op{wind}(P_y^+|_{[s,s']})$, so must also lie in $[a,b]$. 
\end{proof}

The following lemma will be important in the iterative argument used to prove confluence of geodesics in Section~\ref{sec-finite-geo}.

\begin{lem} \label{lem-geo-arc}
Almost surely, the following is true for each $0 < s < s' < \infty$. 
Let $\mcl I$ be a finite collection of disjoint arcs of $\bdy\mcl B_s^\bullet$. 
For each $I\in\mcl I$, let $I'$ be the set of $y\in\bdy\mcl B_{s'}^\bullet$ such that the leftmost $D_h$-geodesic from 0 to $y$ passes through $I$.
Then each $I'$ is either empty or is a connected arc of $\bdy\mcl B_{s'}^\bullet$ and the arcs $I'$ for different choices of $I\in\mcl I$ are disjoint.
\end{lem}
\begin{proof}
It is obvious that the sets $I'$ for different choices of $I$ are disjoint since each $y\in\mcl B_{s'}^\bullet$ gives rise to a unique leftmost geodesic. 
We need to show that each $I'$ is connected.

Let $\pi : \BB R\times [0,1] \rta \ol{\mcl B_{s'}^\bullet \setminus \mcl B_s^\bullet)}$ be a universal covering map normalized so that $\pi(\BB R\times\{0\}) = \bdy \mcl B_s^\bullet$, $\pi(\BB R\times\{1\}) = \bdy \mcl B_{s'}^\bullet$, and $\pi(x , y) = \pi(x + n , y)$ for each $n\in\BB Z$ and $(x,y) \in \BB R\times [0,1]$. 
For each $D_h$-geodesic $P$ from 0 to a point of $\ol{\BB C\setminus\mcl B_{s'}^\bullet}$, let $\wh P$ be the lift of $P|_{[s,s']}$ w.r.t.\ $\pi$ which starts from a point in $[0,1)\times \{0\}$. 
Lemma~\ref{lem-geo-winding} implies that there is an interval $[a,b] \times \{1\} $ of length at most 1 such that $\wh P(s') \in [a,b]\times \{1\}$ for every such $D_h$-geodesic $P$. 

For $q\in \BB Q^2\setminus \mcl B_{s'}^\bullet$, let $P_q$ be the unique $D_h$-geodesic from 0 to $q$.  
For each such $q$, the path $\wh P_q$ divides $\BB R\times [0,1]$ into two connected components. 
Lemma~\ref{lem-non-cross} implies that each lift of a $D_h$-geodesic from $\bdy \mcl B_s^\bullet$ to $\bdy \mcl B_{s'}^\bullet$ must be in the closure of one of these connected components, i.e., it cannot cross~$\wh P_q$. 
 
For $I\in\mcl I$, let $\wh I$ be the subset of $[0,1)\times\{0\}$ such that $\pi(\wh I) = I$ and let $\wh I' \subset [a,b]$ be the subset of $[a,b] \times \{1\}$ such that $\pi(\wh I') = I'$. By possibly re-choosing $\pi$ so that $\pi(0,0)$ is a endpoint of one of the intervals in $\mcl I$, we can arrange that each $\wh I$ is an interval. 
To show that each $I'$ is connected, it is enough to show that $\wh I'$ is connected.
Suppose $(y_1, 1) , (y_2,1) \in \wh I'$ with $y_1 < y_2$ and that $y_1 < y < y_2$. 
We claim that $(y,1) \in \wh I'$. 
To see this, let $P_{y_1}^-$ and $P_{y_2}^-$ be the leftmost $D_h$-geodesics from 0 to $\pi(y_1)$ and $\pi(y_2)$, respectively, so that $P_{y_1}^-(s) , P_{y_2}^-(s) \in I$. 
If we choose $q_1,q_2 \in \BB Q\setminus \mcl B_{s'}^\bullet$ in such a way that $q_1$ (resp.\ $q_2$) is close to the interior of the counterclockwise arc of $\bdy\mcl B_{s'}^\bullet$ from $\pi(y_1)$ to $\pi(y )$ (resp.\ $\pi(y)$ to $\pi(y_2)$) then $y_1 \leq \wh P_{q_1}(s')  < \wh P_{q_2}(s') \leq y_2$.  

The lifts $\wh P_{y_1}$ and $\wh P_{y_2}$ cannot cross $\wh P_{q_1}$ and $\wh P_{q_2}$, so since $\wh P_{y_1}(s) , \wh P_{y_2}(s) \in \wh I$, it follows that $\wh P_{q_1}(s)$ and $\wh P_{q_2}(s)$ belong to $\wh I$. 
Let $P_y$ be the leftmost $D_h$-geodesic from $\mcl B_s^\bullet$ to $\pi(y)$. Then $\wh P_y$ cannot cross $\wh P_{q_1}$ or $\wh P_{q_2}$, so it must be the case that $\wh P_y(y) \in \wh I$. 
Therefore $y \in \wh I'$ so $\wh I'$ is connected.
\end{proof}

\subsection{FKG inequalities}
\label{sec-fkg}

In this subsection we will prove a variant of the FKG inequality which applies to a weak LQG metric. 
For the statement, we recall the definition of $e^{\xi f}\cdot D$ for a continuous function $f$ and a metric $D$ from~\eqref{eqn-metric-f}. We also recall from Remark~\ref{remark-other-domains} that $D$ gives rise to a metric associated with a zero-boundary GFF on any domain $U\subset\BB C$.

\begin{prop}[FKG for the LQG metric] \label{prop-fkg-metric}
Let $\gamma \in (0,2)$, let $U\subset\BB C$ be an open domain, let $ h$ be a zero-boundary GFF on $U$, and let $D$ be a weak $\gamma$-LQG metric.  
Let $\Phi$ and $\Psi$ be bounded, real-valued measurable functions on the space of continuous metrics on $U$ which are non-decreasing in the sense that for any two such metrics $D_1,D_2$ with $D_1(z,w) \leq D_2(z,w) $ for all $z,w\in U$, one has $\Phi(D_1) \leq \Phi(D_2)$ and $\Psi(D_1)\leq \Psi(D_2)$. 
Suppose further that $\Phi$ and $\Psi$ are a.s.\ continuous at $D_h$ in the sense that for every (possibly random) sequence of continuous functions $\{f^n\}_{n\in\BB N}$ which converges to zero uniformly on $U$, one has $\Phi(e^{\xi f^n} \cdot D_h ) \rta \Phi(D_h)$ and $\Psi(e^{\xi f^n} \cdot D_h) \rta \Psi(D_h)$.  
Then $\op{Cov}(\Phi(D_h), \Psi(D_h)) \geq 0$.
\end{prop}

We will typically apply Proposition~\ref{prop-fkg-metric} to functionals of the form 
\eqb
\Phi(D) = \BB 1\left\{ \sup_{u \in A , v \in B} D (u,v) \geq c\right\} \quad \text{or} \quad
\Phi(D) = \BB 1\left\{ \inf_{u \in A , v \in B} D (u,v) \geq c\right\}
\eqe
where $c> 0$ and $A,B \subset U$. Note that functionals of this form include distances between points and sets as well as diameters of sets (taking $A = B$). 
Such functionals $\Phi$ are obviously non-decreasing. Moreover, such functionals are a.s.\ continuous at $D_h$ since the probability that the supremum or infimum in question is exactly equal to $c$ is zero. This can be seen using Axiom~\ref{item-metric-f} and the fact that adding a smooth compactly supported function to $h$ affects its law in an absolutely continuous way.

We expect that Proposition~\ref{prop-fkg-metric} is true without the continuity hypothesis, but our proof does not give this.

The basic idea of the proof of Proposition~\ref{prop-fkg-metric} is to first prove a version of the FKG inequality for continuous, positively correlated Gaussian functions using the FKG inequality for finite-dimensional Gaussian vectors~\cite{pitt-positively-correlated} and an approximation argument (Lemma~\ref{lem-fkg-cont}). We then transfer this to the GFF using the white noise decomposition (Lemma~\ref{lem-fkg-gff}) and finally deduce Proposition~\ref{prop-fkg-metric} using Axiom~\ref{item-metric-f}. 
These intermediate FKG inequalities appear to be standard results, but we could not find sufficiently general statements in the literature so we will deduce them directly from the FKG inequality for finite-dimensional positive correlated Gaussian vectors~\cite{pitt-positively-correlated}. 

\begin{lem}[FKG for continuous Gaussian functions] \label{lem-fkg-cont}
Let $(X,D)$ be a locally compact metric space, let $\mcl C(X,\BB R)$ be the space of continuous, real-valued functions on $X$ equipped with the local uniform topology.
Let $\Phi$ and $\Psi$ be bounded measurable functions from $\mcl C(X,\BB R)$ to $\BB R$  
which are non-decreasing in the sense that $\Phi(f) \geq \Phi(g)$ and $\Psi(f) \geq \Psi(g)$ whenever $f(x) \geq g(x)$ for every $x\in X$. 
Let $f$ be a Gaussian random continuous function on $X$ and suppose that $\op{Cov}(f(x) , f(y)) \geq 0$ for every $x,y\in X$.  
If $\Phi$ and $\Psi$ are each a.s.\ continuous at $f$, then $\op{Cov}(\Phi(f) ,\Psi(f)) \geq 0$. 
\end{lem}
\begin{proof}
We will deduce the lemma from the FKG inequality for finite positively correlated Gaussian vectors~\cite{pitt-positively-correlated}. 
To do this, we will approximate $f$ by a sequence of functions which depend on only finitely many Gaussian random variables (this is how the continuity assumption on $\Phi$ and $\Psi$ is used).

Let $\{K_n\}_{n\in\BB N}$ be an increasing sequence of compact subsets of $X$ whose union is all of $X$. 
For $n\in\BB N$, choose $\delta_n > 0$ such that 
\eqb \label{eqn-fkg-modulus}
\BB P\left[ \sup_{x,y\in K_n : D(x,y) \leq \delta_n} |f(x) - f(y)| \leq 2^{-n} \right] \geq 1 -  2^{-n}  .
\eqe
Then, choose a finite collection of points $x_1^n,\dots,x_{N_n}^n$ such that the union of the metric balls $\mcl B_{\delta_n}(x_j^n)$ covers $K_n$.
Let $\phi_1^n,\dots,\phi_{N_n}^n$ be a partition of unity subordinate to this open cover of $K_n$, so that each $\phi_j^n$ is a continuous function on $X$ taking values in $[0,1]$ and supported in $\mcl B_{\delta_n}(x_j^n)$ and $\sum_{j=1}^{N_n} \phi_j^n \equiv 1$ on $K_n$.

Let $f^n := \sum_{j=1}^{N_n} f(x_j^n) \phi_j^n$. 
Since $\phi_j^n(x) = 0$ whenever $|x-x_j^n|  > \delta_n$, the Borel-Cantelli lemma and~\eqref{eqn-fkg-modulus} tell us that a.s.\ for large enough $n$, 
\eqb
|f(x) - f^n(x)| \leq \sum_{j=1}^{N_n} |f(x) - f(x_j^n)| \phi_j^n(x) \leq 2^{-n} ,\quad\forall x \in K_n .
\eqe 
Therefore $f^n \rta f$ uniformly on compact subsets of $X$. 

The function $f^n$ depends only on the values $f(x_1^n) , \dots,f(x_{N_n}^n)$ and increasing one of these values can only increase $f^n$. 
Hence $\Phi(f^n)$ and $\Psi(f^n)$ are non-decreasing functions of the finite-dimensional positively correlated Gaussian vector $(f(x_1^n),\dots,f(x_{N_n}^n))$. 
By the FKG inequality for finite-dimensional positively correlated Gaussian vectors~\cite{pitt-positively-correlated}, 
\eqbn
\BB E\left[ \Phi(f^n) \Psi(f^n) \right] \geq \BB E\left[ \Phi(f^n) \right] \BB E\left[ \Psi(f^n) \right] .
\eqen
Since $f^n \rta f$ locally uniformly and $\Phi$ and $\Psi$ are each bounded and a.s.\ continuous at $f$, sending $n \rta\infty$ now concludes the proof. 
\end{proof}

\begin{lem}[FKG for the GFF] \label{lem-fkg-gff}
Let $U\subset\BB C$ be an open domain and let $ h$ be a zero-boundary GFF on $U$. 
Let $\Phi$ and $\Psi$ be bounded, real-valued measurable functions on the space of distributions on $U$ which are non-decreasing in the sense that for any continuous, non-negative function $f$ on $U$, one has $\Phi(h + f) \geq \Phi(h)$ and $\Psi(h + f) \geq \Psi(h)$. 
Suppose also that $\Phi$ and $\Psi$ are a.s.\ continuous at $h$ in the sense that for any sequence of (possibly random) functions $\{f^n\}_{n\in\BB N}$ on $U$ which converge uniformly to zero on compact subsets of $U$, a.s.\ $\Phi(h + f^n) \rta \Phi(h )$ and $\Psi(h + f^n) \rta \Psi(h )$.  
Then $\op{Cov}(\Phi(h), \Psi(h)) \geq 0$.
\end{lem}
\begin{proof}
We will use the white noise decomposition of $h$ to reduce to the statement of Lemma~\ref{lem-fkg-cont}.  
Let $W$ be a space-time white noise on $U\times [0,\infty)$ so that for any square integrable functions $\phi , \psi : U\times [0,\infty) \rta \BB R$, the random variables $\int_0^\infty \int_U \phi(w,t)  \, W( d^2w,dt)$ and $\int_0^\infty \int_U \psi(w,t)  \, W( d^2w,dt)$ are centered Gaussian with covariance $\int_0^\infty \int_U \phi(w,t) \psi(w,t) \,  d^2w\,dt$. 
Let $p_t(z,w)$ for $z,w\in U$ be the transition density for Brownian motion run up to time $t$ and stopped upon exiting $U$, i.e., if $B$ is such a Brownian motion started from $z$, then $p_t(z,w) = \BB P[B_t \in  d^2w , B([0,t]) \in U]$. 
For $0  \leq s < t \leq \infty$ and $z\in U$, let  
\eqbn
h_{s,t}(z) := \sqrt \pi \int_{s^2}^{t^2} p_{r/2}( z,w) \, W( d^2w,dr)  .
\eqen
It is easily checked using the Kolmogorov continuity criterion that $h_{s,t}$ for $0 < s \leq t \leq \infty$ a.s.\ admits a continuous modification. 
For $s = 0$, $h_{0,t}$ does not admit a continuous modification and is instead interpreted as a random distribution.  
By~\cite[Lemma 5.4]{rhodes-vargas-review}, $h := h_{0,\infty}$ is the zero-boundary GFF on $U$.   
Since $p_t(z,w)$ is non-negative, each $h_{s,t}$ has non-negative covariances. 

For $\Phi$ as in the statement of the lemma and $t >0$, we define a functional on the space of continuous functions $f : U\rta\BB R$ by $\Phi_{h_{0,t}}(f) := \Phi(f + h_{0,t} )$. 
Then for a fixed realization of $h_{0,t}$, the functional $\Phi_{h_{0,t}}$ is non-decreasing in the sense of Lemma~\ref{lem-fkg-cont}.
Moreover, for a continuous function $f$ on $U$ one has $\Phi_{h_{0,t}}(h_{t,\infty}  +f ) =\Phi(h  +  f)$, so $\Phi_{h_{0,t}}$ is a.s.\ continuous at $h_{t,\infty}$ in the sense of Lemma~\ref{lem-fkg-cont}.   
The continuous Gaussian function $h_{t,\infty}$ is positively correlated,  Gaussian, and independent from $\{h_{0,s}\}_{s\leq t}$.
 Therefore, if we define $\Psi_{h_{0,t}}$ analogously to $\Phi_{h_{0,t}}$ and apply Lemma~\ref{lem-fkg-cont} to $\Phi_{h_{0,t}}(h_{t,\infty}) = \Phi(h)$ and $\Psi_{h_{0,t}}(h_{t,\infty}) = \Psi(h)$, we get
\eqb
\label{eqn:pos_cor_cond}
\BB E\left[ \Phi(h) \Psi(h) \,|\, \{h_{0,s}\}_{s\leq t} \right] \geq \BB E\left[ \Phi(h)   \,|\,\{h_{0,s}\}_{s\leq t} \right] \BB E\left[  \Psi(h) \,|\, \{h_{0,s}\}_{s\leq t} \right].
\eqe
By the Kolmogorov zero-one law applied to the independent random variables $\{h_{s,t} : 2^{-k} \leq s \leq t \leq 2^{-k+1}\}$ for $k\in\BB N$, $\bigcap_{t > 0} \sigma(\{h_{0,s}\}_{s\leq t})$ is the trivial $\sigma$-algebra. 
We can therefore take a limit as $t \to 0$ in~\eqref{eqn:pos_cor_cond} and apply the backward martingale convergence theorem to conclude the proof.
\end{proof}

\begin{proof}[Proof of Proposition~\ref{prop-fkg-metric}] 
By Axiom~\ref{item-metric-f}, it is a.s.\ the case that for each continuous function $f$ on $U$ one has $D_{h+f} = e^{\xi f}\cdot D_h$. In particular, if $f$ is non-negative then $D_{h+f} \geq D_h$ 
Therefore $g\mapsto \Phi(D_g)$ and $g\mapsto \Phi(D_g)$ are non-decreasing and a.s.\ continuous at $h$ in the sense of Lemma~\ref{lem-fkg-gff}.  
Hence the proposition statement follows from Lemma~\ref{lem-fkg-gff}. 
\end{proof}

\subsection{Iterating events for the GFF in an annulus}
\label{sec-annulus-iterate}

Throughout this subsection, $h$ denotes a whole-plane GFF normalized so that $h_1(0) = 0$. 
A key tool in our proofs is the following local independence property for events which depend on the GFF in disjoint concentric annuli, which is essentially proven in~\cite[Section 4]{mq-geodesics}; see~\cite[Lemma 3.1]{local-metrics} for the statement we use here. (The statement in~\cite{local-metrics} also allows the events to depend on a collection of metrics, instead of just the GFF, but we will not need this here since our metrics are determined by the GFF). 
We recall the definition of the annulus $\BB A_{r_1,r_2}(z) = B_{r_2}(z)\setminus \ol{B_{r_1}(z)}$ from~\eqref{eqn-annulus-def}.

\begin{lem}[\!\!\cite{local-metrics}] \label{lem-annulus-iterate}
Fix $0 < s_1<s_2 < 1$. Let $\{r_k\}_{k\in\BB N}$ be a decreasing sequence of positive real numbers such that $r_{k+1} / r_k \leq s_1$ for each $k\in\BB N$ and let $\{E_{r_k} \}_{k\in\BB N}$ be events such that $E_{r_k} \in \sigma\left( (h-h_{r_k}(0)) |_{\BB A_{s_1 r_k , s_2 r_k}(0)  } \right)$ for each $k\in\BB N$. 
For $K\in\BB N$, let $N(K)$ be the number of $k\in [1,K]_{\BB Z}$ for which $E_{r_k}$ occurs. 
\begin{enumerate} 
\item For each $a > 0$ and each $b\in (0,1)$, there exists $p = p(a,b,s_1,s_2) \in (0,1)$ and $c = c(a,b,s_1,s_2) > 0$ such that if \label{item-annulus-iterate-high}
\eqb \label{eqn-annulus-iterate-prob}
\BB P\left[ E_{r_k}  \right] \geq p , \quad \forall k\in\BB N  ,
\eqe 
then 
\eqb \label{eqn-annulus-iterate}
\BB P\left[ N(K)  < b K\right] \leq c e^{-a K} ,\quad\forall K \in \BB N. 
\eqe
\item For each $p\in (0,1)$, there exists $a > 0$, $b\in (0,1)$, and $c > 0$, depending only on $p,s_1,s_2$ such that if~\eqref{eqn-annulus-iterate-prob} holds, then~\eqref{eqn-annulus-iterate} holds. \label{item-annulus-iterate-pos}
\end{enumerate}
\end{lem}
  
We have the following variant of Lemma~\ref{lem-annulus-iterate} where we explore annuli outward instead of inward.

\begin{lem} \label{lem-annulus-iterate-inverse}
Fix $1   < S_1  < S_2 < \infty$. Let $\{r_k\}_{k\in\BB N}$ be an increasing sequence of positive real numbers such that $r_{k+1} / r_k \geq S_2$ for each $k\in\BB N$ and let $\{E_{r_k} \}_{k\in\BB N}$ be events such that $E_{r_k} \in \sigma\left( (h-h_{r_k}(0)) |_{\BB A_{S_1 r_k , S_2 r_k}(0)  } \right)$ for each $k\in\BB N$. 
For $K\in\BB N$, let $N(K)$ be the number of $k\in [1,K]_{\BB Z}$ for which $E_{r_k}$ occurs. 
\begin{enumerate} 
\item For each $a > 0$ and each $b\in (0,1)$, there exists $p = p(a,b,S_1,S_2) \in (0,1)$ and $c = c(a,b,S_1,S_2) > 0$ such that if 
\eqb \label{eqn-annulus-iterate-high-inverse}
\BB P\left[ E_{r_k}  \right] \geq p , \quad \forall k\in\BB N  ,
\eqe 
then 
\eqb
\BB P\left[ N(K)  < b K\right] \leq c e^{-a K} ,\quad\forall K \in \BB N. 
\eqe
\item For each $p\in (0,1)$, there exists $a > 0$, $b\in (0,1)$, and $c > 0$, depending only on $p,S_1,S_2$ such that if~\eqref{eqn-annulus-iterate-prob} holds, then~\eqref{eqn-annulus-iterate} holds. \label{item-annulus-iterate-pos-inverse}
\end{enumerate}
\end{lem}

Lemma~\ref{lem-annulus-iterate-inverse} can be proven using the exact same argument used to prove Lemma~\ref{lem-annulus-iterate}.
Alternatively, Lemma~\ref{lem-annulus-iterate-inverse} is an immediate consequence of Lemma~\ref{lem-annulus-iterate} and the following lemma.

\begin{lem} \label{lem-gff-invert}
Let $h$ be a whole-plane GFF normalized so that $h_1(0) = 0$.
Then the composition of $h$ with the inversion map $z\mapsto 1/z$ has the same law as $h$.
\end{lem}
\begin{proof}
Let $(\cdot,\cdot)_\nabla$ denote the Dirichlet inner product. 
By the conformal invariance of $(\cdot,\cdot)_\nabla$, if $f$ is any smooth function supported on a compact subset of $\BB C\setminus \{0\}$ one has $(h(1/\cdot) ,f)_\nabla = (h , f(1/\cdot) )_\nabla$. Therefore, if $g$ is another such function, then
\alb
\op{Cov}\left( (h(1/\cdot) ,f)_\nabla , (h(1/\cdot) , g)_\nabla \right) 
=  (f(1/\cdot) , g(1/\cdot))_\nabla 
&= (f,g)_\nabla \notag\\
&= \op{Cov}\left( (h ,f)_\nabla , (h , g)_\nabla \right) .
\ale
Since a Gaussian process is determined by its covariance structure, it follows that the restrictions of $h$ and $h(1/\cdot)$ to any compact subset of $\BB C\setminus \{0\}$ agree in law modulo additive constant. We know that the additive constants for $h$ and $h(1/\cdot)$ are the same since the condition $h_1(0) =0$ is preserved under inversion. 
The restrictions of $h$ to compact subsets of $\BB C\setminus \{0\}$ a.s.\ determine $h$ since the $\sigma$-algebras $\bigcap_{r > 0} \sigma(h|_{B_r(0)})$ and $\bigcap_{R >0} \sigma(h|_{\BB C\setminus B_R(0)})$ are trivial. 
\end{proof}

\subsection{Harmonic exposure of boundary intervals}
\label{sec-disconnect-set}

In this section we will prove some deterministic geometric lemmas for general classes of domains in $\BB C$. These lemmas will eventually be applied to the complement of a filled $D_h$-metric ball in Sections~\ref{sec-confluence} and~\ref{sec-one-geodesic}. Since we have a rather poor understanding of the geometry of the boundary of such a filled metric ball, it is important that the bounds are uniform over all possible domains satisfying certain mild constraints. 
The reader may wish to skip this subsection on a first read and refer back to the estimates as they are used. 
We first prove a universal bound to the effect that most boundary arcs of a simply connected planar domain $U$ containing zero can be disconnected from $\infty$ by a small set. 
 
\begin{lem} \label{lem-disconnect-set}
There is a universal constant $A>0$ such that the following is true. 
Let $U \subset\BB C$ be a bounded simply connected domain containing 0 and view $\bdy U$ as a collection of prime ends. 
Also let $n\in\BB N$ and let $\mcl I$ be a collection of $\# \mcl I = n$ arcs of $\bdy U$ which intersect only at their endpoints. 
Then for $C> 0$, the number of arcs $I\in\mcl I$ which can be disconnected from 0 in $U$ by a path in $U$ of Euclidean diameter at most $C n^{-1/2}$ is at least $\left(1 - A C^{-2} \op{area}(U) \right) n$. 
\end{lem}
\begin{proof}
See Figure~\ref{fig-disconnect-set} for an illustration of the proof. 
Let $\phi : \BB D\rta U$ be a conformal map with $\phi(0) = 0$. For $I\in\mcl I$, let $r_I$ be the length of the arc $\phi^{-1}(I)$. Since the arcs $\phi^{-1}(I)$ intersect only at their endpoints, there can be at most 8 such arcs with $r_I \geq \pi/4$. By removing these arcs from our collection, we can assume without loss of generality that $r_I \leq \pi/4$ for each $I\in\mcl I$.

\begin{figure}[t!]
 \begin{center}
\includegraphics[scale=.85]{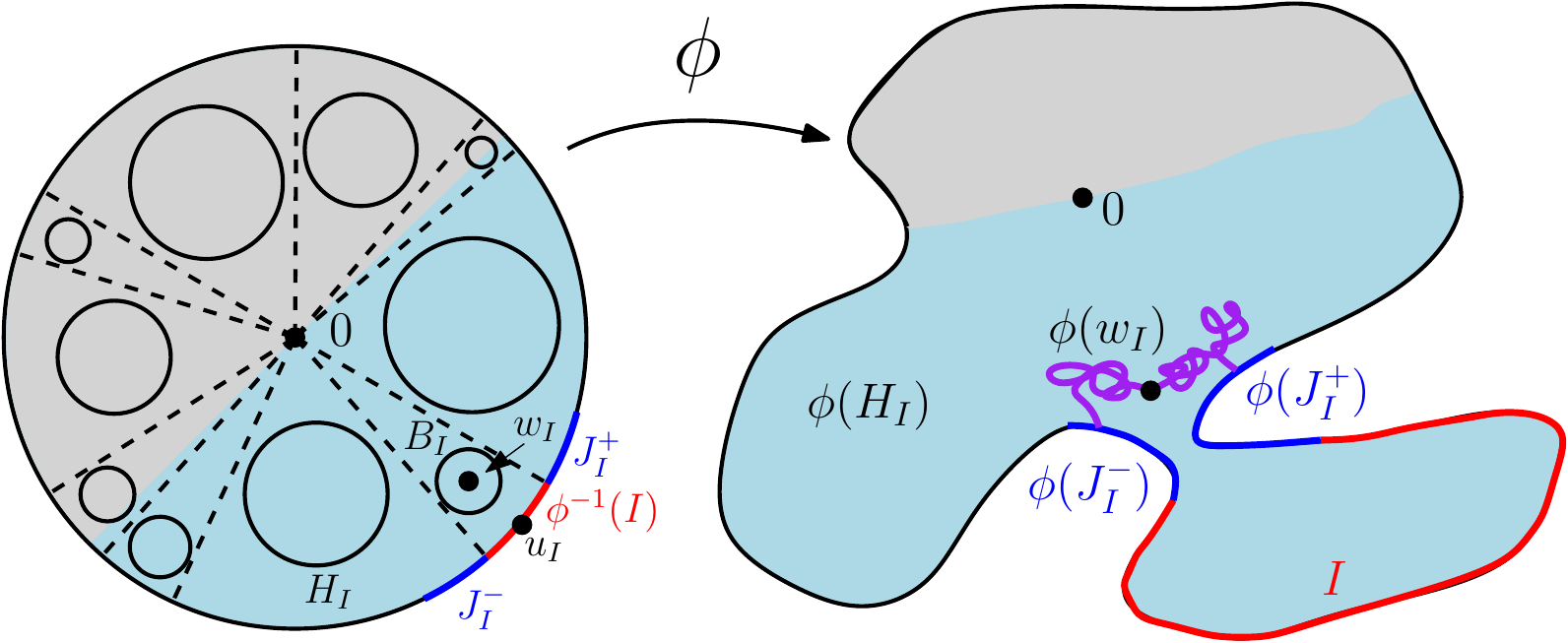}
\vspace{-0.01\textheight}
\caption{Illustration of the proof of Lemma~\ref{lem-disconnect-set}. For each arc $I\in\mcl I$, we consider a point $w_I \in\BB D$ whose distance to $\phi^{-1}(I)$ is proportional to the length of $\phi^{-1}(I)$. A Brownian motion started from $w_I$ has uniformly positive probability to first exit the half-disk $H_I$ in each of the two arcs $J_I^\pm$ on either side of $\phi^{-1}(I)$ before getting too far away from $w_I$. Using the Beurling estimate, this shows that $I$ can be disconnected from 0 in $\phi(H_I)$ by the union of two paths whose Euclidean diameters are comparable to the Euclidean distance from $\phi(w_I)$ to $\bdy U$. By distortion estimates, this distance is comparable to $|\phi'(w_I)| (1-|w_I|)$. There cannot be too many intervals for which this last quantity is large since $\int_{\BB D} |\phi'(z)|^2 \,d^2 z = \op{area}(U)$. 
}\label{fig-disconnect-set}
\end{center}
\vspace{-1em}
\end{figure} 

For $I\in\mcl I$, let $J_I^-$ (resp.\ $J_I^+$) be the arc of $\bdy\BB D$ of length $r_I$ which comes immediately before (resp.\ after) $\phi^{-1}(I)$ in the counterclockwise direction. 
Also let $u_I \in \bdy\BB D$ be the center of $\phi^{-1}(I)$ and let $H_I$ be the half-disk of $\BB D$ with the property that $u_I$ is the center of the arc $\bdy\BB D\cap\bdy H_I$. 
Since $r_I \leq  \pi/4$, we have that $J_I^-$, $\phi^{-1}(I)$, and $J_I^+$ intersect only at their endpoints and are contained in $\bdy H_I$. 

Define $w_I := (1- r_I) u_I$. 
The probability that a Brownian motion started from $w_I$ exits $H_I$ at a point in $ J_I^- $ is at least some universal constant $a >0$. By the conformal invariance of Brownian motion, a Brownian motion started from $\phi(w_I)$ has probability at least $a$ to exit $\phi(H_I)$ at a point in $\phi(J_I^-)$. On the other hand, the Beurling estimate shows that for $R > 1$, the probability that such a Brownian motion travels distance $R \op{dist}(\phi(w_I) , \bdy U ) $ without hitting $\bdy U$ is at most a universal constant times $ R^{-1/2}$. 
If we take $R$ to be a sufficiently large universal constant, then this last quantity is smaller than $a$, so with positive probability, a Brownian motion started from $\phi(w_I)$ exits $\phi(H_I)$ at a point of $\phi(J_I^-)$, and does so before it leaves the ball of radius $R \op{dist}(\phi(w_I) , \bdy U ) $ centered at $\phi(w_I)$. 
Consequently, there is a path in $\phi(H_I)$ from $\phi(w_I)$ to $\phi(J_I^-)$ with Euclidean diameter at most $2 R \op{dist}(\phi(w_I) , \bdy U )$ which is disjoint from $\phi(H_I)$. 
Symmetrically, the same statement holds with $J_I^+$ in place of $J_I^-$.

Concatenating the paths for $J_I^-$ and $J_I^+$ gives a path in $  \phi(H_I)$ with Euclidean diameter at most $4 R \op{dist}(\phi(w_I) , \bdy U )$ from a point of $\phi(J_I^-)$ to a point of $\phi(J_I^+)$. 
This path necessarily divides $U$ into at least two connected components. 
Since the path is disjoint from $\phi(H_I)$ and $0\in\bdy \phi(H_I)$, it must disconnect 0 from $I$ in $U$. Consequently, 
\eqb  \label{eqn-disconnect-path}
\text{$I$ is disconnected from 0 in $U$ by a path of diameter $\leq 4 R \op{dist}(\phi(w_I) , \bdy U)$}.
\eqe

We will now show that $\op{dist}(\phi(w_I) , \bdy U)$ is small for most $I\in\mcl I$. 
Set $B_I := B_{r_I/100}(w_I)$.  
We observe that $B_I$ is contained in the slice of $\BB D$ bounded by $\phi^{-1}(I)$ and the line segments from the two endpoints of $\phi^{-1}(I)$ to $0$, so since the arcs $\phi(I)$ for $I\in\mcl I$ are disjoint, the balls $B_I$ are disjoint. 
For $I\in\mcl I$, one has $\op{dist}(w_I , \bdy \BB D) = r_I$. 
By the Koebe distortion theorem and the Koebe quarter theorem, there is a universal constant $A_0  > 0$ such that 
\eqb \label{eqn-use-koebe}
\inf_{z \in B_I} |\phi'(z)| \geq A_0^{-1}  |\phi'(w_I)| \quad \text{and} \quad  \op{dist}\left( \phi(w_I) , \bdy U \right) \leq A_0 |\phi'(w_I)| r_I  .
\eqe
Applying the first inequality, then the second inequality, in~\eqref{eqn-use-koebe} shows that
\eqb
\op{area}(U)  =  
\int_{\BB D} |\phi'(z)|^2 \,d^2 z
\geq \frac{1}{A_0^2} \sum_{I\in\mcl I} |\phi'(w_I)|^2  \frac{\pi r_I^2}{100^2} 
\geq \frac{\pi}{100^2 A_0^4} \sum_{I\in\mcl I}   \op{dist}\left( \phi(w_I) , \bdy U \right)^2    .
\eqe
Since $r_I \leq 1$ for all $I$, we can apply the Chebyshev inequality to get that for $C>0$, 
\eqb \label{eqn-big-dist-count}
\#\left\{ I \in \mcl I :  \op{dist}\left( \phi(w_I) , \bdy U \right) \geq \frac{C}{4 R} n^{-1/2} \right\} \leq A C^{-2} \op{area}(U) n 
\eqe
for  $A =   100^2 \pi^{-1} A_0^4 \times 16 R^2$. 
Combining~\eqref{eqn-disconnect-path} and~\eqref{eqn-big-dist-count} concludes the proof. 
\end{proof}

In the next several lemmas, we will use the following notation. 
For a domain $U\subset\BB C$ containing~0, we define
\eqb 
\op{inrad}(U) := \sup\{r  >0: B_r(0) \subset U\} 
\quad\text{and} \quad \op{outrad}(U) := \inf\left\{R  > 0 : U \subset B_R(0) \right\} .
\eqe
Similarly, for a compact connected set $K\subset\BB C$ containing 0, we define
\eqb \label{eqn-outrad-inrad}
\op{inrad}(K) := \sup\left\{ r > 0 : B_r(0) \subset K  \right\} 
\quad \text{and} \quad
\op{outrad}(K) := \inf\left\{ R > 0 :  K  \subset B_R(0) \right\} .
\eqe

By inverting, we obtain an analog of Lemma~\ref{lem-disconnect-set-infty} when we want to disconnect boundary arcs from $\infty$ instead of from 0.

\begin{lem} \label{lem-disconnect-set-infty}
There is a universal constant $A>0$ such that the following is true. Let $K\subset\BB C$ be a compact connected set which contains a neighborhood of 0 and whose complement is connected.  
Also let $n\in\BB N$ and let $\mcl I$ be a collection of $\# \mcl I = n$ arcs of $\bdy K$ which intersect only at their endpoints. 
Then for $C> 0$, the number of arcs $I\in\mcl I$ which can be disconnected from $\infty$ in $\BB C\setminus K$ by a path in $\BB C\setminus K$ of Euclidean diameter at most $C n^{-1/2}$ is at least 
\eqb
\left(1 - \frac{A}{C^2} \frac{\op{outrad}(K)^4}{\op{inrad}(K)^2} \right) n
\eqe  
\end{lem}
\begin{proof}
Let $  U$ be the image of $\BB C\setminus K$ under the map $\phi : z\mapsto 1/z$.
Then $U$ is contained in the Euclidean ball of radius $1/\op{inrad}(K)$ centered at zero, so $\op{area}( U) \leq \pi / \op{inrad}(K)^2$. Since $\inf_{z\in \BB C\setminus K} |\phi'(z)| \geq 1/\op{outrad}(K)^2$, the diameter of each subset of $ \BB C \setminus K$ is at most $\op{outrad}(K)^2$ times the diameter of the corresponding subset of $U$. 
The statement of the lemma therefore follows from Lemma~\ref{lem-disconnect-set} applied to $ U$, with $C/\op{outrad}(K)^2$ in place of $C$.  
\end{proof}

We next prove two lemmas to the effect that a collection of arcs which cover $\bdy U$ (or $\bdy K$) must include at least one arc which is sufficiently ``exposed", in a certain quantitative sense. These lemmas will be used in the proof of Lemma~\ref{lem-one-point}. As in the previous lemmas, it is crucial that the bounds be uniform over all choices of $U$ (or $K$). We will need a particular way of measuring distances in $U$ which is slightly different from the ordinary Euclidean distance (see Figure~\ref{fig-exposed-arc} for an illustration of why this is needed). 

Let $U\subset\BB C$ be a simply connected domain and view $\bdy U$ as a set of prime ends. 
If $X\subset U$, we define the \emph{prime end closure} $\op{Cl}'(X)$ to be the set of points in $z\in U\cup\bdy  U$ with the following property: if $\phi : U\rta \BB D$ is a conformal map, then $\phi(z)$ lies in $\ol{\phi(X)}$. 
 For $z,w\in U\cup \bdy U $ we define
\eqb \label{eqn-d^U-def} 
d^U(z,w) = \inf\left\{\op{diam}(X) : \text{$X$ is a connected subset of $U$ with $z,w\in \op{Cl}'(X) $}\right\} ,
\eqe
where here $\op{diam}$ denotes the Euclidean diameter. Then $d^U$ is a metric on $U \cup \bdy U$ which is bounded below by the Euclidean metric on $\BB C$ restricted to $U \cup \bdy U$ and bounded above by the internal Euclidean metric on $U \cup \bdy U$. Note that $d^U$ is not a length metric. 
We similarly define $\op{Cl}'(\cdot)$ and $d^U$ in the case when $U$ is an unbounded open subset of $\BB C$ such that $\BB C\setminus U$ is compact, in which case we use a conformal map to $\BB C\setminus \BB D$ instead of a conformal map to $\BB D$.

\begin{figure}[t!]
 \begin{center}
\includegraphics[scale=.75]{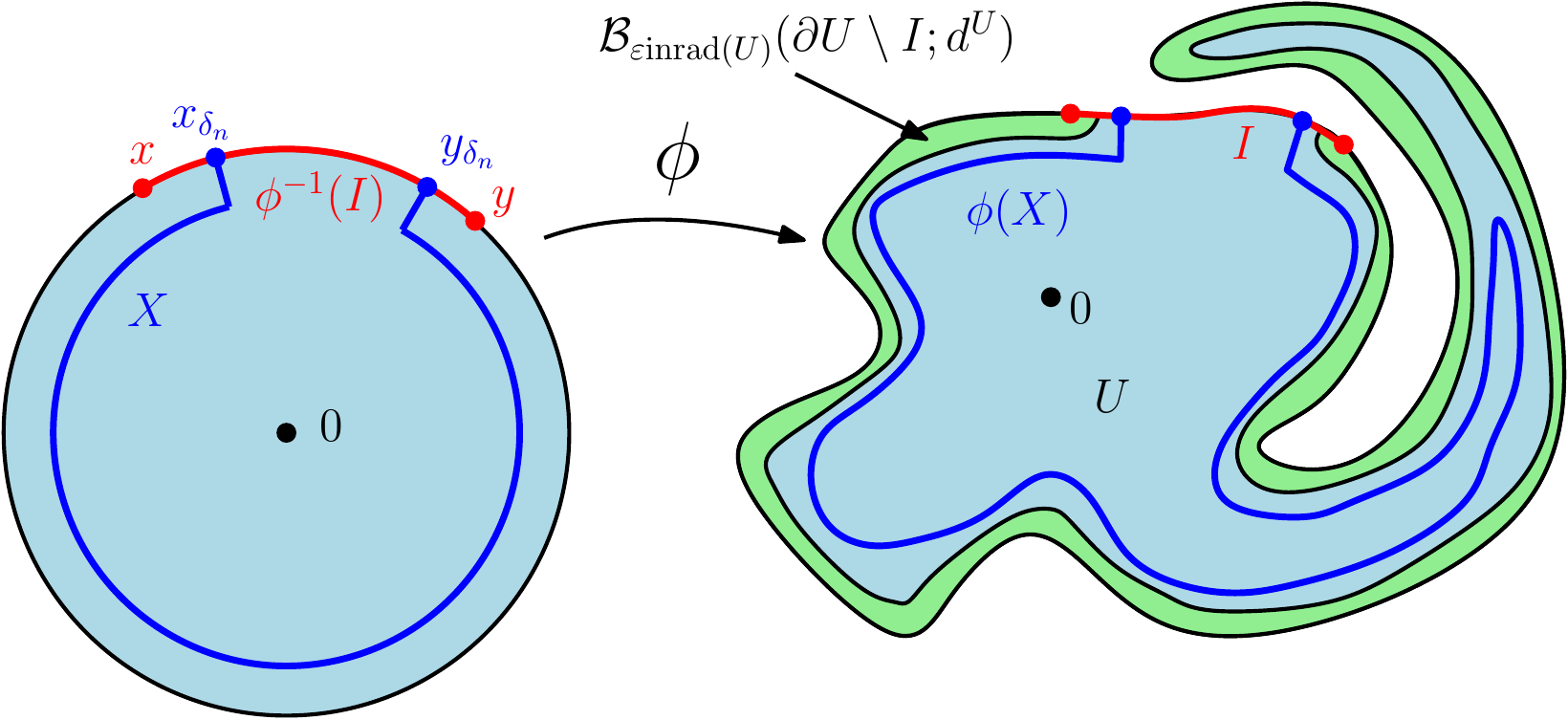}
\vspace{-0.01\textheight}
\caption{Illustration of the proof of Lemma~\ref{lem-exposed-arc}. We start by choosing $I\in\mcl I$ such that the length of $\phi^{-1}(I)$ is bounded below by $2\pi/n$ (i.e., the harmonic measure of $I$ from 0 in $U$ is at least $1/n$). We then build a blue ``shield" $X$ in $\BB D$ which separates 0 and part of $\phi^{-1}(I)$ from the rest of $\bdy\BB D$. We show using distortion estimates that the image of this shield under $\phi$ lies at uniformly positive $d^U$-distance from $\bdy U\setminus I$. Hence for a small enough $\ep  >0$, $\phi(X)$ separates $\mcl B_{\ep \op{inrad}(U)}(\bdy U\setminus I ; d^U)$ from 0 and from at least one point of $I$. 
The right side of the figure shows why we use $d^U$ instead of Euclidean distance: the illustrated domain $U$ has a long ``arm" which is very close to $I$ with respect to the ambient Euclidean distance, but it not close to $I$ with respect to $d^U$. 
}\label{fig-exposed-arc}
\end{center}
\vspace{-1em}
\end{figure}

\begin{lem} \label{lem-exposed-arc}  
Let $U \subset \BB C$ be a bounded simply connected domain containing 0 and view $\bdy U$ as a collection of prime ends. 
Let $n\in\BB N$ and let $\mcl I$ be a collection of $\# \mcl I = n$ arcs of $\bdy U$ whose union is all of $\bdy U$ (not necessarily disjoint).  
There exists $\ep > 0$ depending only on $n$ and the ratio $\op{outrad}(U) / \op{inrad}(U)$ such that for any choice of $U$ and $\mcl I$, there exists $I\in\mcl I$ such that the following is true. The arc $I$ is \emph{not} disconnected from 0 in $U$ by the $d^U$-neighborhood $\mcl B_{\ep \op{inrad}(U)}(\bdy U\setminus I ; d^U)$. 
\end{lem}
\begin{proof}
See Figure~\ref{fig-exposed-arc} for an illustration and outline of the proof. 
By scaling, we can assume without loss of generality that $\op{inrad}(U) = 1$. Let $R := \op{outrad}(U)$. 

Let $\phi : \BB D\rta U$ be a conformal map which fixes zero. 
Since $\#\mcl I = n$, there must exist $I\in\mcl I$ such that the $\bdy\BB D$-length of $\phi^{-1}(I)$ is at least $2\pi/n$.  
Let $x$ and $y$ be the endpoints of $I$, in counterclockwise order. 
Let $\delta_n := 1/(4n)$ and let $x_{\delta_n}$ (resp.\ $y_{\delta_n}$) be the point of $I$ which lies at $\bdy\BB D$-distance $\delta_n$ from $x$ (resp.\ $y$) in the counterclockwise (resp.\ clockwise) direction. 
Let $X$ be the union of the linear segments $[(1-\delta_n) x_{\delta_n} , x_{\delta_n}]$ and $[(1-\delta_n )y_{\delta_n}, y_{\delta_n}]$ and the counterclockwise arc $A_{\delta_n}$ of $\bdy B_{1-\delta_n}(0)$ from $y_{\delta_n}$ to $x_{\delta_n}$ ($X$ is shown in blue on the left side of Figure~\ref{fig-exposed-arc}). 

We claim that there is a constant $\ep  = \ep(n , R)   > 0$ such that 
\eqb \label{eqn-exposed-dist}
d^U\left( \phi(X  ) , \bdy U  \setminus I   \right) \geq \ep  .
\eqe
Assuming~\eqref{eqn-exposed-dist}, we conclude the proof as follows. 
The relation~\eqref{eqn-exposed-dist} implies that $\phi(X )$ is disjoint from $\mcl B_\ep(\bdy U\setminus I ; d^U)$. 
Since $\mcl B_\ep(\bdy U\setminus I ; d^U)$ is connected, by conformally mapping back to $\BB D$ it follows that there is a connected component of $U\setminus \phi(X )$ which contains 0 and which has some point of $I$ in its prime end closure. 
Therefore, $I$ is not disconnected from 0 by $\mcl B_{\ep }(\bdy U\setminus I ; d^U)$ (recall that we have assumed that $\op{inrad}(U) =1$). 
 
Let us now prove~\eqref{eqn-exposed-dist}. Since we have assumed that $\op{inrad}(U) =1$, we have $\BB D\subset U$. By the Koebe quarter theorem, $\phi^{-1}(U) = \BB D$ contains a Euclidean ball of radius at least $|(\phi^{-1})'(0)|/4$ centered at 0. Therefore we must have $|(\phi^{-1})'(0)|/4 \leq 1$ so $|\phi'(0)| \geq 1/4$.
For $u \in A_{\delta_n}$ we have $|u| = 1-\delta_n$ so by the Koebe distortion theorem $|\phi'(u)| \geq  |\phi'(0)| \delta_n /(2 - \delta_n)^3 \geq \delta_n / (4(2-\delta_n)^3)$.
By the Koebe quarter theorem applied to the restriction of $\phi$ to $B_{\delta_n}(u)$, we then obtain
\eqb \label{eqn-exposed-koebe}
\op{dist}(\phi(A_{\delta_n}) , \bdy U) \geq  \frac{\delta_n^2 }{ 16 (2 - \delta_n)^3}   .
\eqe 

We now need to deal with the small linear segments of $X$. 
Let $J_{\delta_n} \subset \phi^{-1}(I)$ be the counterclockwise arc of $\bdy\BB D$ from $x$ to $x_{\delta_n}$. 
The probability that a Brownian motion started from 0 exits $\BB D$ in $  J_{\delta_n}$ before hitting the segment $[(1-\delta_n) x_{\delta_n} , x_{\delta_n}]$ is at least a universal constant $c>0$ times $\delta_n$.  
By the conformal invariance of Brownian motion, the probability that a Brownian motion started from 0 exits $U$ in $\phi(J_{\delta_n})$ before hitting $\phi([(1-{\delta_n}) x_{\delta_n} , x_{\delta_n}])$ is at least $c  \delta_n $. 
Similarly, the probability that a Brownian motion started from 0 exits $U$ in $I\setminus \phi(J_{\delta_n})$ before hitting $\phi([(1-\delta_n) x_{\delta_n} , x_{\delta_n}])$ is at least~$c \delta_n $. 

There is a constant $\zeta = \zeta(n,R) > 0$ such that for any set $Y\subset B_R(0) \setminus B_{1/2}(0)$ with Euclidean diameter at most $\zeta$, the probability that a Brownian motion started from 0 hits $Y$ before exiting $B_R(0)$ is at most $(c/2)\delta_n$. 
The preceding paragraph implies that neither $\phi(J_{\delta_n})$ nor $I\setminus \phi(J_{\delta_n})$ can be disconnected from 0 in $U\setminus \phi([(1-\delta_n ) x_{\delta_n} , x_{\delta_n}])$ by a set of Euclidean diameter smaller than $\zeta$. 
This implies that $d^U( \phi([(1-\delta_n ) x_{\delta_n} , x_{\delta_n}]) , \bdy U\setminus I  ) \geq \zeta$. 
A symmetric argument shows that the same is true for $ \phi([(1-\delta_n ) y_{\delta_n} , y_{\delta_n}]) , \bdy U\setminus I  )$. 
Combining this with~\eqref{eqn-exposed-koebe} gives~\eqref{eqn-exposed-dist} with $\ep = \max\{\zeta , \frac{\delta_n^2  }{ 16 (2 - \delta_n)^3}  \}$.  
\end{proof}

We now deduce an analog of Lemma~\ref{lem-exposed-arc} for domains which are conformally equivalent to $\BB C\setminus \ol{\BB D}$ instead of to $\BB D$.

\begin{lem} \label{lem-exposed-arc-invert}  
Let $K \subset \BB C$ be a compact connected set whose complement is connected and view $\bdy K$ as a collection of prime ends.  Let $n\in\BB N$ and let $\mcl I$ be a collection of $\# \mcl I = n$ (not necessarily disjoint) arcs of $\bdy K$ whose union is all of $\bdy K$.  There exists $\ep > 0$ depending only on $n$ and the ratio $\op{outrad}(K) / \op{inrad}(K)$ such that for any choice of $U$ and $\mcl I$, there exists $I\in\mcl I$ such that the following is true. The arc $I$ is \emph{not} disconnected from $\infty$ in $\BB C\setminus K$ by the $d^{\BB C\setminus K}$-neighborhood $\mcl B_{\ep \op{outrad}(U)}(K\setminus I ; d^{\BB C\setminus K} )$.  
\end{lem}
\begin{proof}
Let $\phi(z) := 1/z$ and let 
$U := \phi( \BB C\setminus K)$. 
Then $U$ is simply connected and we have $\op{outrad}(U) = 1/\op{inrad}(K)$ and $\op{inrad}(U) = 1 / \op{outrad}(K)$.  
Lemma~\ref{lem-exposed-arc} shows that there exists $\ep > 0$ as in the statement of the lemma such that for any choice of $U$ and $\mcl I$, there exists $I\in\mcl I$ such that $I$ is \emph{not} disconnected from $\infty$ in $\BB C\setminus K$ by the $d^U$-neighborhood $\phi( \mcl B_{\ep / \op{outrad}(K)}(\bdy U\setminus I ; d^U) )$. 
We have $|\phi'(z)| \leq 1/\op{inrad}(K)^2$ on $\BB C\setminus B_{\op{inrad}(K)}(0) \supset \BB C\setminus K$, so for any set $A\subset\BB C\setminus K$, the Euclidean diameter of $\phi(A)$ is at most $\op{diam}(A) / \op{inrad}(K)^2$. 
It follows that 
\eqb
 \phi\left( \mcl B_{\ep' \op{inrad}(K)}\left( K \setminus I ; d^{\BB C\setminus K} \right) \right)  \subset \mcl B_{\ep / \op{outrad}(K)}(\bdy U\setminus I ; d^U) \quad\text{for} \quad \ep' = \frac{\op{inrad}(K)}{\op{outrad}(K)}\ep .
\eqe
Therefore, the statement of the lemma is true with $\ep'$ in place of $\ep$.
\end{proof}

\section{Finitely many leftmost geodesics across an LQG annulus}
\label{sec-confluence}

In this section we will prove a more quantitative version of Theorem~\ref{thm-finite-geo0} (Theorem~\ref{thm-finite-geo}) which, as we explain just below, immediately implies Theorem~\ref{thm-finite-geo}. 
The main difference between Theorem~\ref{thm-finite-geo0} and Theorem~\ref{thm-finite-geo} is that the latter gives a bound for the number of leftmost $D_h$-geodesics across an LQG annulus which is \emph{uniform in the Euclidean size of the LQG annulus}. 
This is important since we are only assuming tightness across scales (Axiom~\ref{item-metric-coord}) instead of exact scale invariance. 
To quantify what Euclidean ``scale" we will be working in,\footnote{We write $\BB r$ for the fixed Euclidean scale we are working with. The symbol $r$ is used for other, possibly random, radii which arise in the proof.} 
we will consider the following stopping times for $\{(\mcl B_s^\bullet, h|_{\mcl B_s^\bullet})\}_{s\geq 0}$: 
\eqb \label{eqn-tau_r-def}
\tau_{\BB r} := \inf\left\{ s  > 0 : \mcl B_s^\bullet \not\subset B_{\BB r}(0) \right\} ,\quad\forall \BB r > 0. 
\eqe
We also recall the definition of leftmost geodesics from Lemma~\ref{lem-leftmost-geodesic} and the scaling constants $\frk c_{\BB r}$ for $\BB r >0$ from Axiom~\ref{item-metric-coord}.

\begin{thm} \label{thm-finite-geo}
For each $t > 0$ and $p \in (0,1)$, there exists $N = N(t , p) \in\BB N$ such that the following is true for each $\BB r > 0$ and each stopping time $\tau$ for $\{(\mcl B_s^\bullet, h|_{\mcl B_s^\bullet})\}_{s\geq 0}$ such that $\tau \in [\tau_{\BB r} , \tau_{2 \BB r}]$ a.s. With probability at least $p$, that there are at most $N$ points of $\bdy\mcl B_{\tau  }^\bullet$ which are hit by leftmost $D_h$-geodesics from 0 to $\bdy\mcl B_{\tau + t \frk c_{\BB r} e^{\xi h_{\BB r}(0)} }^\bullet$. 
\end{thm}

By Axiom~\ref{item-metric-coord}, typically $\tau_{\BB r}$ and $\tau_{2\BB r}$ are each comparable to $\frk c_{\BB r} e^{\xi h_{\BB r}(0)}$, so typically $\tau  + t\frk c_{\BB r} e^{\xi h_{\BB r}(0)}$ is of the same order of magnitude as $\tau$. 
Theorem~\ref{thm-finite-geo} is the first result of this paper which gives any sort of confluence of geodesics.  Prior to this point, we have not ruled out the possibility that any two distinct $D_h$-geodesics started from 0 intersect only at their common starting point. 

\begin{proof}[Proof of Theorem~\ref{thm-finite-geo0} assuming Theorem~\ref{thm-finite-geo}]
By Axiom~\ref{item-metric-translate} and the translation invariance of the law of $h$, viewed modulo additive constant, we can assume without loss of generality that $z = 0$.
For $0 < t < s < \infty$, let $X_{t,s}$ be the set of points $x\in\bdy\mcl B_t^\bullet$ which are hit by leftmost $D_h$-geodesics from 0 to $\bdy\mcl B_s^\bullet$.  
Theorem~\ref{thm-finite-geo} says that for each stopping time $\tau$ as in that theorem and each $t > 0$, a.s.\ $X_{\tau , \tau+t \frk c_{\BB r} e^{\xi h_{\BB r}(0)}}$ is finite.

We first observe that in the setting of Theorem~\ref{thm-finite-geo}, there is a.s.\ a \emph{unique} $D_h$-geodesic from 0 to each point $x\in X_{\tau , \tau+t \frk c_{\BB r} e^{\xi h_{\BB r}(0)}}$. 
Indeed, consider such a point $x$ and let $y\in \bdy\mcl B_{\tau + t \frk c_{\BB r} e^{\xi h_{\BB r}(0)} }^\bullet$ such that the leftmost $D_h$-geodesic $P_y$ from 0 to $y$ passes through $x$.
By Lemma~\ref{lem-leftmost-geodesic}, we can find $q\in \BB Q^2 \setminus \mcl B_{\tau + t \frk c_{\BB r} e^{\xi h_{\BB r}(0)} }^\bullet$ such that the distance between $P_y$ and the (a.s.\ unique) $D_h$-geodesic $P_q$ from 0 to $q$ w.r.t.\ the Euclidean uniform metric is smaller than the minimum of $|x - x'|$ over all $ x' \in  X_{\tau , \tau+t \frk c_{\BB r} e^{\xi h_{\BB r}(0)}} \setminus \{x\}$. 
Since $P_q|_{[0,\tau+t\frk c_{\BB r} e^{\xi h_{\BB r}(0)}]}$ is a leftmost $D_h$-geodesic from 0 to $\bdy\mcl B_{\tau + t \frk c_{\BB r} e^{\xi h_{\BB r}(0)} }^\bullet$, it follows that $P_q$ must pass through $x$. By the uniqueness of $D_h$-geodesics to rational points (Lemma~\ref{lem-geo-unique}), it follows that there is only on $D_h$-geodesic from 0 to $x$. 

Therefore, it is a.s.\ the case that for each rational $\BB r> 0$ and each rational $b >0$, the set $X_{\tau_{\BB r} , \tau_{\BB r}   + b \frk c_{\BB r} e^{\xi h_{\BB r}(0)}}$ is finite and the $D_h$-geodesic from 0 to each point of this set is unique. 
We now argue that the same holds for $X_{t,s}$ for all $0<t<s<\infty$ simultaneously.
Since $D_h$ induces the Euclidean topology on $\BB C$, it is easily seen that $\BB r \mapsto \tau_{\BB r}$ is continuous and surjective: indeed, if this function had an upward jump then there would be some non-trivial interval of times $s$ during which the Euclidean radius of $\mcl B_s^\bullet$ does not increase. 
Therefore, for any given times $0  <t < s < \infty$, we can find a rational $\BB r>  0$ and a small rational $b > 0$ such that $t < \tau_{\BB r} <  \tau_{\BB r} + b \frk c_{\BB r} e^{\xi h_{\BB r}(0)} <s $.
For every leftmost $D_h$-geodesic $P$ from 0 to $\bdy\mcl B_s^\bullet$, the restriction $P|_{[0,\tau_{\BB r} +b\frk c_{\BB r} e^{\xi h_{\BB r}(0)}]}$ is a leftmost $D_h$-geodesic from 0 to $\bdy\mcl B_{\tau +b\frk c_{\BB r} e^{\xi h_{\BB r}(0)} }^\bullet$, so $P(\tau_{\BB r} +b\frk c_{\BB r} e^{\xi h_{\BB r}(0)})$ is one of the points $x \in X_{\tau_{\BB r} , \tau_{\BB r} +b \frk c_{\BB r} e^{\xi h_{\BB r}(0)}}$.
Consequently, $P$ coincides with the unique $D_h$-geodesic $P_x$ from 0 to $x$ until time $\tau_{\BB r}$. Therefore, $P|_{[0,t]}$ is equal to one of the finitely many paths $P_x|_{[0,t]}$ for $x\in  X_{\tau_{\BB r} , \tau_{\BB r} +b \frk c_{\BB r} e^{\xi h_{\BB r}(0)}}$. 
\end{proof}

\subsection{Outline of the proof}
\label{sec-conf-outline}

The rest of this section is devoted to the proof of Theorem~\ref{thm-finite-geo}. In fact, we will prove a more quantitative version of the theorem (Theorem~\ref{thm-finite-geo-quant} below) which gives bounds for $N$ in terms of $t$ and $p$ provided we truncate on a certain global regularity event. 

We now outline the proof of Theorem~\ref{thm-finite-geo}.
Fix a stopping time $\tau$ for $\{(\mcl B_s^\bullet , h|_{\mcl B_s^\bullet})\}_{s\geq 0}$ such that $\op{diam}(\mcl B_\tau^\bullet)$ is typically of order $\BB r > 0$, as in Theorem~\ref{thm-finite-geo}. 
We start by considering an arbitrary finite collection $\mcl I_0$ of disjoint arcs of $\bdy\mcl B_{\tau }^\bullet$, chosen in a manner depending only on $(\mcl B_{\tau }^\bullet, h|_{\mcl B_{\tau }^\bullet})$. 
We will show that for each $p  \in (0,1)$, there exists $N\in\BB N$, which does \emph{not} depend on $\mcl I_0$ or on $\BB r$, such that with probability at least $p$, there are at most $N$ arcs in $\mcl I_0$ which are hit by a leftmost $D_h$-geodesic from 0 to $\bdy\mcl B_{\tau + t \frk c_{\BB r} e^{\xi h_{\BB r}(0)}}^\bullet$ (we consider leftmost geodesics because of their monotonicity properties, in particular Lemma~\ref{lem-geo-arc}). 
By taking $\mcl I_0$ to be a large collection of arbitrarily small arcs, this will show that the set of points of $\bdy\mcl B_{\tau}^\bullet$ which are hit by leftmost $D_h$-geodesics from 0 to $\bdy\mcl B_{\tau + t \frk c_{\BB r} e^{\xi h_{\BB r}(0)}}^\bullet$ can be covered by at most $N$ arbitrarily small arcs with probability at least $p$, hence this set is a.s.\ finite.

To bound the number of arcs in $\mcl I_0$ which are hit by a leftmost $D_h$-geodesic from 0 to $\bdy\mcl B_{\tau + t \frk c_{\BB r} e^{\xi h_{\BB r}(0)}}^\bullet$, we proceed as follows.
Let $n := \#\mcl I_0$. We first show that there exist exponents $\alpha  , \theta > 0$ (depending only on the choice of metric) such that if we condition on $(\mcl B_s^\bullet , h|_{\mcl B_s^\bullet})$ (and truncate on an appropriate global regularity event), then the following is true.
\begin{enumerate}[A.]
\item At least $3n/4$ of the arcs $I\in\mcl I_0$ can be disconnected from $\infty$ in $\BB C\setminus\mcl B_{\tau}^\bullet$ by a set of Euclidean diameter at most $n^{-1/4}$. \label{item-outline-disc}
\item For each arc $I\in\mcl I_0$ which can be disconnected from $\infty$ in $\BB C\setminus\mcl B_{\tau}^\bullet$ by a set of Euclidean diameter at most $n^{-1/4}$, it holds with conditional probability at least $1- O_n( n^{-\alpha})$ that no leftmost $D_h$-geodesic from 0 to a point outside of $\mcl B_{\tau +n^{-\theta}}^\bullet$ can pass through $I$. \label{item-outline-geo}
\end{enumerate}
Property~\ref{item-outline-disc} follows from the deterministic estimate for general planar domains given in Lemma~\ref{lem-disconnect-set-infty}, and does not require any information at all about the geometry of $\bdy\mcl B_{\tau}^\bullet$.

Property~\ref{item-outline-geo} is established in Sections~\ref{sec-clsce-event} and~\ref{sec-geo-kill} by using Weyl scaling (Axiom~\ref{item-metric-f}) and the Markov property of the GFF to show that with positive conditional probability given $(\mcl B_{\tau}^\bullet , h|_{\mcl B_{\tau}^\bullet})$, one can build a ``shield" around $I$ which no $D_h$-geodesic from a point far from $I$ to 0 can pass through. This shield will consist of two concentric annuli of the form $\BB A_{2r,3r}(z)$ and $\BB A_{3r,4r}(z)$ for appropriate choices of $r >0$ and $z\in\BB C$ with the following properties.
The annulus $\BB A_{2r,3r}(z)$ disconnects $I$ from $\infty$ in $\BB C\setminus \mcl B_\tau^\bullet$ and the $D_h$-distance from any point of $\BB A_{3r,4r}(z)$ to $\bdy\mcl B_\tau^\bullet$ is smaller than the $D_h$-distance between the inner and outer boundaries of $\BB A_{2r,3r}(z)$. 
Such annuli are illustrated in Figure~\ref{fig-geo-kill}. 
It is easy to see from the definition of a geodesic (and is explained carefully in the proof of Lemma~\ref{lem-geo-kill-pt}) that if such annuli exist, then no $D_h$-geodesic from 0 to a point outside of $\mcl B_{\tau}^\bullet \cup B_{4r}(z)$ can hit $I$. 

We will set things up so that one has a logarithmic number of essentially independent chances to build a shield of the above form, so the probability that no such shield exists decays like a negative power of $n$. The events needed to build the shield depend on the metric in a reasonably continuous way, so Axiom~\ref{item-metric-coord} allows us to get a lower bound for the probability that the shield exists which is uniform in $\BB r$ (see Lemma~\ref{lem-clsce-event-pos}). 
There is some subtlety here since the event that the shield exists depends on both the zero-boundary part of the field outside of $\mcl B_{\tau}^\bullet$ and the pair $(\mcl B_{\tau}^\bullet , h|_{\mcl B_{\tau}^\bullet})$.
To deal with this, we will use Lemma~\ref{lem-annulus-iterate-inverse} to find many annuli surrounding $I$ where the harmonic part of $h|_{\BB C\setminus \mcl B_\tau^\bullet}$ is under control and only try to build a shield in these ``good" annuli (see Lemma~\ref{lem-clsce-event-pos}). 
We will also need to use the FKG inequality (Proposition~\ref{prop-fkg-metric}) to deal with the behavior of $h$ very close to $\bdy\mcl B_\tau^\bullet$. 
See Figure~\ref{fig-geo-kill} for an illustration of this part of the argument. 
 
We will then apply properties~\ref{item-outline-disc} and~\ref{item-outline-geo} iteratively to ``kill off" all of the geodesics from all but a constant order number of the arcs in $\mcl I_0$. This is done in Section~\ref{sec-finite-geo} and is illustrated in Figure~\ref{fig-metric-ball-iterate}. To this end, we define radii $s_k$ and collections of boundary arcs $\mcl I_k$ for $k\in\BB N$ inductively as follows. We set $s_0 = \tau$ and we define $s_k$ to be a stopping time (to be specified precisely below) which is between $s_{k-1}$ and $ s_{k-1} + (\#\mcl I_{k-1})^{-\theta} \frk c_{\BB r} e^{\xi h_{\BB r}(0)}$. For $I \in \mcl I_{k-1}$, we define $I'$ to be the set of points in $\bdy\mcl B_{s_k}^\bullet $ which are hit by a leftmost $D_h$-geodesic from 0 which passes through $I$ and we set $\mcl I_k := \{I' : I\in \mcl I_{k-1} , I'\not=\emptyset\}$. Basic properties of geodesics (see Lemma~\ref{lem-geo-arc}) show that $\mcl I_k$ is a collection of disjoint arcs of $\bdy \mcl B_{s_k}^\bullet$.
By applying properties~\ref{item-outline-disc} and~\ref{item-outline-geo} above with $s_k$ in place of $\tau$, we get that typically $\#\mcl I_{k+1} \leq \frac12 \#\mcl I_k$ (Lemma~\ref{lem-half-count}). 
Using this, we infer that if $N$ is chosen sufficiently large, in a manner which does not depend on $\mcl I_0$, and $K$ is the smallest $k\in\BB N$ for which $\#\mcl I_k \leq N$, then which high probability $s_K \leq \tau + t \frk c_{\BB r} e^{\xi h_{\BB r}(0)}$ (Lemma~\ref{lem-count-radius}). We do not prove any bound for $K$, just for $s_K$.
This shows that with high probability, there are at most $N$ arcs in $\mcl I_0$ which are hit by leftmost $D_h$-geodesics from 0 to $\bdy\mcl B_{\tau + t \frk c_{\BB r} e^{\xi h_{\BB r}(0)}}^\bullet$. Since the arcs in $\mcl I_0$ can be made arbitrarily small and $N$ does not depend on $\mcl I_0$, this implies Theorem~\ref{thm-finite-geo}.

\subsection{Good annuli}
\label{sec-clsce-event}

We now define an event for a Euclidean annulus which will eventually be used to build ``shields" surrounding boundary arcs of a filled $D_h$-metric ball through which $D_h$-geodesics to 0 cannot pass. See Figure~\ref{fig-geo-event} for an illustration.

\begin{figure}[t!]
 \begin{center}
\includegraphics[scale=1]{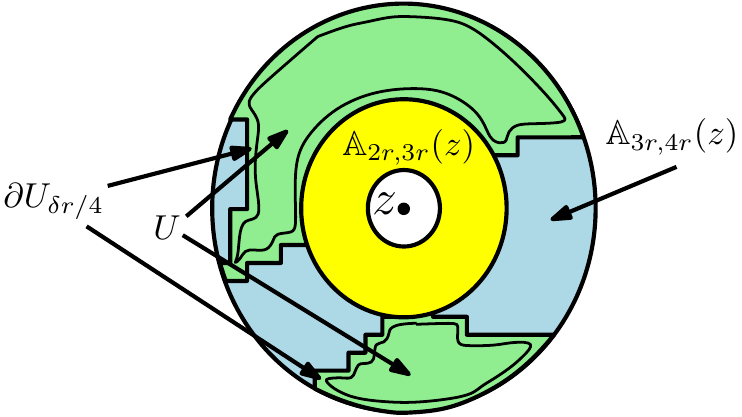}
\vspace{-0.01\textheight}
\caption{Illustration of the definitions in Section~\ref{sec-clsce-event}. The set $\mcl U_r(z) = \mcl U_r(z;\delta)$ consists of open subsets $U$ of $\BB A_{3r,4r}(z)$ such that  $\BB A_{3r,4r}(z) \setminus  U$ is a finite union of sets of the form $S \cap \BB A_{3r,4r}(z)$ for $\delta r\times\delta r$ squares $S\in   \mcl S_{\delta r}(\BB A_{3r,4r}(z))$ (i.e., with corners in $\delta \BB Z^2$). One such set is shown in light green. For each $U\in\mcl U_r(z)$, $E_r^U(z)$ is the event that (1) the $D_h$-distance across the yellow annulus $\BB A_{2r,3r}(z)$ is bounded below, (2) the $D_h$-diameter of each of the squares $S\in\mcl S_{ \delta r}(\BB A_{3r,4r}(z))+z$ is much smaller than the $D_h$-distance across $\BB A_{2r,3r}(z)$ and (3) the harmonic part of $h|_U$ is bounded above on the set $U_{\delta r/4} \subset U$ (outlined in black). 
}\label{fig-geo-event}
\end{center}
\vspace{-1em}
\end{figure}

For $\ep > 0$, $z\in \BB C$, and a set $V\subset\BB C$, we define the collection of Euclidean squares
\eqb \label{eqn-square-def}
\mcl S_\ep^z(V) := \left\{ [x, x+\ep] \times [y,y+\ep] : (x,y)\in\ep\BB Z^2 +z , \, ([x, x+\ep] \times [y,y+\ep])\cap V \not=\emptyset\right\}. 
\eqe
Note that $\mcl S_\ep^z(V)$ depends only on the value of $z$ modulo $\ep\BB Z^2$ and that $\mcl S_\ep^z(V) - z  = \mcl S_\ep^0(V-z)$. 
 
For $z\in\BB C$, $r> 0$, and $\delta\in (0,1)$, we define $\mcl U_r(z) = \mcl U_r(z;\delta)$ to be the (finite) set of open subsets $U$ of the annulus $\BB A_{3r,4r}(z)$ such that $\BB A_{3r,4r}(z) \setminus  U$ is a finite union of sets of the form $S \cap \BB A_{3r,4r}(z)$ for squares $S \in \mcl S_{\delta r}^z(\BB A_{3r,4r}(z)) $.
For $U\in \mcl U_r(z;\delta)$ and $\ep > 0$, we define
\eqb \label{eqn-U-nbd-def}
U_\ep := \left\{ u \in U : \op{dist}(z , \bdy U) > \ep \right\}
\eqe 
where $\op{dist}$ denotes Euclidean distance.
 
For $z\in\BB C$, $r  > 0$, parameters $  \delta  , c    \in (0,1)$ and $A > 0$, and $U\in \mcl U_r(z;\delta)$, we let $E_r^U(z) = E_r^U(z;  c , \delta,A) $ be the event that the following is true. 
\begin{enumerate}  
\item $D_h \left(\bdy B_{2r}(z) , \bdy B_{3r}(z) \right) \geq c \frk c_r e^{\xi h_r(z)}$. \label{item-clsce-across}
\item One has \label{item-clsce-ball}
\eqb 
\max_{S\in\mcl S_{ \delta r}^z(\BB A_{3r,4r}(z)) } \sup_{u,v\in S} D_h\left(u,v ; \BB A_{ 2r, 5r}(z) \right) \leq \frac{c}{100} \frk c_r e^{\xi h_r(z)}  . 
\eqe
\item Let $\frk h^U$ be the harmonic part of $h|_U$. Then, in the notation~\eqref{eqn-U-nbd-def}, \label{item-clsce-harmonic} 
\eqb 
\sup_{u   \in U_{\delta r / 4}} |\frk h^U(u) - h_r(z) |  \leq A  .
\eqe  
\end{enumerate}
We also define
\eqb \label{eqn-clsce-all-event}
E_r(z) = E_r(z;c,\delta,A) := \bigcap_{U\in \mcl U_r(z;\delta)} E_r^U(z) .
\eqe
The first two conditions in the definition of $E_r^U(z)$ do not depend on $U$, so the only difference between $E_r(z)$ and $E_r^U(z)$ is that for the former event, condition~\ref{item-clsce-harmonic} is required to hold for all choices of $U$ simultaneously.

We think of annuli $\BB A_{2r,5r}(z)$ for which $E_r(z)$ occurs as ``good". 
We will show in Lemma~\ref{lem-clsce-event-pos} just below that $\BB P[E_r(z)]$ can be made close to 1 by choosing the parameters $\delta, c , A$ appropriately, in a manner which is uniform over the choices of $r$ and $z$, 
The reason for separating $E_r(z)$ and $E_r^U(z)$ is that conditioning on $E_r^U(z)$ is easier than conditioning on $E_r(z)$ (see Lemma~\ref{lem-cond-diam-small} just below). 

The occurrence of $E_r^U(z)$ or $E_r(z)$ is unaffected by adding a constant to the field. By this and the locality of $D_h$ (Axiom~\ref{item-metric-local}), these events are determined by $h|_{\BB A_{2r,5r}(z)}$, viewed modulo additive constant. This will be important when we apply Lemma~\ref{lem-annulus-iterate-inverse} below.  

We will eventually apply condition~\ref{item-clsce-harmonic} with $U$ equal to $\BB A_{3r,4r}(z)$ minus the union of the set of squares in $\mcl S_\ep^z(\BB A_{3r,4r}(z))$ which intersect a filled $D_h$-metric ball centered at 0. Condition~\ref{item-clsce-harmonic} together with the Markov property of $h$ will allow us to show that with uniformly positive conditional probability given $h|_{\BB C\setminus U}$ and the event $E_r^U(z)$, the $D_h$-diameter of $U_{\delta r/ 4}$ is small (see Lemma~\ref{lem-cond-diam-small}). 
This combined with condition~\ref{item-clsce-ball} will show that with uniformly positive conditional probability given $h|_{\BB C\setminus U}$ and $E_r^U(z)$, \emph{every} point of $\BB A_{3r,4r}(z)$ lies at $D_h$-distance strictly smaller than $c r^{\xi Q} e^{\xi h_r(z)}$ from the filled $D_h$-metric ball. 
Due to condition~\ref{item-clsce-across}, this will prevent a $D_h$-geodesic from crossing $\BB A_{2r,3r}(z)$ before entering this filled metric ball. 
See Figure~\ref{fig-geo-kill} for an illustration of how the events $E_r^U(z)$ will eventually be used.

\begin{lem} \label{lem-clsce-event-pos}
For each $p \in (0,1)$, we can find parameters $ c , \delta  \in (0,1)$ and $A>0$ such that, in the notation~\eqref{eqn-clsce-all-event}, we have $\BB P[E_r(z)] \geq p$ for each $z\in\BB C$ and $r > 0$.
\end{lem}
\begin{proof}
By translation invariance and tightness across scales (Axioms~\ref{item-metric-translate} and~\ref{item-metric-coord}), the laws of the reciprocals of the scaled distances $\frk c_r^{-1} e^{-\xi h_r(z)} D_h(\bdy B_{2r}(z) , \bdy B_{3r}(z))$ for $z\in\BB C$ and $r>0$ are tight. Therefore, we can find $c = c(p) > 0$ such that for each $z\in\BB C$ and $r > 0$, condition~\ref{item-clsce-across} in the definition of $E_r^U(z)$ occurs with probability at least $1 - (1-p)/3$.
Similarly, Axioms~\ref{item-metric-translate} and~\ref{item-metric-coord} show that we can find $\delta = \delta(p , c)  \in (0,1)$ such that condition~\ref{item-clsce-ball} in the definition of $E_r(z)$ occurs with probability at least $1-(1-p)/3$. 
For a given choice of $\delta$, the collection of open sets $\mcl U_r(z;\delta)$ is finite, and is equal to $r \mcl U_1(0;\delta) + z$ (here we use the translation by $z$ in~\eqref{eqn-square-def}). 
Since $\frk h^U$ is continuous away from $\bdy U$, for any fixed choice of $U \in \mcl U_1(0;\delta)$, a.s.\ $\sup_{u  \in U_{\delta/4 }} |\frk h^U(u)  |  < \infty$. 
By combining this with the translation and scale invariance of the law of $h$, modulo additive constant, we find that there exists $A> 0$ (depending on $\delta$) such that with probability at least $1-(1-p)/3 $, condition~\ref{item-clsce-harmonic} in the definition of $E_r^U(z)$ holds simultaneously for every $U\in \mcl U_r(z;\delta)$.
\end{proof}

\begin{lem} \label{lem-cond-diam-small}
For any choice of parameters $c,\delta,A$, there is a constant $\frk p = \frk p( c, \delta , A) > 0$ such that the following is true. 
Let $r  > 0$, let $z\in\BB C$, and let $U \in \mcl U_r(z) =  \mcl U_r(z;\delta)$. 
Also let $\mcl V(U) \subset\mcl U_r(z)$ be the set of connected components of $U$. 
Almost surely,
\eqb  \label{eqn-cond-diam-small} 
\BB P\left[ \max_{V\in \mcl V(U)} \sup_{u,v \in V} D_h\left( u , v ; \BB A_{ 2r ,5r}(z) \right) \leq \frac{c}{2} \frk c_r e^{\xi h_r(z)}    \,\big|\,h|_{\BB C\setminus U} , E_r^U(z) \right] \geq \frk p .
\eqe  
\end{lem}
\begin{proof}  
By Axiom~\ref{item-metric-f}, the statement of the lemma does not depend on the choice of additive constant for $h$, so we can assume without loss of generality that $h$ is normalized so that $h_{r_0}(z_0) = 0$ for some $z_0 \in \BB C$ and $r_0 > 0$ such that $B_{r_0}(z_0) \cap U =\emptyset$ (we could take $z_0 = z$ and $r_0 = r$, but we find that the proof is clearer if we do not normalize so that $h_r(z) =0$). 
For $U\in\mcl U_r(z)$, let $\rng h^U$ be the zero-boundary part of $h|_U$. By the Markov property of the field (see, e.g., \cite[Lemma 2.2]{local-metrics}), we can write $h|_U = \rng h^U + \frk h^U$, where $\rng h^U$ is a zero-boundary GFF in $U$ which is independent from $h|_{\BB C\setminus U}$ (and hence also from the harmonic part $\frk h^U$). 
We define the metric $D_{\rng h^U}$ as in Remark~\ref{remark-other-domains}. 

Since conditions~\ref{item-clsce-across} and~\ref{item-clsce-harmonic} in the definition of $E_r^U(z)$ are determined by $h|_{\BB C\setminus U}$, the conditional law of $\rng h^U$ given $ h|_{\BB C\setminus U}$ and the event $E_r^U(z)$ is the same as the conditional law of $\rng h^U$ given $\frk h^U$ and the event $F_r(z)$ that condition~\ref{item-clsce-ball} in the definition of $E_r^U(z)$ occurs (note that this condition does not depend on $U$).

The main idea of the proof is as follows. By Axiom~\ref{item-metric-f}, we have $D_{h|_U} = e^{\xi \frk h^U} \cdot D_{\rng h^U}$. Condition~\ref{item-clsce-harmonic} in the definition of $E_r^U(z)$ gives an upper bound for $\frk h^U$, so we just need to prove that certain $ D_{\rng h^U}$-distances are very small with positive conditional probability when we condition on $h|_{\BB C\setminus U}$ and $E_r^U(z)$ (see~\eqref{eqn-cond-diam-event} for the precise event we need). Since $\rng h^U$ is independent from $h|_{\BB C\setminus U}$ and by the preceding paragraph, we only need to prove a bound for $D_{\rng h^U}$-distances when we condition on $F_r(z)$. 

This can be done as follows. By Axiom~\ref{item-metric-f}, if we let $f$ be a large bump function supported on a compact subset of $U$ which is close to all of $U$, then $D_{\rng h^U - f}$ distances are much smaller than $D_{\rng h^U}$ distances. 
On the other hand, the laws of $\rng h^U$ and $\rng h^U-f$ are mutually absolutely continuous. 
This shows that $D_{\rng h^U}$-distances are small with positive probability under the unconditional law of $\rng h^U$.
This last estimate can be made uniform over the choice of $U\in\mcl U_r(z)$ since $\mcl U_r(z) = r \mcl U_1(0) + z$ and $\mcl U_1(0)$ is a finite set (this is why we restrict to domains $U$ which are made up of small squares in a fixed grid). To add in the conditioning on $F_r(z)$, we use the FKG inequality (Proposition~\ref{prop-fkg-metric}). 
\medskip

\noindent\textit{Step 1: reducing to an event for the zero-boundary part.}
For $U\in\mcl U_r(z)$, define $V_{\delta r/2 }$ and $V_{\delta r/4}$ for $V \in \mcl V(U)$ as in~\eqref{eqn-U-nbd-def} but with $V$ in place of $U$. Note that $U_{\delta r/2} = \bigcup_{V\in\mcl V(U)} V_{\delta r/2}$. Let
\eqb \label{eqn-cond-diam-event}
G^U := \left\{ \max_{V\in\mcl V(U)} \sup_{u,v \in V_{\delta r/2}} D_{\rng h^U  }\left( u ,v ; V_{\delta r/4  } \right) \leq \frac{c}{100}    e^{-\xi A} \frk c_r   \right\}.
\eqe
By condition~\ref{item-clsce-harmonic} in the definition of $E_r^U(z)$ and Axiom~\ref{item-metric-f}, if $E_r^U(z) \cap G^U$ occurs then 
\eqb \label{eqn-cond-diam-add}
\max_{V\in\mcl V(U)} \sup_{u,v \in V_{\delta r/2 }} D_{h}\left( u ,v ; V_{\delta r/4  } \right) \leq \frac{c}{100} \frk c_r e^{\xi h_r(z)} . 
\eqe 
Since any two points of any $V \in \mcl V(U)$ can be joined by a path in $V$ which is contained in the union of $V_{\delta r/2}$ and at most two squares $S\in \mcl S_{\delta r}^z(\BB A_{2 r,4 r}(z))$, combining~\eqref{eqn-cond-diam-add} with condition~\ref{item-clsce-ball} in the definition of $E_r^U(z)$ shows that if $E_r^U(z)\cap G^U$ occurs, then the event in~\eqref{eqn-cond-diam-small} occurs. 

To prove the lemma, it therefore suffices to find $\frk p = \frk p(c,\delta,A) >0$ such that a.s.\ $\BB P[G^U \,|\, h|_{\BB C\setminus U} , E_r^U(z) ] \geq \frk p$ for each $U\in \mcl U_r(z)$.
By the above discussion about $F_r(z)$ and since $G^U$ is determined by $\rng h^U$, so is independent from $h|_{\BB C\setminus U}$, we only need to show that a.s.\ 
\eqb \label{eqn-cond-diam-show}
\BB P[G^U \,|\,  F_r(z)  ]  \geq \frk p  , \quad  \forall U\in\mcl U_r(z)  .
\eqe 
\medskip

\noindent\textit{Step 2: lower bound without conditioning on $F_r(z)$.}
We first argue that there is a constant $\frk p = \frk p(c , \delta,A) >0$ such that a.s.\ 
\eqb \label{eqn-cond-diam-good}
\BB P\left[ G^U  \right] \geq \frk p , \quad  \forall U\in\mcl U_r(z) .
\eqe 
By Axioms~\ref{item-metric-translate} and~\ref{item-metric-coord} and since $\mcl U_r(z) = r \mcl U_1(0) + z$ and $\#\mcl U_1(0)$ depends only on $\delta$, we can find a constant $C>0$, depending only on $\delta$, such that for each $z\in\BB C$, $r>0$, and $U\in\mcl U_r(z)$, 
\eqb \label{eqn-cond-diam-pos}
\BB P\left[ \max_{V\in\mcl V(U)} \sup_{u,v \in V_{\delta r/2 }} D_{\rng h^U  }\left( u ,v ; V_{\delta r/4 } \right) \leq  C \frk c_r    \right] \geq \frac12 .
\eqe

Let $U^0 := r^{-1}(U-z) \in \mcl U_1(0)$ and let $g^0 : U^0 \rta [0,1]$ be a smooth compactly supported bump function on $U^0$ which is identically equal to 1 on $U^0_{\delta/4 }$.
Let $g : U \rta [0,1]$ be defined by $g(\cdot) = g^0(r^{-1}(\cdot-z))$. 
Then the Dirichlet energy of $g$ equals the Dirichlet energy of $g^0$, which depends only on $U^0$. 
Let $f := \xi^{-1} (\log C - \log(c e^{-\xi A}/100)) g$. 
By Axiom~\ref{item-metric-f}, if the event in~\eqref{eqn-cond-diam-pos} occurs, then $G^U$ occurs with $\rng h^U -  f$ in place of $\rng h^U$. 

By a standard calculation for the GFF, the laws of $\rng h^U$ and $\rng h^U -f$ are mutually absolutely continuous and the law of the Radon-Nikodym derivative depends only on the Dirichlet energy of $f$, which in turn depends only on $U^0, \delta, c,A$. 
It follows that $\BB P[G^U]$ is bounded below by a constant depending only on $U^0, \delta , c, A$. 
Since the number of possibilities for $U^0$ depends only on $\delta$, by taking the minimum over all such possibilities we get~\eqref{eqn-cond-diam-good} for an appropriate choice of $\frk p$. 
\medskip

\noindent\textit{Step 3: adding the conditioning on $F_r(z)$.}
We now use the FKG inequality to add in the conditioning on $F_r(z)$. Indeed, under the conditional law given $h|_{\BB C\setminus U}$, both $\BB 1_{G^U}$ and $\BB 1_{F_r(z)}$ are non-increasing functions of the metric $D_{\rng h^U}$ (note that $\BB 1_{F_r(z)}$ depends also on $h|_{\BB C\setminus U}$, but we can still view it as a function of $\rng h^U$ when we condition on a fixed realization of $h|_{\BB C\setminus U}$). Moreover, it is easily seen that these functions are a.s.\ continuous at $D_{\rng h^U}$ in the sense of Proposition~\ref{prop-fkg-metric}: in the case of $G^U$, this follows since the probability that the supremum in~\eqref{eqn-cond-diam-event} is exactly equal to $\frac{c}{100}    e^{-\xi A} \frk c_r$ is zero. A similar justification holds for $F_r(z)$.
 By Proposition~\ref{prop-fkg-metric}, the events $F_r(z)$ and $G^U$ are positively correlated under the conditional law given $h|_{\BB C\setminus U}$.  
Therefore,~\eqref{eqn-cond-diam-good} implies that~\eqref{eqn-cond-diam-show} holds.
\end{proof}

\subsection{Cutting off geodesics from a boundary arc}
\label{sec-geo-kill}

We will now use the events of the preceding subsection to build ``shields" which prevent $D_h$-geodesics from hitting a given arc of a filled metric ball. 
Fix parameters $c , \delta \in (0,1)$ and $A>0$ and define $E_r(z) = E_r(z; c,\delta,A)$ as in~\eqref{eqn-clsce-all-event}
For $z\in\BB C$ and $\BB r >0$, let $\rho_{\BB r}^0(z) := \BB r$ and for $n\in\BB N$, inductively define
\eqb \label{eqn-good-radius-def}
\rho_{\BB r}^n(z) := \inf\left\{ r \geq 6\rho_{\BB r}^{n-1}(z) : \text{$r = 2^k \BB r$ for some $k\in\BB Z$}, \: E_r(z) \: \text{occurs} \right\} .
\eqe
Since $E_r(z)$ is determined by $h|_{\BB A_{2r,5r}(z)}$, it follows that $\rho_{\BB r}^n(z)$ is a stopping time for the filtration generated by $h|_{B_{5r}(z)}$ for $r \geq \BB r$. 

\begin{lem} \label{lem-clsce-iterate}
For each $q > 0$, we can find parameters $c,\delta\in(0,1)$ and $A >0$ and another parameter $\eta > 0$, all depending on $q$, such that the following is true.
Uniformly over all $\BB r > 0$ and $z\in\BB C$, 
\eqb
\BB P\left[ \rho_{\BB r}^{\lfloor \eta\log C \rfloor}(z)  > C \BB r \right] = O_C(C^{-q} ) ,\quad \text{as $C\rta\infty$}.
\eqe
\end{lem}
\begin{proof}
Since each $E_r(z)$ is determined by $(h-h_r(z))|_{\BB A_{2r,5r}(z)}$, this follows by combining Lemma~\ref{lem-annulus-iterate-inverse} (applied with $r_k = 8^k \BB r$, say, $S_1 = 2$, $S_2= 5$, $a = \log q$, $b = 1/2$, say, and $K =\lfloor\log_6 C\rfloor$) and Lemma~\ref{lem-clsce-event-pos}.  
\end{proof}

\begin{lem} \label{lem-clsce-all}
There exists a choice of parameters $c,\delta\in(0,1)$ and $A >0$ and another parameter $\eta > 0$, depending only on the choice of metric $D$, such that the following is true. 
For each compact set $K\subset\BB C$, it holds with probability $1-O_\ep(\ep^2)$ (at a rate depending on $K$) that
\eqb
\rho_{\ep \BB r}^{\lfloor \eta\log\ep^{-1} \rfloor}(z) \leq   \ep^{1/2} \BB r , \quad\forall z\in \left(\frac{\ep \BB r}{4} \BB Z^2 \right) \cap B_{\ep \BB r}(\BB r K) .
\eqe
\end{lem}
\begin{proof}
This follows from Lemma~\ref{lem-clsce-iterate} (applied $\ep \BB r$ in place of $\BB r$, with $C =\ep^{-3/2}$, and with $q = 6$, say) and a union bound over $O_\ep(\ep^{-2})$ points in $\left(\frac{\ep \BB r}{4} \BB Z^2 \right) \cap B_{\ep \BB r}(\BB r K)$. 
\end{proof}

Henceforth assume that $c , \delta , A$, and $\eta$ are as in Lemma~\ref{lem-clsce-all}.
For $\ep > 0$, $\BB r > 0$, and a compact set $K\subset\BB C$, let 
\eqb \label{eqn-extra-radius-eucl}
R_{\BB r}^\ep(K) :=  6 \sup\left\{ \rho_{\ep \BB r}^{\lfloor \eta \log \ep^{-1} \rfloor}(z) : z\in \left( \frac{\ep \BB r}{4} \BB Z^2 \right) \cap B_{\ep \BB r}\left( K \right) \right\} +\ep\BB r ,
\eqe
so that each of the radii $\rho_{\ep \BB r}^n(z)$ for $z\in \left( \frac{\ep \BB r}{4} \BB Z^2 \right) \cap B_{\ep \BB r}\left( K \right)$ and $n\in[1,\eta\log\ep^{-1}]_{\BB Z}$ is determined by $R_{\BB r}^\ep(K)$ and $h|_{B_{R_{\BB r}^\ep(K)}(K)}$. 
Lemma~\ref{lem-clsce-all} shows that for each fixed choice of $K$, $\BB P[ R_{\BB r}^\ep(r K) \leq (6\ep^{1/2} + \ep) \BB r ]$ tends to 1 as $\ep\rta 0$, at a rate which is uniform in $\BB r$.  

Recall from Section~\ref{sec-geodesic} that $\mcl B_s^\bullet$ for $s > 0$ denotes the filled $D_h$-ball of radius $s$ centered at zero. 
For $s >0$, define
\eqb \label{eqn-extra-radius}
\sigma_{s,\BB r}^\ep := \inf\left\{ s' > s :   B_{R_{\BB r}^\ep(\mcl B_s^\bullet)}(\mcl B_s^\bullet) \subset \mcl B_{s'}^\bullet  \right\}   ,
\eqe  
so that $\mcl B_{\sigma_{s,\BB r}^\ep}^\bullet$ contains $B_{6\rho_{\ep \BB r}^{\lfloor\eta\log\ep^{-1}\rfloor}(z)}(z)$ for each $z\in B_{\ep \BB r}(\mcl B_s^\bullet)$. 
Since each $\rho_{\ep \BB r}^{\lfloor \eta \log \ep^{-1} \rfloor}(z)$ is a stopping time for the filtration generated by $h|_{B_{5r}(z)}$ for $r\geq \ep\BB r$, it follows that if $\tau$ is a stopping time for $\left\{ \left( \mcl B_t^\bullet , h|_{\mcl B_t^\bullet} \right) \right\}_{t\geq 0}$, then so is $\sigma_{\tau , \BB r}^\ep$. 
The following lemma will be used to ``kill off" the $D_h$-geodesics from 0 which hit a given boundary arc of a filled $D_h$-metric ball.

\begin{lem}
\label{lem-geo-kill-pt}
There exists $\alpha  >0$, depending only on the choice of metric, such that the following is true. 
Let $\BB r > 0$, let $\tau$ be a stopping time for the filtration generated by $\left\{ \left( \mcl B_s^\bullet , h|_{\mcl B_s^\bullet} \right) \right\}_{s \geq 0}$,
and let $x\in\bdy\mcl B_\tau^\bullet$ and $\ep\in (0,1)$ be chosen in a manner depending only on $( \mcl B_\tau^\bullet  , h|_{\mcl B_\tau^\bullet} )$. 
There is an event $ G_x^\ep \in \sigma\left(\mcl B_{  \sigma_{\tau,\BB r}^\ep}^\bullet  , h|_{B_{  \sigma_{\tau,\BB r}^\ep}^\bullet} \right)$ with the following properties.
\begin{enumerate}[A.]
\item If, in the notation~\eqref{eqn-extra-radius-eucl}, we have $ R_{\BB r}^\ep(\mcl B_\tau^\bullet) \leq   \op{diam} \mcl B_\tau^\bullet$ and $G_x^\ep $ occurs, then no $D_h$-geodesic from 0 to a point in $\BB C\setminus B_{R_{\BB r}^\ep(\mcl B_\tau^\bullet)}(\mcl B_\tau^\bullet)$ can enter $B_{\ep \BB r}(x) \setminus \mcl B_\tau^\bullet$. \label{item-geo-event-kill-pt}
\item There is a deterministic constant $C_0 >1$ depending only on the choice of metric such that a.s. $\BB P\left[     G_x^\ep \,\big|\, \mcl B_\tau^\bullet  , h|_{\mcl B_\tau^\bullet} \right] \geq 1 -  C_0 \ep^\alpha$. \label{item-geo-event-prob-pt}
\end{enumerate}
\end{lem} 
\begin{proof}
See Figure~\ref{fig-geo-kill} for an illustration of the proof.   
We will outline the argument just below, after introducing some notation.
\medskip

\begin{figure}[t!]
 \begin{center}
\includegraphics[scale=.85]{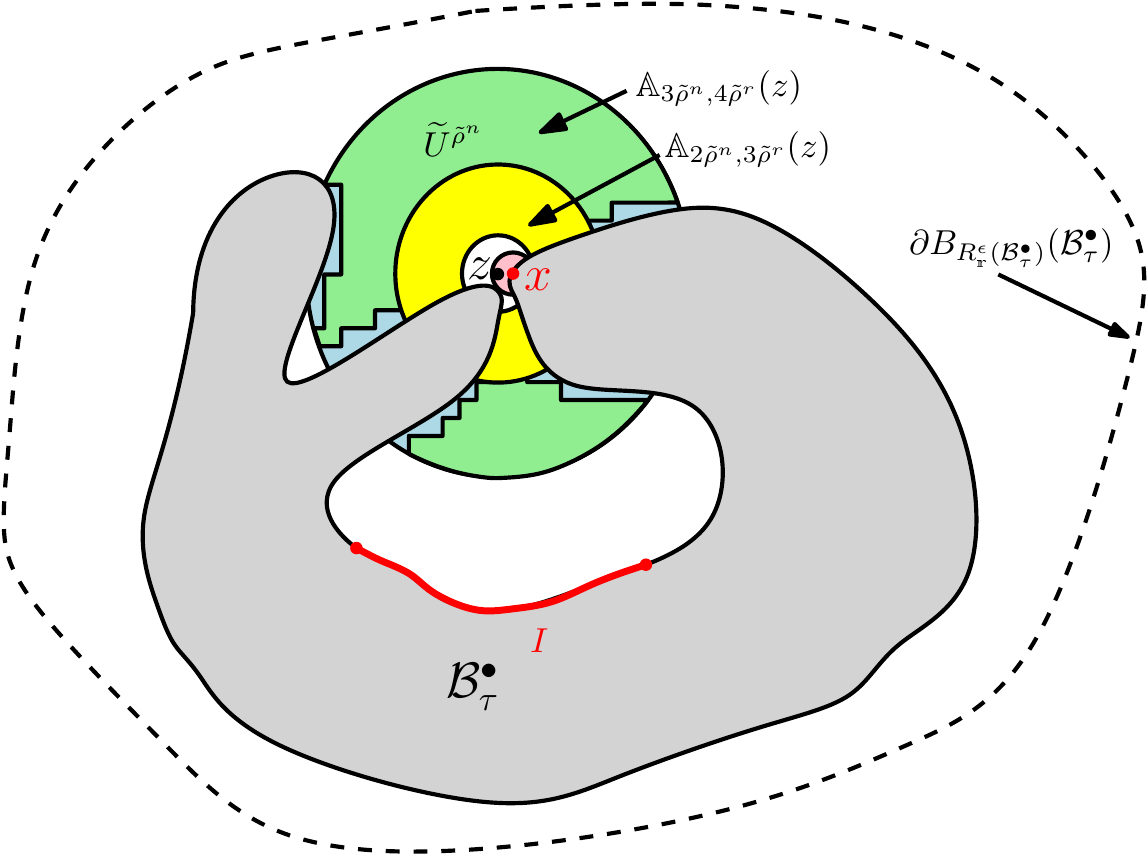}
\vspace{-0.01\textheight}
\caption{Illustration of the proof of Lemma~\ref{lem-geo-kill-pt}. The point $z \in \frac{\ep \BB r}{4} \BB Z^2$ is chosen so that $B_{\ep \BB r}(x) \subset B_{2\ep \BB r}(z)$. On the event $G_x^\ep$ defined in~\eqref{eqn-geo-event-def}, there is some $n \in [1,\eta\log\ep^{-1}]_{\BB Z}$ for which the following is true. With $\wt\rho^n$ as in~\eqref{eqn-good-radius-def'}, the sum of the $D_h$-diameter of the light green region $\wt U^{\wt\rho^n}$ and the maximal $D_h$-diameters of the light blue squares $S\in\mcl S_{\delta\wt\rho^n}^z(\BB A_{2\wt\rho^n,4\wt\rho^n}(z))$ is strictly smaller than the $D_h$-distance across the yellow annulus $\BB A_{2\wt\rho^n,3\wt\rho^n}(z)$. 
This makes it so that each point in $\BB A_{3\wt\rho^n,4\wt\rho^n}(z)$ is $D_h$-closer to a point of $\mcl B_\tau^\bullet\cap \BB A_{3\wt\rho^n,4\wt\rho^n}(z)$ than it is to $B_{\ep \BB r}(x)$. Hence no $D_h$-geodesic from a point outside of $B_{R_{\BB r}^\ep(\mcl B_s^\bullet)}(\mcl B_s^\bullet)$ to 0 can enter $B_{\ep \BB r}(x)\setminus \mcl B_\tau^\bullet$.
The condition that $ R_{\BB r}^\ep(\mcl B_\tau^\bullet) \leq  \op{diam} \mcl B_\tau^\bullet$ ensures that $\BB A_{3\wt\rho^n,4\wt\rho^n}(z)$ intersects $\mcl B_\tau^\bullet$. 
We can also prevent $D_h$-geodesics from hitting an arc $I$ of $\bdy\mcl B_\tau^\bullet$ by choosing $x$ so that $B_{\ep \BB r}(x)$ disconnects $I$ from $\infty$ in $\BB C\setminus\mcl B_\tau^\bullet$; see Lemma~\ref{lem-geo-kill}. 
}\label{fig-geo-kill}
\end{center}
\vspace{-1em}
\end{figure}

\noindent\textit{Step 1: setup.}
We can choose $z\in \left( \frac{\ep \BB r}{4} \BB Z^2 \right) \cap B_{\ep \BB r}\left( \mcl B_\tau^\bullet  \right)$ such that $B_{\ep \BB r}(x) \subset B_{2\ep \BB r}(z)$, in a manner depending only on $( \mcl B_\tau^\bullet  , h|_{\mcl B_\tau^\bullet} )$.  
Recalling the set of squares $\mcl S_{\delta r}^z(\cdot)$ from~\eqref{eqn-square-def}, for $r > 0$ we define
\eqb \label{eqn-ball-comp-set}
\wt U^r := \wt U^r(z) := \BB A_{3r,4r}(z) \setminus  \bigcup \left\{ S \in \mcl S_{\delta r}^z(\BB A_{3r,4r}(z)) : S \cap \mcl B_\tau^\bullet \not=\emptyset \right\} .
\eqe
Note that $\wt U^r$ belongs to the set $\mcl U_r(z)$ of Section~\ref{sec-clsce-event} and $\wt U^r$ is determined by $(\mcl B_\tau^\bullet , h|_{\mcl B_\tau^\bullet})$. 

Let $\wt\rho^0 := \ep \BB r$ and for $n\in\BB N$, inductively define
\eqb \label{eqn-good-radius-def'}
\wt\rho^n = \wt\rho_{\ep \BB r}^n(z) := \inf\left\{ r \geq 6\wt\rho^{n-1} : \text{$r = 2^k \BB r$ for some $k\in\BB Z$}, \: E_r^{\wt U^r}(z) \: \text{occurs} \right\} .
\eqe
In other words, $\wt\rho^n$ is defined in the same manner as $\rho_{\ep \BB r}^n(z)$ from~\eqref{eqn-good-radius-def} (with $\ep \BB r$ in place of $\BB r$) but with $E_r^{\wt U^r}(z)$ instead of $E_r(z)$.
This means that $E_{\wt\rho^n}^U(z)$ is only required to occur for $U = \wt U^{\wt\rho^n}$ instead of for every $U\in\mcl U_{\wt\rho^n}(z)$. 
By this and the definition~\eqref{eqn-extra-radius-eucl} of $R_{\BB r}^\ep(\mcl B_\tau^\bullet)$, 
\eqb \label{eqn-good-radii-compare}
\wt\rho^n \leq \rho_{\ep \BB r}^n(z) ,\: \forall n \in \BB N_0 \quad \text{and hence} \quad \wt\rho^{\lfloor \eta\log\ep^{-1}\rfloor} \leq \frac16 R_{\BB r}^\ep(\mcl B_\tau^\bullet).
\eqe
The reason for considering $\wt\rho^n$ instead of $\rho_{\ep \BB r}^n(z)$ is because we can only condition on $E_r^U(z)$, not on $E_r(z)$, in Lemma~\ref{lem-cond-diam-small}.
 
Recalling that $\mcl V(\wt U^{\wt\rho^n})$ denotes the set of connected components of $\wt U^{\wt\rho^n}$, we define 
\eqb \label{eqn-geo-event-def}
G_x^\ep := \left\{ \exists   n\in [1,\eta\log\ep^{-1}]_{\BB Z}  \: \text{s.t.}\: \max_{V\in\mcl V(\wt U^{\wt\rho^n} )} \sup_{u,v \in V} D_h\left( u , v ; \BB A_{ 2\wt\rho^n , 5\wt\rho^n }(z) \right) \leq \frac{c}{2} \frk c_{\wt\rho^n} e^{\xi h_{\wt\rho^n}(z)} \right\}  .
\eqe
In other words, $G_x^\ep$ is the event that the event of Lemma~\ref{lem-cond-diam-small} occurs for at least one of the sets $\wt U^{\wt\rho^n}$ for $n \in [1,\eta\log\ep^{-1}]_{\BB Z}$. 
\begin{itemize}
\item To check property~\ref{item-geo-event-kill-pt}, we will show that if $G_x^\ep$ occurs and $n$ is as in the definition of $G_x^\ep$, then no $D_h$-geodesic from a point outside of $\BB C\setminus B_{R_{\BB r}^\ep(\mcl B_\tau^\bullet)}(\mcl B_\tau^\bullet)$ to 0 can cross between the inner and outer boundaries of the annulus $\BB A_{2\wt\rho^n,4\wt\rho^n}(z)$. As we will explain in Step 2 below, the reason for this is that the $D_h$-distance from any point of $\BB A_{3\wt\rho^n,4\wt\rho^n}(z)$ to $\mcl B_\tau^\bullet$ is shorter than the $D_h$-distance across $\BB A_{2\wt\rho^n,3\wt\rho^n}(z)$, so it is more efficient to enter $\mcl B_\tau^\bullet$ before crossing $\BB A_{2\wt\rho^n,3\wt\rho^n}(z)$. 
\item To check property~\ref{item-geo-event-prob-pt}, we will first apply Lemma~\ref{lem-cond-diam-small} for a possible realization of $\wt U^{\wt\rho^n}$ to get that the conditional probability of $G_x^\ep$ given $z,\ep,\wt\rho^n, \wt U^{\wt\rho^n}$, and $h|_{\BB C\setminus \wt U^{\wt\rho^n}}$ is at least $\frk p$. We will then multiply this estimate over all $n\in [1,\eta\log\ep^{-1}]$ to get $\BB P[G_x^\ep] \geq 1- C_0 \ep^\alpha$ for $\alpha $ slightly less than $ \eta\log(1/(1-\frk p))$ and an appropriate choice of $C_0 > 1$.
\end{itemize}

Since $z$ and $\wt U^r$ for $r > 0$ are each determined by $(\mcl B_\tau^\bullet , h|_{\mcl B_\tau^\bullet})$, it follows that each $E^{\wt U^r}(z)$ is determined by $(\mcl B_\tau^\bullet , h|_{\mcl B_\tau^\bullet})$ and $h|_{\BB A_{2r, 5r}(z)}$. 
Hence $\wt\rho^n$ is a stopping time for the filtration generated by $h|_{B_{5r}(z)}$ for $r\geq \ep\BB r$ and $(\mcl B_\tau^\bullet , h|_{\mcl B_\tau^\bullet})$.
By~\eqref{eqn-good-radii-compare} and the definition~\eqref{eqn-extra-radius} of $\sigma_{\tau,\BB r}^\ep$, we have $B_{5\wt\rho^n}(z)  \subset \mcl B_{ \sigma_{\tau,\BB r}^\ep}^\bullet$.
By combining these statements with~\eqref{eqn-geo-event-def} and the locality of the metric (Axiom~\ref{item-metric-local}), we get that $G_x^\ep \in \sigma\left(\mcl B_{ \sigma_{\tau,\BB r}^\ep}^\bullet  , h|_{B_{  \sigma_{\tau,\BB r}^\ep}^\bullet} \right)$. 
\medskip

\noindent\textit{Step 2: proof that $G_x^\ep$ satisfies property~\ref{item-geo-event-kill-pt}.} 
Assume that $ R_{\BB r}^\ep(\mcl B_\tau^\bullet) \leq   \op{diam} \mcl B_\tau^\bullet$ and $G_x^\ep$ occurs. 
Choose $n\in [1,\eta\log\ep^{-1}]_{\BB Z}$ as in the definition~\eqref{eqn-geo-event-def} of $G_x^\ep$. 
Then 
\eqbn
\ep \BB r \leq \wt\rho^n \leq  \frac16  R_{\BB r}^\ep(\mcl B_\tau^\bullet)  \leq \frac16  \op{diam}\mcl B_\tau^\bullet  .
\eqen
By our choice of $z$, this means that $\BB A_{3\wt\rho^n , 4\wt\rho^n}(z)$ intersects $\mcl B_\tau^\bullet$ and $\BB A_{2\wt\rho^n , 3\wt\rho^n}(z)$ disconnects $B_{\ep \BB r}(x)$ from $\infty$. 

By~\eqref{eqn-ball-comp-set}, each point of $\BB A_{3\wt\rho^n,4\wt\rho^n}(z)\setminus \wt U^{\wt\rho^n}$ is contained in one of the squares $S \in \mcl S_{\delta \wt\rho^n}^z(\BB A _{3\wt\rho^n ,4\wt\rho^n}(z))$ such that $ S \cap \mcl B_\tau^\bullet \not=\emptyset$. 
By this and condition~\ref{item-clsce-ball} in the definition of $E_{\wt\rho^n}^{\wt U^{\wt\rho^n}}(z)$, each point of $\BB A_{3\wt\rho^n , 4\wt\rho^n}(z) \setminus \wt U^{\wt\rho^n}$ lies at $D_h$-distance at most $(c/100) \frk c_{\wt\rho^n} e^{\xi h_{\wt\rho^n}(z)} $ from $ \mcl B_\tau^\bullet$.
This together with the definition~\eqref{eqn-geo-event-def} of $G_x^\ep$ shows that 
\eqb \label{eqn-ball-complement-diam}
\sup_{u \in \BB A_{3\wt\rho^n , 4\wt\rho^n}(z)} D_h\left(u , \mcl B_\tau^\bullet \right)  < c \frk c_{\wt\rho^n} e^{\xi h_{\wt\rho^n}(z)}  .
\eqe

As a $D_h$-geodesic from a point outside of $\mcl B_\tau^\bullet$ to $0$ hits $\bdy\mcl B_\tau^\bullet$ exactly once, if such a geodesic hits $B_{\ep \BB r}(x) \setminus \mcl B_\tau^\bullet$, then it hits $B_{\ep \BB r}(x)$ \emph{before} entering $\mcl B_\tau^\bullet$. 
Therefore, to prove property~\ref{item-geo-event-kill-pt}, it suffices to consider a path $P$ from a point outside of $\BB C\setminus B_{R_{\BB r}^\ep(\mcl B_\tau^\bullet)}(\mcl B_\tau^\bullet)$ to 0 which enters $B_{\ep \BB r}(x)$ before entering $\mcl B_\tau^\bullet$ and show that $P$ cannot be a $D_h$-geodesic.

 Since $\BB A_{2\wt\rho^n  , 3\wt\rho^n }(z)$ disconnects $B_{\ep \BB r}(x)$ from $\infty$, the path $P$ must cross from the outer boundary of $\BB A_{2\wt\rho^n,3\wt\rho^n}(z)$ to the inner boundary of $\BB A_{2\wt\rho^n,3\wt\rho^n}(z)$ before hitting $B_{\ep \BB r}(x)$, and hence also before hitting $\mcl B_\tau^\bullet$.  
By condition~\ref{item-clsce-across} in the definition of $E_{\wt\rho^n}(z)$, each path between the inner and outer boundaries of $\BB A_{2\wt\rho^n,3\wt\rho^n}(z)$ has $D_h$-length at least $c\frk c_{\wt\rho^n} e^{\xi h_{\wt\rho^n}(z)}$. Hence, the $D_h$-length of the segment of $P$ after the first time it enters $\BB A_{2\wt\rho^n , 3\wt\rho^n}(z)$ must be at least $c \frk c_{\wt\rho^n} e^{\xi h_{\wt\rho^n}(z)}   + \tau$.  

But, $P$ must enter $\BB A_{3\wt\rho^n , 4\wt\rho^n}(z)$ before entering $\BB A_{2\wt\rho^n , 3\wt\rho^n}(z)$, so by~\eqref{eqn-ball-complement-diam} $P$ must hit a point at $D_h$-distance strictly smaller than $c \frk c_{\wt\rho^n} e^{\xi h_{\wt\rho^n}(z)}   $ from $\bdy\mcl B_\tau^\bullet$ before entering $\BB A_{2\wt\rho^n , 3\wt\rho^n}(z)$.
Such a point lies at $D_h$-distance strictly smaller than $c \frk c_{\wt\rho^n} e^{\xi h_{\wt\rho^n}(z)} + \tau $ from $0$. 
Combining this with the preceding paragraph shows that $P$ cannot be a $D_h$-geodesic to 0.  
\medskip

\noindent\textit{Step 3: proof that $G_x^\ep$ satisfies property~\ref{item-geo-event-prob-pt}.} 
For $n\in\BB N$, let 
\eqb \label{eqn-geo-event-once}
\wt G^n:= \left\{ \max_{V\in\mcl V(\wt U^{\wt\rho^n} )} \sup_{u,v \in V} D_h\left( u , v ; \BB A_{ 2\wt\rho^n , 5\wt\rho^n }(z) \right) \leq \frac{c}{2} \frk c_{\wt\rho^n} e^{\xi h_{\wt\rho^n}(z)} \right\}
\eqe
be the event appearing in the definition~\eqref{eqn-geo-event-def} of $G_x^\ep$, so that $G_x^\ep = \bigcup_{n=1}^{\lfloor\eta\log\ep^{-1}\rfloor} \wt G^n$. 

Let $\frk p$ be as in Lemma~\ref{lem-cond-diam-small} with our above choice of $c, \delta   , A$. Since we chose these parameters in a manner depending only on the choice of metric, $\frk p$ depends only on the choice of metric.
Just below, we will show using Lemma~\ref{lem-cond-diam-small} that a.s.\ 
\eqb  \label{eqn-use-cond-diam-small'} 
\BB P\left[ \wt G^n     \,\big|\,     \wt U^{\wt\rho^n} ,  h|_{\BB C\setminus \wt U^{\wt\rho^n}} \right] \geq \frk p , \quad\forall n \in \BB N.
\eqe 
Before proving~\eqref{eqn-use-cond-diam-small'}, we explain why~\eqref{eqn-use-cond-diam-small'} implies property~\ref{item-geo-event-prob-pt}. 
Recall that each $\wt\rho^n$ is a stopping time for the filtration generated by $h|_{B_{5r}(z)}$ for $r\geq 0$ and $(\mcl B_\tau^\bullet , h|_{\mcl B_\tau^\bullet})$ and each of the sets $\wt U^r$ for $r  > 0$ from~\eqref{eqn-ball-comp-set} is determined by $(\mcl B_\tau^\bullet , h|_{\mcl B_\tau^\bullet})$. 
By the locality of $D_h$ (Axiom~\ref{item-metric-local}), for each $n\in\BB N$ the event $\wt G^n$ of~\eqref{eqn-geo-event-once} is determined by $h|_{B_{6\wt\rho^n}(z)}$ and $(\mcl B_\tau^\bullet , h|_{\mcl B_\tau^\bullet})$. 

Since $\wt\rho^n \geq 6\wt\rho^{n-1}$ and $\wt U^{\wt\rho^n}$ is disjoint from $\mcl B_\tau^\bullet$, it follows that $h|_{B_{6\wt\rho^m}(z)}$ and $(\mcl B_\tau^\bullet , h|_{\mcl B_\tau^\bullet})$ and hence also $\wt G^m$ for $m\leq n-1$ is determined by $(\wt U^{\wt\rho^n} ,  h|_{\BB C\setminus \wt U^{\wt\rho^n}})$.  
Hence we can iterate~\eqref{eqn-use-cond-diam-small'} $\lfloor\eta\log\ep^{-1}\rfloor$ times to get that the conditional probability given $(\mcl B_\tau^\bullet, h|_{\mcl B_\tau^\bullet})$ that $\wt G^n$ does not occur for every $n\in [1,\eta\log\ep^{-1}]_{\BB Z}$ is at most $(1-\frk p)^{\lfloor \eta\log\ep^{-1} \rfloor}$.
That is, a.s.\ $\BB P\left[     G_x^\ep \,\big|\, \mcl B_\tau^\bullet  , h|_{\mcl B_\tau^\bullet} \right] \geq 1 -  C_0 \ep^\alpha$ for $\alpha $ slightly smaller than $ \eta \log ( 1 / (1-\frk p))$ and an appropriate choice of $C_0 >1$ depending only on the choice of metric. 

It remains to justify~\eqref{eqn-use-cond-diam-small'}. 
To this end, let $\frk e  >0$, let $\frk z \in \frac{\frk e \BB r}{4} \BB Z^2$, let $\frk r \geq 6^n \frk e \BB r$ be a dyadic multiple of $\frk e \BB r$, and let $\frk U \in \mcl U_{\frk r}(\frk z)$.
We will study the conditional law given $\{\ep = \frk e,  z =\frk z  , \wt\rho^n  =\frk r , \wt U^{\wt\rho^n} = \frk U \}$.
By Lemma~\ref{lem-cond-diam-small},
\eqb  \label{eqn-use-cond-diam-small} 
\BB P\left[ \max_{V\in \mcl V(\frk U)} \sup_{u,v \in V} D_h\left( u , v ; \BB A_{ \frk r ,5 \frk r}(\frk z) \right) \leq \frac{c}{2} \frk c_{\frk r}  e^{\xi h_{\frk r}(\frk z)}    \,\big|\,h|_{\BB C\setminus \frk U} , E_{\frk r}^{\frk U}(\frk z) \right] \geq \frk p .
\eqe 
We will now argue that 
\eqb \label{eqn-determined-from-outside}
\{ \ep = \frk e, z =\frk z  , \wt\rho^n  =\frk r ,  \wt U^{\wt\rho^n} = \frk U \} \in \sigma\left( h|_{\BB C\setminus \frk U} , \BB 1_{ E_{\frk r}^{\frk U}(\frk z)} \right) .
\eqe 

Recall that $\mcl B_\tau^\bullet$ is a local set for $h$ (Lemma~\ref{lem-ball-local}).
Since $\mcl B_\tau^\bullet$ is a.s.\ determined by $h$, the event $\{\mcl B_\tau^\bullet \cap\frk U = \emptyset\}$ is a.s.\ determined by $h|_{\BB C\setminus \frk U}$ and moreover $(\mcl B_\tau^\bullet , h|_{\mcl B_\tau^\bullet} )$ is a.s.\ determined by $h|_{\BB C\setminus \frk U}$ on $\{\mcl B_\tau^\bullet \cap\frk U = \emptyset\}$.  

The points $z$ and $\ep$ and the sets $\{\wt U^r\}_{r > 0}$ from~\eqref{eqn-ball-comp-set} are all determined by $(\mcl B_\tau^\bullet , h|_{\mcl B_\tau^\bullet} )$.  
Hence each of these objects is determined by $h|_{\BB C\setminus \frk U}$ on the event $\{\mcl B_\tau^\bullet \cap\frk U = \emptyset\}$. 
Each of the events $E_r^{\wt U^r}(z)$ is determined by $\mcl B_\tau^\bullet$ and $h|_{\BB A_{2r,5r}(z)}$. 
Since $\wt\rho^n$ is the \emph{smallest} radius $r \geq 6\rho^{n-1}$ which is a dyadic multiple of $\ep \BB r$ for which $E_r^{\wt U^r}(z)$ occurs, it follows that the event $\{\ep = \frk e , z=\frk z , \wt\rho^n =\frk r\}$ is determined by $(\mcl B_\tau^\bullet , h|_{\mcl B_\tau^\bullet})$, $E_{\frk r}^{\frk U}(z)$, and $h|_{\BB C\setminus \BB A_{2\frk r , 5\frk r}(z)}$. 
We have $\frk U\subset  \BB A_{2\frk r , 5\frk r}(z)$, so $\{\ep = \frk e , z=\frk z , \wt\rho^n =\frk r\}$ is a.s.\ determined by $\BB 1_{E_{\frk r}(z)}$ and $h|_{\BB C\setminus \frk U}$ on the event $\{\mcl B_\tau^\bullet \cap\frk U = \emptyset\}$. 

By~\eqref{eqn-ball-comp-set}, we have $\mcl B_\tau^\bullet \cap\frk U = \emptyset$ on the event $\{\wt U^{\wt\rho^n} = \frk U\}$. 
Combining this with the preceding two paragraphs gives~\eqref{eqn-determined-from-outside}. 
Combining~\eqref{eqn-use-cond-diam-small} and~\eqref{eqn-determined-from-outside} and using that $\{\frk z = z , \wt\rho^n = \frk r , \wt U^{\wt\rho^n} = \frk U\} \subset E_{\frk r}^{\frk U}(\frk z)$ by definition shows that~\eqref{eqn-use-cond-diam-small'} holds.
\end{proof}

We will most often use the following variant of Lemma~\ref{lem-geo-kill-pt} where we prevent $D_h$-geodesics from hitting a boundary arc rather than a neighborhood of a point.
 
\begin{lem} \label{lem-geo-kill}
Let $\alpha$ be as in Lemma~\ref{lem-geo-kill-pt}. 
Let $\BB r > 0$, let $\tau$ be a stopping time for the filtration generated by $\left\{ \left( \mcl B_s^\bullet , h|_{\mcl B_s^\bullet} \right) \right\}_{s \geq 0}$.
Also let $\ep \in (0,1)$ and $I\subset \bdy \mcl B_\tau^\bullet$ be an arc, each chosen in a manner depending only on $( \mcl B_\tau^\bullet  , h|_{\mcl B_\tau^\bullet} )$, such that $I$ can be disconnected from $\infty$ in $\BB C\setminus \mcl B_\tau^\bullet$ by a set of Euclidean diameter at most $\ep \BB r$.   
There is an event $G_I \in \sigma\left(\mcl B_{  \sigma_{\tau,\BB r}^\ep}^\bullet  , h|_{B_{  \sigma_{\tau,\BB r}^\ep}^\bullet} \right)$ with the following properties.
\begin{enumerate}[A.]
\item If $ R_{\BB r}^\ep(\mcl B_\tau^\bullet) \leq  \op{diam} \mcl B_\tau^\bullet$ and $G_I  $ occurs, then no $D_h$-geodesic from 0 to a point in $\BB C\setminus  \mcl B_{ \sigma_{\tau,\BB r}^\ep}^\bullet $ can pass through $I$. \label{item-geo-event-kill}
\item There is a deterministic constant $C_0 >1$ depending only on the choice of metric such that a.s. $\BB P\left[     G_I \,\big|\, \mcl B_\tau^\bullet  , h|_{\mcl B_\tau^\bullet} \right] \geq 1 -  C_0 \ep^\alpha$. \label{item-geo-event-prob}
\end{enumerate}
\end{lem} 
\begin{proof}
Since $I$ can be disconnected from $\infty$ in $\BB C\setminus \mcl B_\tau^\bullet$ by a set of Euclidean diameter at most $\ep \BB r$, we can choose a point $x\in\bdy\mcl B_\tau^\bullet$ in a manner depending only on $( \mcl B_\tau^\bullet  , h|_{\mcl B_\tau^\bullet} )$ such that $B_{\ep \BB r}(x)$ disconnects $I$ from $\infty$ in $\BB C\setminus \mcl B_\tau^\bullet$. 
Let $G_I := G_x^{ \ep}$ be the event of Lemma~\ref{lem-geo-kill-pt} for this choice of $x$. 
Then $G_I$ satisfies condition~\ref{item-geo-event-prob} in the lemma statement.
Moreover, $B_{\ep \BB r}(x) \subset B_{R_{\BB r}^\ep(\mcl B_\tau^\bullet)}(\mcl B_\tau^\bullet)$ by~\eqref{eqn-extra-radius-eucl}, so each path from a point in the unbounded connected component of $\BB C\setminus B_{R_{\BB r}^\ep(\mcl B_\tau^\bullet)}(\mcl B_\tau^\bullet)$ which first hits $\bdy\mcl B_\tau^\bullet$ at a point of $I$ must pass through $B_{\ep \BB r}(x)$. 
By~\eqref{eqn-extra-radius-eucl}, $\BB C\setminus \mcl B_{ \sigma_{\tau,\BB r}^\ep}^\bullet$ is contained in the unbounded connected component of $\BB C\setminus B_{R_{\BB r}^\ep(\mcl B_\tau^\bullet)}(\mcl B_\tau^\bullet)$. 
By this and the corresponding condition from Lemma~\ref{lem-geo-kill-pt}, we get that $G_I$ satisfies condition~\ref{item-geo-event-kill} in the lemma statement.
\end{proof}

\subsection{Proof of Theorem~\ref{thm-finite-geo}}
\label{sec-finite-geo}

Continue to fix parameters $c,\delta , A ,\eta$ for which the conclusion of Lemma~\ref{lem-clsce-all} holds. For the rest of the paper we will no longer need to recall the precise definitions of the events $E_r(z)$ and $E_r^U(z)$. Rather, we only need the conclusions of Lemmas~\ref{lem-clsce-all} and~\ref{lem-geo-kill}. 
 
We will actually prove a much more quantitative version of Theorem~\ref{thm-finite-geo} (see Theorem~\ref{thm-finite-geo-quant} below) which gives a quantitative bound on how large $N$ needs to be in terms of $t$ and $p$ provided we truncate on a global regularity event, which we now define. 

It is shown in~\cite[Theorem 1.7]{lqg-metric-estimates} that if $\xi =\gamma/d_\gamma$ and $Q = 2/\gamma+\gamma/2$, then $D_h$ is a.s.\ locally $\chi$-H\"older continuous w.r.t.\ the Euclidean metric for any $\chi \in (0,\xi(Q-2))$. 
Henceforth fix such a $\chi$, chosen in a manner depending only on $\xi$ and $Q$. 
We also recall the stopping time $\tau_{\BB r}$ from~\eqref{eqn-tau_r-def}. 
For $a \in (0,1)$, we define $\mcl E_{\BB r}(a)$ to be the event that the following is true.
\begin{enumerate}
\item $B_{a \BB r}(0) \subset \mcl B_{\tau_{\BB r}}^\bullet$.  \label{item-quantum-ball-contained}
\item $ \tau_{3\BB r} - \tau_{2\BB r} \geq a \frk c_{\BB r} e^{-\xi h_{\BB r}(0)}$. \label{item-quantum-ball-compare}
\item $\frk c_{\BB r}^{-1} e^{-\xi h_{\BB r}(0)} D_h(u,v) \leq    \left( \frac{ |u - v| }{\BB r} \right)^\chi$ for each $u,v \in B_{4 \BB r}(0)$ with $|u-v|/\BB r \leq a$. \label{item-holder-cont}
\item In the notation~\eqref{eqn-good-radius-def}, we have $\rho_{\ep \BB r}^{\lfloor \eta \log \ep^{-1} \rfloor}(z) \leq \ep^{1/2} \BB r $ for each $ z\in \left( \frac{\ep \BB r}{4} \BB Z^2 \right) \cap B_{4 \BB r}(0) $ and each dyadic $\ep \in (0,a]$.    \label{item-good-radii-small}
\end{enumerate}

\begin{lem} \label{lem-finite-geo-reg}
For each $p\in (0,1)$, there exists $a = a(p) > 0$ such that $\BB P[\mcl E_{\BB r}(a)] \geq p$ for every $\BB r >0$.
\end{lem}
\begin{proof}
By Axiom~\ref{item-metric-coord}, if $a$ is chosen sufficiently small then the probability of each of conditions~\ref{item-quantum-ball-contained} and~\ref{item-quantum-ball-compare} is at least $1-(1-p)/4$.   
By~\cite[Proposition 3.18]{lqg-metric-estimates}, if $a$ is chosen sufficiently small then the probability of condition~\ref{item-holder-cont} is at least $1-(1-p)/4$.
By Lemma~\ref{lem-clsce-all} and a union bound over dyadic values of $\ep$ with $\ep \in (0, a]$, the probability of condition~\ref{item-good-radii-small} is at least $1-(1-p)/4$.
\end{proof}

Theorem~\ref{thm-finite-geo} will be an immediate consequence of Lemma~\ref{lem-finite-geo-reg} together with the following quantitative estimate.
For the statement, we recall the times $\tau_{\BB r}$ from~\eqref{eqn-tau_r-def}.

\begin{thm} \label{thm-finite-geo-quant}
For each $a \in (0,1)$, there is a constant $b_0 > 0$ depending only on $a$ and constants $b_1 , \beta > 0$ depending only on the choice of metric $D$ such that the following is true. 
For each $\BB r>0$, each $N\in\BB N$, and each stopping time $\tau$ for $\{(\mcl B_s^\bullet , h|_{\mcl B_s^\bullet})\}_{s\geq 0}$ with $\tau \in [\tau_{\BB r} ,\tau_{2\BB r}]$ a.s., the probability that $\mcl E_{\BB r}(a)$ occurs and there are more than $N$ points of $\bdy\mcl B_{\tau}^\bullet$ which are hit by leftmost $D_h$-geodesics from 0 to $\bdy\mcl B_{\tau +  N^{-\beta} \frk c_{\BB r} e^{\xi h_{\BB r}(0)} }^\bullet$ is at most $b_0 e^{-b_1 N^\beta}$. 
\end{thm}

In the rest of this section we prove Theorem~\ref{thm-finite-geo-quant}. 
Fix $\BB r > 0$ and a stopping time $\tau$ as in Theorem~\ref{thm-finite-geo-quant}.
Let $\mcl I_0$ be a collection of disjoint boundary arcs of $\mcl B_{\tau}^\bullet$, chosen in a manner depending only on $(\mcl B_{\tau}^\bullet , h|_{\mcl B_{\tau}^\bullet} )$. We allow arcs to be open, half-open, or closed. In particular, the union of all of the arcs of $\mcl I_0$ is allowed to be all of $\bdy\mcl B_{\tau}^\bullet$ (this is in fact the typical case we will be interested in).  The idea of the proof is to show that for \emph{any} choice of $\mcl I_0$, the probability that $\mcl E_{\BB r}(a)$ occurs and there are more than $N$ arcs in $\mcl I_0$ which are hit by leftmost $D_h$-geodesics from 0 to $\bdy\mcl B_{\tau_{\BB r}  +  N^{-\beta} \frk c_{\BB r} e^{\xi h_{\BB r}(0)} }^\bullet$ is at most $b_0 e^{-b_1 N^\beta}$.  Taking $\mcl I_0$ to be a huge collection of tiny arcs will then prove Theorem~\ref{thm-finite-geo-quant}.  See Figure~\ref{fig-metric-ball-iterate} for an illustration of the setup.

\begin{figure}[t!]
 \begin{center}
\includegraphics[scale=.85]{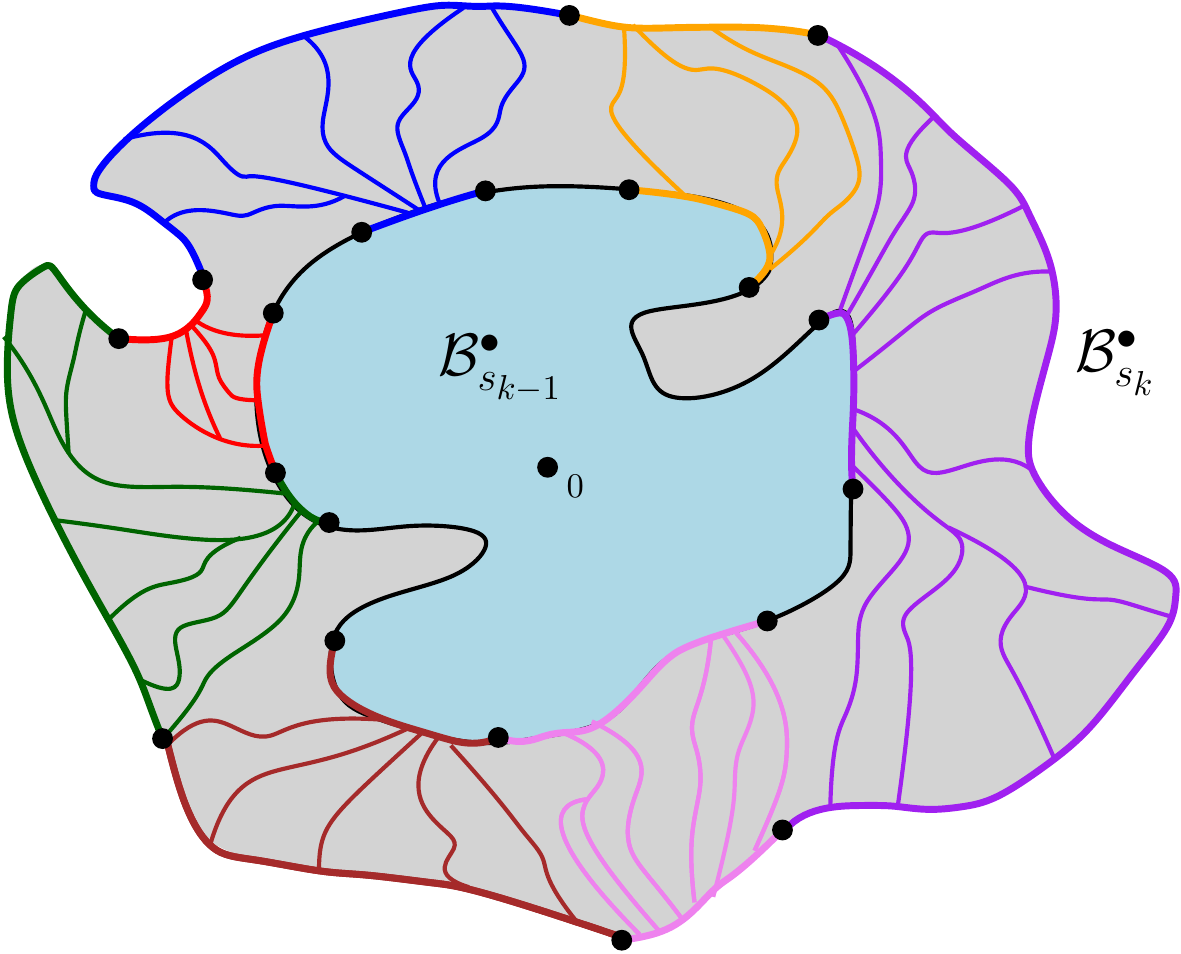}
\vspace{-0.01\textheight}
\caption{Illustration of one stage of the iterative procedure used in the proof of Theorem~\ref{thm-finite-geo}. The endpoints of the arcs in $\mcl I_{k-1}$ are shown with black dots. Each of the intervals $I\in\mcl I_{k-1}$ with $I' \not=\emptyset$ (i.e., those which are hit by leftmost geodesics from 0 to $\bdy\mcl B_{s_k}^\bullet$) is assigned a unique color. The corresponding arc $I' \subset\mcl B_{s_k}^\bullet$ and some representative leftmost geodesics from $I'$ to $I$ are shown in the same color. Black arcs $I\in\mcl I_{k-1}$ are ones with $I'=\emptyset$. Note that here we have shown some geodesics merging into each other, but we have not yet established whether or not this happens. We show in Lemma~\ref{lem-half-count} that typically the number of arcs in $\mcl I_k$ decays geometrically in $k$. 
}\label{fig-metric-ball-iterate}
\end{center}
\vspace{-1em}
\end{figure}

\subsubsection{Inductive definition of radii and boundary arcs}

We start by inductively defining for each $k\in\BB N$ a radius $s_k \geq s_{k-1}$ and a finite collection $\mcl I_k$ of disjoint boundary arcs of $\mcl B_{s_k}^\bullet$, chosen in a manner depending only on $(\mcl B_{s_k}^\bullet , h|_{\mcl B_{s_k}^\bullet} )$ and satisfying $\#\mcl I_k \leq \#\mcl I_{k-1}$. 

Set $s_0 = \tau$ and let $\mcl I_0$ be as above.  
Inductively, suppose $s_{k-1}$ and $\mcl I_{k-1}$ have been defined. 
Let $\ep_{k-1} $ be the smallest dyadic number with $\ep_{k-1} \geq (\#\mcl I_{k-1})^{-1/4}$ and in the notation~\eqref{eqn-extra-radius} define 
\eqb \label{eqn-iterate-time-def}
s_k :=   \sigma_{s_{k-1},\BB r}^{\ep_{k-1} } .
\eqe
For $I\in\mcl I_{k-1}$, let $I'$ be the set of points $x  \in \bdy \mcl B_{s_k}^\bullet$ for which the leftmost $D_h$-geodesic from 0 to $x$ passes through $I$.
Note that we may have $I'=\emptyset$. Define
\eqb \label{eqn-arc-set-def}
\mcl I_k := \{I' : I \in \mcl I_{k-1} , I'\not=\emptyset\} .
\eqe 

We make the following observations about $\mcl I_k$. 
\begin{itemize}
\item By Lemma~\ref{lem-geo-arc}, $\mcl I_k$ is a collection of disjoint arcs of $\bdy\mcl B_{s_k}^\bullet$. The arcs in $\mcl I_k$ can be closed, open, or half-open. 
\item $\mcl I_k$ is determined by $(\mcl B_{s_k}^\bullet , h|_{\mcl B_{s_k}^\bullet} )$: indeed, this is because of Axiom~\ref{item-metric-local} and the fact that a $D_h$-geodesic from 0 to $\bdy\mcl B_{s_k}^\bullet$ cannot exit $\mcl B_{s_k}^\bullet$, so must also be a geodesic for the internal metric of $D_h$ on $\mcl B_{s_k}^\bullet$. 
\item If we let $I^{(k)}$ for $I\in\mcl I_0$ be the set of $x\in\bdy\mcl B_{s_k}^\bullet$ for which the leftmost $D_h$-geodesic from 0 to $x$ pass through $I$, then we can equivalently define $\mcl I_k = \{I^{(k)} : I\in \mcl I_0 , I^{(k)} \not=\emptyset\}$. Indeed, this is because $I^{(k)}$ is obtained from $I$ by applying the operation $I\mapsto I'$ $k$ times.  
\end{itemize}
\medskip

\subsubsection{The cardinalities of the $\mcl I_k$'s decrease geometrically}

To lighten notation, in what follows we abbreviate 
\eqb \label{eqn-interval-count-def}
n_k := \#\mcl I_k \quad \text{so that} \quad \ep_k \in [n_k^{-1/4} , 2n_k^{-1/4}].
\eqe   
For $N\in\BB N$, we also define  
\eqb \label{eqn-finite-geo-time}
K_N :=  \min \left\{ k\in\BB N_0 :  s_k > \tau_{3 \BB r } \: \text{or} \: n_k < N  \right\} .
\eqe
We observe that $K_N$ is a stopping time for the filtration generated by $\{(\mcl B_{s_k}^\bullet , h|_{\mcl B_{s_k}^\bullet})\}_{k\in \BB N_0}$ and that $K_N$ is a non-increasing function of $N$.
We will also have occasion to consider a parameter $N_0 \in \BB N$, which we will eventually choose to be sufficiently large in a manner depending only on $a$ (at several places we will assume that $N_0$ is sufficiently large that some estimate is true). 
Most of our estimates will require that $N\geq N_0$. 

The key ingredient in the proof of Theorem~\ref{thm-finite-geo-quant} is the following lemma, which tells us that the arc counts $n_k$ typically decrease almost geometrically in $k$. 

\begin{lem} \label{lem-half-count}
Let $k_0 = 0$ and for $j\in\BB N$, inductively let $k_j$ be the smallest $k\geq k_{j-1}$ for which $n_k \leq \frac12 n_{k_{j-1}}$.  There are constants $C > 1$ and $\wt\beta > 0$ depending only on the choice of metric such that if $N_0$ is chosen to be sufficiently large, in a manner depending only on~$a$, then for $j \in \BB N_0$ and $M > 1$,  
\eqb \label{eqn-half-count}
\BB P\left[ k_{j+1}  - k_j  > M ,\, k_{j+1} \leq K_{N_0} ,\, \mcl E_{\BB r}(a) \,\big|\, \mcl B_{s_{k_j}}^\bullet , h|_{\mcl B_{s_{k_j}}^\bullet} \right] 
\leq (   n_{k_j} / C )^{ - \wt\beta M } .
\eqe
\end{lem}
 
To prove Lemma~\ref{lem-half-count}, we will show that, roughly speaking, $n_k \leq \frac12 n_{k-1}$ with high conditional probability given $(\mcl B_{s_k}^\bullet , h|_{\mcl B_{s_k}^\bullet})$, then multiply the resulting estimate over $M$ values of $k$ to get~\eqref{eqn-half-count}. 
Recall that $\ep_k$ is a negative power of 2 chosen so that $\ep_k \in [n_k^{-1/4} , 2 n_k^{-1/4})$. 
The basic idea is that Lemma~\ref{lem-disconnect-set-infty} tells us that most of the arcs of $\bdy\mcl B_{s_k}^\bullet$ can be disconnected from $\infty$ by sets of diameter smaller than $\ep_k$, and Lemma~\ref{lem-geo-kill} tells us that each of these arcs is unlikely to survive to the next step (i.e., one has $I' = \emptyset$ with high probability). 

Let us first record what we get from Lemmas~\ref{lem-disconnect-set-infty} and~\ref{lem-geo-kill}. 
For $k\in\BB N_0$, let $\mcl I_k^*$ be the set of ``bad" arcs $I\in\mcl I_k$ which \emph{cannot} be disconnected from $\infty$ in $\BB C\setminus \mcl B_{s_k}^\bullet$ by a set of Euclidean diameter at most $\ep_k  \BB r $.  
By Lemma~\ref{lem-disconnect-set-infty} applied with $n = n_k \asymp \ep_k^{-4} $ and $C \asymp   \BB r  / \ep_k$ and condition~\ref{item-quantum-ball-contained} in the definition of $\mcl E_{\BB r}(a)$, there is a universal constant $A  > 0$ such that on $\mcl E_{\BB r}(a)$, 
\eqb \label{eqn-good-arc-count}
\# \mcl I_k^* \leq A \frac{2^4}{a^2} n_k^{1/2}  ,\quad \forall k\in [0,K_1 -1]_{\BB Z} .
\eqe 
We henceforth assume that $N_0$ is chosen sufficiently large that $A \frac{2^4}{a^2} N_0^{-1/2}  \leq 1/4$, so that the right side of~\eqref{eqn-good-arc-count} is smaller than $n_k/4$ whenever $n_k \geq N_0$ (which in particular is the case if $k \leq K_{N_0}-1$).  
Then on $\mcl E_{\BB r}(a)$, 
\eqb \label{eqn-good-arc-count'}
\# \mcl I_k^* \leq \frac14 n_k ,\quad\forall k \in [0,K_{N_0} -1]_{\BB Z}. 
\eqe

We will now explain how Lemma~\ref{lem-geo-kill} allows us to ``kill off" most of the arcs not in $\mcl I_k^*$. 
For $I\in  \mcl I_k \setminus \mcl I^*_k $, let $G_I$ be the event of Lemma~\ref{lem-geo-kill} with $\tau = s_k$ and $\ep = \ep_k$ and define a second set of bad arcs
\eqb \label{eqn-bad-arc-set}
\mcl I_k^{**} := \left\{ I \in \mcl I_k \setminus \mcl I_k^* : \text{$G_I$ does not occur} \right\} .
\eqe 
By assertion~\ref{item-geo-event-kill} of Lemma~\ref{lem-geo-kill}, if (in the notation~\eqref{eqn-extra-radius-eucl}) we have $R_{\BB r}^{\ep_k }(\mcl B_{s_k}^\bullet) \leq   \op{diam}\left( \mcl B_{s_k}^\bullet \right) $ and $G_I$ occurs, then the interval $I'$ from the definition~\eqref{eqn-arc-set-def} of $\mcl I_{k+1}$ is empty.

We now want to say that this condition is satisfied for all $k \leq K_{N_0}-1$. 
Assume that $N_0$ is chosen sufficiently large that $N_0^{-1/4} \leq a/2$. 
Then $\ep_k \leq a$ for each $k \in [0,K_{N_0}-1]_{\BB Z}$ so if $\mcl E_{\BB r}(a)$ occurs, then condition~\ref{item-good-radii-small} in the definition of $\mcl E_{\BB r}(a)$ shows that
\eqb \label{eqn-good-radii-k}
\rho_{\ep_k \BB r}^{\lfloor \eta \log \ep^{-1} \rfloor}(z) \leq \ep_k^{1/2} \BB r ,\quad \forall  z\in \left( \frac{\ep_k r}{4} \BB Z^2 \right) \cap B_{4\BB r}(0) ,
\quad\forall k \in [0,K_{N_0}-1]_{\BB Z} .
\eqe
By condition~\ref{item-quantum-ball-contained} in the definition of $\mcl E_{\BB r}(a)$ and the definition~\eqref{eqn-finite-geo-time} of $K_{N_0}$, one has $B_{a \BB r}(0) \subset \mcl B_{s_k}^\bullet \subset B_{3 \BB r}(0)$ for each $k \in [0,K_{N_0}-1]_{\BB Z}$.
In particular, $\op{diam}(\mcl B_{s_k}^\bullet) \geq a\BB r$.  
Therefore,~\eqref{eqn-good-radii-k} together with the definition~\eqref{eqn-extra-radius-eucl} of $R_{\BB r}^{\ep_k}(\mcl B_{s_k}^\bullet)$ shows that if $\mcl E_{\BB r}(a)$ occurs, then
\eqb \label{eqn-good-radii-max-k}
R_{\BB r}^{\ep_k}(\mcl B_{s_k}^\bullet) 
\leq (6\ep_k^{ 1/2} + \ep_k) \BB r 
\leq  7 n_k^{-1/8} \BB r 
\leq \frac{7}{ a N_0^{1/8} } \op{diam}\left( \mcl B_{s_k}^\bullet \right)  ,\quad\forall k \in [0,K_{N_0}-1]_{\BB Z}. 
\eqe 
Hence, if we choose $N_0$ sufficiently large that $a  N_0^{1/8} \geq 7$, then on $\mcl E_{\BB r}(a)$, we have $R_{\BB r}^{\ep_k}(\mcl B_{s_k}^\bullet)  \leq \op{diam}\left( \mcl B_{s_k}^\bullet \right)$ for each $k \in [0,K_{N_0}-1]_{\BB Z}$. 
By combining this with the preceding paragraph, if $\mcl E_{\BB r}(a)$ occurs and $k \in [0,K_{N_0}-1]_{\BB Z}$, then $I' = \emptyset$ for every $I\in \mcl I_k \setminus (\mcl I_k^* \cup \mcl I_k^{**})$. Therefore,
\eqb \label{eqn-arc-count}
n_{k+1} = \#\mcl I_{k+1} \leq  \#\mcl I_k^* + \# \mcl I_k^{**} ,\quad\forall k \in [0,K_{N_0} -1]_{\BB Z} .
\eqe 

We will also need assertion~\ref{item-geo-event-prob} of Lemma~\ref{lem-geo-kill}, which tells us that there is an exponent $\alpha > 0$ and a constant $C_0 > 0$ depending only on the choice of metric such that for $k \in \BB N_0$, 
\eqb \label{eqn-arc-kill}
\BB P\left[ G_I \,\big|\,  \mcl B_{s_k}^\bullet , h|_{\mcl B_{s_k}^\bullet}   \right] \geq 1 - C_0 \ep_k^{\alpha }  ,\quad\forall I \in \mcl I_k \setminus \mcl I_k^* .
\eqe 

\begin{proof}[Proof of Lemma~\ref{lem-half-count}]
By~\eqref{eqn-bad-arc-set} and \eqref{eqn-arc-kill}, $\BB E\left[ \# \mcl I_k^{**}  \,\big|\,  \mcl B_{s_k}^\bullet , h|_{\mcl B_{s_k}^\bullet}   \right] \leq C_0 \ep_k^\alpha  \# \left( \mcl I_k \setminus \mcl I_k^* \right)$. 
By Markov's inequality, for $k\in\BB N_0$, 
\eqb \label{eqn-arc-kill'}
\BB P\left[ \# \mcl I_k^{**} > \frac14 \# \left( \mcl I_k \setminus \mcl I_k^* \right)  \,\big|\,  \mcl B_{s_k}^\bullet , h|_{\mcl B_{s_k}^\bullet}   \right]
\leq 4 C_0 \ep_k^{\alpha }  .
\eqe
For $k \in [k_j ,k_{j+1}-1]_{\BB Z}$, one has $n_k \geq \frac12 n_{k_j}$ and hence $\ep_k \leq   (  n_{k_j} /4)^{-1/4}$. 

By Lemma~\ref{lem-geo-kill}, $G_I \in \sigma (\mcl B_{s_{k+1}}^\bullet , h|_{\mcl B_{s_{k+1}}^\bullet} )$ for each $I\in \mcl I_k$. 
Since $k_j$ is a stopping time for the filtration generated by $(\mcl B_{s_k}^\bullet , h|_{\mcl B_{s_k}^\bullet}  )$ and $\mcl I_k , \mcl I_k^* , \mcl I_k^{**} \in \sigma (\mcl B_{s_{k+1}}^\bullet , h|_{\mcl B_{s_{k+1}}^\bullet} )$, 
we can set $k =  k_j + m$ in~\eqref{eqn-arc-kill'} and iterate $M$ times to get that 
\allb \label{eqn-arc-kill-iterate}
&\BB P\bigg[ k_{j+1} -k_j > M \: \text{and} \: 
\# \mcl I_{k_j + m}^{**} >  \frac14 \# \left( \mcl I_{k_j+m} \setminus \mcl I_{k_j+m}^* \right) ,\notag\\
&\qquad\qquad\qquad\qquad\qquad \forall m \in [0,M]_{\BB Z}
\,\big|\,  \mcl B_{s_{k_j}}^\bullet , h|_{\mcl B_{s_{k_j}}^\bullet}   \bigg]
\leq ( n_{k_j} / C)^{-\alpha M /4}  
\alle
for an appropriate constant $C>1$ depending only on the choice of metric. 
 
By~\eqref{eqn-good-arc-count'} and~\eqref{eqn-arc-count}, if $\# \mcl I_k^{**} \leq \frac14 \# \left( \mcl I_k \setminus \mcl I_k^* \right)$, $k\leq  K_{N_0}-1$, and $\mcl E_{\BB r}(a)$ occurs, then
\eqb \label{eqn-arc-count'}
n_{k+1} 
\leq  \#\mcl I_k^* + \# \mcl I_k^{**}  
\leq \frac12 n_k .
\eqe 
By the definition of $k_{j+1}$, if $m\in\BB N$, $\# \mcl I_{k_j+m}^{**} \leq \frac14 \# \left( \mcl I_{k_j+m} \setminus \mcl I_{k_j+m}^* \right)$, $k_j + m \leq K_{N_0} -1$, and $\mcl E_{\BB r}(a)$ occurs, then $k_{j+1} \leq k_j + m$. 
In other words, the event inside the conditional probability in~\eqref{eqn-arc-kill-iterate} contains the event $\{k_{j+1} - k_j > M\} \cap \{k_j \leq K_{N_0}-1\} \cap \mcl E_{\BB r}(a)$. 
We therefore get~\eqref{eqn-half-count} with $\wt\beta = \alpha / 4$. 
\end{proof}
  
\subsubsection{Conclusion of the proof}
We will now apply Lemma~\ref{lem-half-count} iteratively to bound how much we need to increase the radius of our metric balls to get down to $N$ remaining boundary arcs. 

\begin{lem} \label{lem-count-radius}
For each $a  \in (0,1)$, there are constants $b_0,b_1,\beta > 0$ as in the statement of Theorem~\ref{thm-finite-geo-quant} such that for each $\BB r > 0$ and $N\in\BB N$, 
\eqb
\BB P\left[ s_{K_N} >  \tau + N^{-\beta} \frk c_{\BB r} e^{\xi h_{\BB r}(0)}   ,\, \mcl E_{\BB r}(a) \right] \leq  b_0 e^{-b_1 N^\beta} .
\eqe
\end{lem}  
\begin{proof} 
Throughout the proof we write $c$ for a constant which is only allowed to depend on the metric and which may change from line to line.
\medskip

\noindent\textit{Step 1: bounding $s_{K_N}$ in terms of the $n_k$'s.}
By the first inequality in~\eqref{eqn-good-radii-max-k}, if $\mcl E_{\BB r}(a)$ occurs and $k\leq K_{N_0}-1$, then $R_{\BB r}^{\ep_k}(\mcl B_{s_k}^\bullet) \leq (\ep_k + 6\ep_k^{1/2}) \BB r$. 
That is, each point of each of the balls $B_{\rho_{\ep_k r}^{\lfloor \eta\log\ep_k^{-1} \rfloor}(z)}(z)$ for  $z\in \left( \frac{\ep_k r}{4} \BB Z^2 \right) \cap B_{\ep_k r}(\mcl B_{s_k}^\bullet)$ lies at Euclidean distance at most $(\ep_k + 6\ep_k^{1/2}) \BB r$ from $\mcl B_{s_k}^\bullet$. 
Since $k \leq K_{N_0}-1$, we have $\mcl B_{s_k}^\bullet \subset B_{3\BB r}(0)$. 
By the H\"older continuity condition~\ref{item-holder-cont} in the definition of $\mcl E_{\BB r} (a)$, if we choose $N_0$ large enough that $(N_0^{-1/4} + 6 N_0^{-1/8}) \leq a$, 
then each such ball is contained in $\mcl B_{s_k + \delta_k}^\bullet$ for $\delta_k :=  (\ep_k + 6  \ep_k^{1/2})^\chi \frk c_{\BB r} e^{\xi h_{\BB r}(0)}$. By the definition~\eqref{eqn-extra-radius} of $\sigma_{s_k , \BB r}^{\ep_k  }$ and~\eqref{eqn-iterate-time-def}, if $\mcl E_{\BB r}(a)$ occurs (and $N_0$ is chosen large enough) then
\eqb
s_{k+1}   = \sigma_{s_k , \BB r}^{\ep_k } \leq s_k + (\ep_k + 6 \ep_k^{1/2})^\chi \frk c_{\BB r} e^{\xi h_{\BB r}(0)} \leq s_k + c n_k^{-\chi/8} \frk c_{\BB r} e^{\xi h_{\BB r}(0)}   ,\quad \forall k \in [0,K_{N_0} -1]_{\BB Z} ,
\eqe
where in the last inequality we recall that $\ep_k \asymp n_k^{1/4}$. 
Summing this estimate and recalling that $s_0 = \tau $ shows that on $\mcl E_{\BB r}(a)$, 
\eqb \label{eqn-radius-upper}
s_{K_N} \leq \tau  + c \frk c_{\BB r} e^{\xi h_{\BB r}(0)}  \sum_{k=1}^{K_N-1} n_k^{-\chi/8}  ,\quad\forall N \geq N_0 .
\eqe  
\medskip

\noindent\textit{Step 2: bounding $s_{K_N}$ in terms of $k_j - k_{j-1}$.}
Fix $N\geq N_0$. 
Let $k_j$ for $j\in\BB N_0$ be as in Lemma~\ref{lem-half-count} and let $J = J_N$ be the smallest $j\in\BB N$ for which $k_{j} \geq K_N$. 
By definition, $n_{k_{j-1}} \geq 2 n_{k_j}$ for each $j\in\BB N_0$.
By iterating this, we get that $n_{k_j} \geq 2^{J-j-1} n_{k_{J-1}} \geq 2^{J-j-1} N$ for each $j \in [0,J-1]_{\BB Z}$.
Since $n_k \geq n_{k_j}/2$ for $k\in [k_j , k_{j+1} -1]_{\BB Z}$, it follows that
\eqb \label{eqn-count-lower}
n_{k_j} \geq 2^{J-j-2}N ,\quad\forall j \in [0,J-1]_{\BB Z}, \quad\forall k \in [k_j , k_{j+1}-1]_{\BB Z} .
\eqe
Since $k_J - 1 \geq K_N -1$, we can plug this into~\eqref{eqn-radius-upper} to get 
\eqb \label{eqn-radius-upper'}
s_{K_N} \leq \tau + c  N^{-\chi/8} \frk c_{\BB r} e^{\xi h_{\BB r}(0)}  \sum_{j=0}^{J-1} (k_{j+1} - k_j) 2^{-(J-j)}  .
\eqe 
\medskip

\noindent\textit{Step 3: bounding $k_j - k_{j-1}$.}
By taking unconditional expectations of both sides of~\eqref{eqn-half-count} from Lemma~\ref{lem-half-count} and then applying~\eqref{eqn-count-lower}, we get that for $M > 0$ and $j \in \BB N_0$, 
\alb
\BB P\left[ k_{j+1}  - k_j  > M ,\,  j\leq J-1 ,\, \mcl E_{\BB r}(a) \right] 
&\leq \BB E\left[  (n_{k_j}/C)^{-\wt\beta M} \right]  \notag\\
&\leq  (4C)^{ \wt\beta M  } N^{- \wt\beta M } \BB E\left[ 2^{- (J-j) \wt\beta M   }\right] .
\ale 
Now set $M = N^{\chi/16}$ and sum over all $j\in\BB N_0$ to get 
\allb
&\BB P\left[\exists j \in [0,J-1]_{\BB Z} \: \text{s.t.} \: k_{j+1}  - k_j   > N^{\chi/16} ,\, \mcl E_{\BB r}(a) \right]  \notag\\
&\qquad\qquad \leq (4C)^{   \wt\beta N^{\chi/16}  } N^{- \wt\beta  N^{\chi/16} } \BB E\left[ \sum_{j=0}^{J-1} 2^{- (J-j)\wt\beta  N^{\chi/16} }\right]  \notag \\
&\qquad\qquad \preceq (4C)^{ \wt\beta N^{\chi/16} } N^{-\wt\beta N^{\chi/16} } ,
\alle
with the implicit constant depending only on the choice of metric. 
This last quantity is bounded above by $b_0 e^{- b_1 N^\beta}$ for constants $b_0,b_1 , \beta > 0$ satisfying the conditions in Theorem~\ref{thm-finite-geo-quant}. 
We can arrange that $\beta < \chi/16$. 
Hence if $\mcl E_{\BB r}(a)$ occurs, then except on an event of probability at most $b_0 e^{-b_1 N^\beta}$, 
\eqb
k_{j+1}  - k_j   \leq  N^{\chi/16} ,\quad\forall j \in [0,J-1]_{\BB Z} .
\eqe
Plugging this bound into~\eqref{eqn-radius-upper'} shows that if $N\geq N_0$ and $\mcl E_{\BB r}(a)$ occurs, then except on an event of probability at most $b_0 e^{-b_1 N^\beta}$, 
\eqb \label{eqn-radius-upper''}
s_{K_N}  \leq \tau  + c N^{-\chi/16} \frk c_{\BB r} e^{\xi h_{\BB r}(0)}  \sum_{j=0}^{J-1} 2^{-(J-j)} 
\leq \tau  + c N^{-\beta}  \frk c_{\BB r} e^{\xi h_{\BB r}(0)} .
\eqe 
After possibly shrinking $\beta$ (in order to absorb the $c$ in~\eqref{eqn-radius-upper''} into a small power of $N^\beta$) and increasing $b_0$ (to make it so that $b_0 e^{-b_1 N^\beta} > 1$ for $N < N_0$), we obtain the statement of the lemma. 
\end{proof}

\begin{proof}[Proof of Theorem~\ref{thm-finite-geo-quant}]
Let $b_0 , b_1 , \beta >0$ be as in Lemma~\ref{lem-count-radius}. 
By condition~\ref{item-quantum-ball-compare} in the definition of $\mcl E_{\BB r}(a)$ and since $\tau\leq \tau_{2\BB r}$, if $N$ is sufficiently large that $N^{-\beta} < a$, then $s_{K_N} \leq \tau  + N^{-\beta} \frk c_{\BB r} e^{\xi h_{\BB r}(0)}$ implies that $s_{K_N}  < \tau_{3 \BB r}$.
By the definition~\eqref{eqn-finite-geo-time} of $K_N$, this implies that $\#\mcl I_{K_N} \leq N$. 
Hence, Lemma~\ref{lem-count-radius} implies that if $\mcl E_{\BB r}(a)$ occurs, then except on an event of probability at most $b_0 e^{-b_1 N^\beta}$, the number of $I\in\mcl I_0$ for which there is a leftmost $D_h$-geodesic from 0 to $\bdy\mcl B_{\tau  + N^{-\beta} \frk c_{\BB r} e^{\xi h_{\BB r}(0)} }^\bullet$ which passes through $I$ is at most $N$.  
This holds \emph{regardless of the initial choice of $\mcl I_0$}. 

Let $X_N$ be the set of points on $\bdy\mcl B_{\tau  }^\bullet$ which are hit by a leftmost $D_h$-geodesic from 0 to $\bdy\mcl B_{\tau  + N^{-\beta} \frk c_{\BB r} e^{\xi h_{\BB r}(0)} }^\bullet$.  
For a given $n_0 \in \BB N$, we choose $\mcl I_0$ to be a collection of $n_0$ disjoint half-open arcs of $\bdy\mcl B_\tau^\bullet$ which cover $\bdy\mcl B_{\tau}^\bullet$ and such that the harmonic measure from $\infty$ in $\BB C\setminus \mcl B_{\tau}^\bullet$ of each $I\in\mcl I_0$ is $1/n_0$. 
Applying the preceding paragraph to this choice of $\mcl I_0$, we find that if $\mcl E_{\BB r}(a)$ occurs, then except on an event of probability at most $b_0 e^{-b_1 N^\beta}$, the set $X_N$ can be covered by at most $N$ boundary arcs of $\mcl B_{\tau}^\bullet$ which each have harmonic measure from $\infty$ at most $1/n_0$. 
Since $n_0$ can be made arbitrarily large, this implies that on $\mcl E_{\BB r}(a)$, it holds except on en event of probability at most $b_0 e^{-b_1 N^\beta}$ that $\# X_N \leq N$. 
\end{proof}

\section{Reducing to a single geodesic}
\label{sec-one-geodesic}

In this section we will prove Theorem~\ref{thm-clsce}. 
Unlike in the case of Theorem~\ref{thm-finite-geo}, we do not prove a version which is uniform over the Euclidean scale since we do not need such a statement in~\cite{gm-uniqueness}. 
The only result from Section~\ref{sec-confluence} which we need in this section is Theorem~\ref{thm-finite-geo}.

To deduce Theorem~\ref{thm-clsce} from Theorem~\ref{thm-finite-geo} we need to reduce from finitely many points hit by geodesics to one point. 
To accomplish this, in Section~\ref{sec-one-geo-pos} we show that if $\tau_{\BB r}$ is as in~\eqref{eqn-tau_r-def}, $A>1$ is fixed, and $I\subset \bdy\mcl B_{\tau_{\BB r}}^\bullet$ is an arc chosen in a manner depending only on $( \mcl B_{\tau_{\BB r}}^\bullet  , h|_{\mcl B_{\tau_{\BB r}}^\bullet} )$, then with positive conditional probability given $( \mcl B_{\tau_{\BB r}}^\bullet  , h|_{\mcl B_{\tau_{\BB r}}^\bullet} )$, \emph{every} $D_h$-geodesic from 0 to a point outside of $B_{A \BB r}(0)$ passes through $I$ (see Lemma~\ref{lem-pos-kill} for a precise statement). 
This is proven using the Markov property of the GFF and the FKG inequality via a similar argument to the one in Section~\ref{sec-geo-kill}. 
The argument is significantly simpler, however, since the estimate we need is much less quantitative than the one in Lemma~\ref{lem-geo-kill}. See Figure~\ref{fig-one-geo} for an illustration of the argument. 

In Section~\ref{sec-one-geo-as}, we combine Theorem~\ref{thm-finite-geo} with the result of Section~\ref{sec-one-geo-pos}, applied with $I$ equal to the set of points in $x \in \bdy\mcl B_{\tau_{2\BB r}}^\bullet$ such that the leftmost $D_h$-geodesic from 0 to $x$ hits some specified point of $\bdy\mcl B_{\tau_{\BB r}}^\bullet$, to show that the following is true.
There is an $M>1$ such that with positive conditional probability given $( \mcl B_{\tau_{\BB r}}^\bullet  , h|_{\mcl B_{\tau_{\BB r}}^\bullet} )$, all of the leftmost $D_h$-geodesics from 0 to $\bdy\mcl B_{\tau_{M \BB r}}^\bullet$ hit the same point of $\bdy\mcl B_{\tau_{\BB r}}^\bullet$. Using the uniqueness of geodesics to rational points, we can improve this to say that with positive probability, \emph{all} of the geodesics from 0 to $\bdy\mcl B_{\tau_{M  \BB r }}^\bullet$ pass through the same point of $\bdy\mcl B_{\tau_{\BB r}}^\bullet$. 
Using a zero-one law argument, we then conclude the proof of Theorem~\ref{thm-clsce}.

\subsection{Killing off all but one geodesic with positive probability}
\label{sec-one-geo-pos}

The following lemma allows us to kill off all of the geodesics which do not hit a specified boundary arc $I\subset\bdy\mcl B_{\tau_{\BB r}}^\bullet$. 
For the statement, we recall the definition of the internal diameter metric $d^U$ for $U\subset\BB C$ a connected open set such that $\BB C\setminus U$ is compact, as in the discussion just after~\eqref{eqn-d^U-def}.

\begin{lem} \label{lem-pos-kill}
For each $A > 1$, $ \BB r > 0$, $\ep \in (0,(A-1)/100)$, and $p \in (0,1)$, there exists $\frk p = \frk p(A , \ep , p) > 0$ such that the following is true. 
Let $\tau_{\BB r}$ be as in~\eqref{eqn-tau_r-def}. 
Let $I\subset \bdy \mcl B_{\tau_{\BB r}}^\bullet$ be a closed boundary arc, chosen in a manner depending only on $( \mcl B_{\tau_{\BB r}}^\bullet  , h|_{\mcl B_{\tau_{\BB r}}^\bullet} )$, with the property that the $d^{\BB C\setminus \mcl B_{\tau_{\BB r}}^\bullet}$-neighborhood $\mcl B_{\ep \BB r}(\bdy\mcl B_{\tau_{\BB r}}^\bullet \setminus I ;  d^{\BB C\setminus \mcl B_{\tau_{\BB r}}^\bullet}  )$ does not disconnect $I$ from $\infty$ in $\BB C\setminus\mcl B_{\tau_{\BB r}}^\bullet$. 
With probability at least $p$, it holds with conditional probability at least $\frk p$ given $( \mcl B_{\tau_{\BB r}}^\bullet  , h|_{\mcl B_{\tau_{\BB r}}^\bullet} )$ that every $D_h$-geodesic from 0 to a point of $\BB C\setminus B_{A  \BB r }(0)$ passes through $I$.
\end{lem}

The proof of Lemma~\ref{lem-pos-kill} is similar to that of Lemma~\ref{lem-geo-kill}, but simpler since we do not need a quantitative bound on probabilities, so we only need to define one event rather than defining an event in every Euclidean annulus. See Figure~\ref{fig-one-geo} for an illustration of the proof.

\begin{figure}[t!]
 \begin{center}
\includegraphics[scale=1]{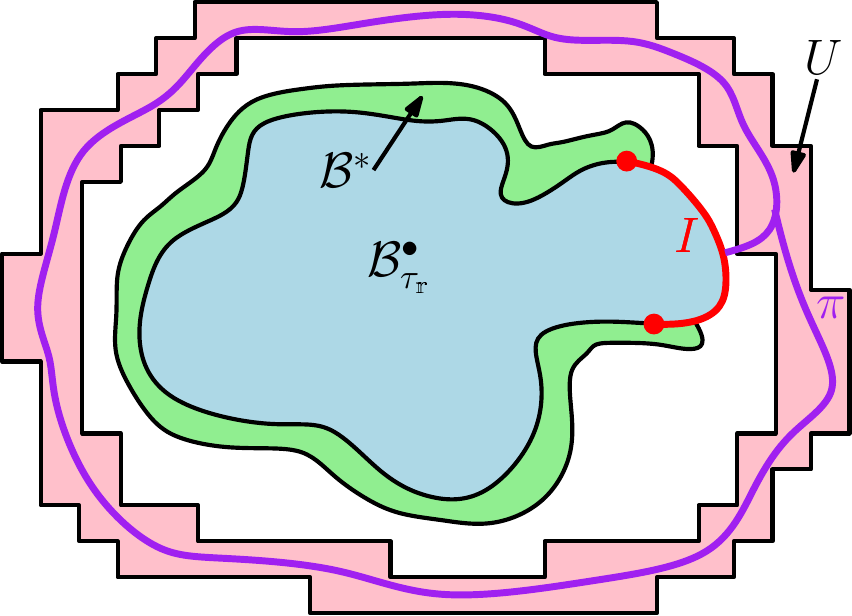}
\vspace{-0.01\textheight}
\caption{Illustration of the proof of Lemma~\ref{lem-pos-kill}. 
On the high-probability event $E_{\BB r}^U$, the $D_h$-distance from the $d^{\BB C\setminus\mcl B_{\tau_{\BB r}}^\bullet}$-neighborhood $\mcl B^* = \mcl B_{\ep \BB r/4}(\bdy\mcl B_{\tau_{\BB r}}^\bullet \setminus I ;  d^{\BB C\setminus \mcl B_{\tau_{\BB r}}^\bullet}  )$ to $\bdy\mcl B_{\tau_{\BB r}}^\bullet\setminus I$ is bounded below, the $D_h$-diameter of ever $\delta \BB r \times \delta \BB r$ square with corners in $\delta \BB r\BB Z^2$ which intersects $B_{A \BB r}(0)$ is small, and the absolute value of the harmonic part of $h|_U$ is bounded above on the set $U_{\delta \BB r/4}$ of points in $U$ lying at distance at least $\delta \BB r/4$ from $\bdy U$. 
In order to force every $D_h$-geodesic from 0 to a point of $\BB C\setminus B_{A \BB r}(0)$ to go through $I$, we consider a path $\pi$ started from $I$ which disconnects $\mcl B_{\tau_{\BB r}}^\bullet$ from $\bdy B_{A \BB r}(0)$ and let $U$ be the set of $\delta \BB r\times\delta \BB r$ squares as above which intersect $\pi$ but not $\mcl B_{\tau_{\BB r}}^\bullet$. 
Using the Markov property of the field and the FKG inequality (Lemma~\ref{lem-cond-diam-pos}), we see that with uniformly positive conditional probability given $( \mcl B_{\tau_{\BB r}}^\bullet  , h|_{\mcl B_{\tau_{\BB r}}^\bullet} )$, the $D_h$-distance from every point of $U$ to $I$ is smaller than the $D_h$-distance from $\bdy\mcl B^*$ to $\bdy\mcl B_{\tau_{\BB r}}^\bullet\setminus I$. This forces every $D_h$-geodesic to 0 which crosses $U$ to enter $\mcl B_{\tau_{\BB r}}^\bullet$ at a point of $I$. 
}\label{fig-one-geo}
\end{center}
\vspace{-1em}
\end{figure}

To lighten notation, write
\eqb \label{eqn-diam-ball-def}
\mcl B^* := \mcl B_{\ep \BB r/4}(\bdy\mcl B_{\tau_{\BB r}}^\bullet \setminus I ;  d^{\BB C\setminus \mcl B_{\tau_{\BB r}}^\bullet}  ) .
\eqe 
By Axiom~\ref{item-metric-coord}, we can find $c = c(A,\ep , p)  >0$ such that with probability at least $1 - (1-p)/3$, each path in $B_{A \BB r}(0)$ with Euclidean diameter at least $\ep \BB r/4$ has $D_h$-length at least $c \frk c_{\BB r} e^{\xi h_{\BB r}(0)}$. By the definition of $\mcl B^*$, each path in $\BB C\setminus \mcl B_{\tau_{\BB r}}^\bullet$ from $\bdy\mcl B_{\tau_{\BB r}}^\bullet \setminus I$ to a point of $\BB C\setminus (\mcl B_{\tau_{\BB r}}^\bullet \cup \mcl B^*)$ has Euclidean diameter at least $\ep \BB r/4$. 
Hence, with probability at least $1-(1-p)/3$, 
\eqb \label{eqn-pos-kill-across}
D_h\left(   \bdy\mcl B_{\tau_{\BB r}}^\bullet \setminus I , \BB C\setminus (\mcl B_{\tau_{\BB r}}^\bullet \cup \mcl B^*)  ; \BB C\setminus\mcl B_{\tau_{\BB r}}^\bullet \right) \geq c \frk c_{\BB r} e^{\xi h_{\BB r}(0)}    .
\eqe

Define the collection of $\delta \BB r\times \delta \BB r$ squares $\mcl S_{\delta \BB r}(B_{A\BB r}(0)) = \mcl S_{\delta \BB r}^0(B_{A\BB r}(0))$ with corners in $\delta \BB r\BB Z^2$ as in~\eqref{eqn-square-def} with $z=0$. 
Again by Axiom~\ref{item-metric-coord}, we can find $\delta =\delta(c,A,\ep) \in (0,\ep/100) $ such that with probability at least $1 - (1-p)/3$, 
\eqb \label{eqn-pos-kill-ball}
\sup_{u,v \in S} D_h(u,v) \leq \frac{c}{100} \frk c_{\BB r} e^{\xi h_{\BB r}(0)} ,\quad\forall S \in \mcl S_{\delta \BB r}(B_{A \BB r}(0)) .
\eqe

Let $\mcl U_{\BB r} $ be the (finite) set of sub-domains $U$ of $B_{ A \BB r}(0)$ such that $B_{ A \BB r}(0)  \setminus U$ is a finite union of sets of the form $S \cap B_{ A \BB r}(0)$ for $S \in \mcl S_{\delta \BB r}(B_{ A \BB r}(0))$. 
For $U\in\mcl U_r $, let $\frk h^U$ be the harmonic part of $h|_U$. Also let $U_{\delta \BB r/4}$ be the set of points in $U$ which lie at Euclidean distance at least $\delta \BB r/4$ from $\bdy U$. 
Since there are only finitely many sets in $\mcl U_{\BB r} $ and by the translation and scale invariance of the law of $h$, modulo additive constant, we can find $C = C(\delta , A , \ep ) > 0$ such that with probability at least $1 - (1-p)/3$, it holds simultaneously for each $U\in\mcl U_{\BB r}$ that
\eqb \label{eqn-pos-kill-harmonic}
\sup_{u   \in U_{\delta  \BB r  / 4}} |\frk h^U(u) - h_{\BB r}(0) |  \leq C    .
\eqe

For a given choice of $U\in \mcl U_{\BB r}$, let $E_{\BB r}^U$ be the event that~\eqref{eqn-pos-kill-across}, \eqref{eqn-pos-kill-ball}, and~\eqref{eqn-pos-kill-harmonic} all hold, so that 
\eqb \label{eqn-pos-kill-prob}
\BB P\left[ \bigcap_{U\in\mcl U_{\BB r}} E_{\BB r}^U \right] \geq p .
\eqe
The reason for considering $E_{\BB r}^U$ instead of $\bigcap_{U\in\mcl U_{\BB r}} E_{\BB r}^U$ is the same as in Section~\ref{sec-clsce-event}: it is easier to condition on $E_{\BB r}^U$, as explained in the following lemma.

\begin{lem} \label{lem-cond-diam-pos}
There is a constant $\frk p = \frk p(A,\ep , p) > 0$ such that the following is true.  
Also let $\mcl V(U)$ denote the set of connected components of $U$. 
On the event that $U\cap \left( \mcl B_{\tau_{\BB r}}^\bullet \cup \mcl B^* \right) = \emptyset$, a.s.\  
\eqb  \label{eqn-cond-diam-pos'} 
\BB P\left[ \max_{V\in\mcl V(U)} \sup_{u,v \in V} D_h\left( u , v \right) \leq \frac{c}{2} \frk c_{\BB r} e^{\xi h_{\BB r}(z)}    \,\big|\,h|_{\BB C\setminus U} , E_{\BB r}^U  \right] \geq \frk p .
\eqe  
\end{lem} 
\begin{proof}
Since $\mcl B_{\tau_{\BB r}}^\bullet$ is a local set for $h$ and $\mcl B^*$ is determined by $( \mcl B_{\tau_{\BB r}}^\bullet  , h|_{\mcl B_{\tau_{\BB r}}^\bullet} )$, the event $\{U\cap \left( \mcl B_{\tau_{\BB r}}^\bullet \cup \mcl B^*\right) = \emptyset\}$ is determined by $h|_{\BB C\setminus U}$. 
Moreover, the intersection of this event with the events in each of~\eqref{eqn-pos-kill-across} and~\eqref{eqn-pos-kill-harmonic} is also determined by $h|_{\BB C\setminus U}$. 
Hence, on the event  $U\cap \left( \mcl B_{\tau_{\BB r}}^\bullet \cup \mcl B^* \right) = \emptyset$, the conditional law of $h|_U$ given $h|_{\BB C\setminus U}$ and $E_{\BB r}^U$ is the same as its conditional law given $h|_{\BB C\setminus U}$ and the event from~\eqref{eqn-pos-kill-ball}. 
With this observation in hand, the lemma follows from the Markov property of the GFF and the FKG inequality (Proposition~\ref{prop-fkg-metric}) via exactly the same argument used to prove Lemma~\ref{lem-cond-diam-small}. 
\end{proof}

\begin{proof}[Proof of Lemma~\ref{lem-pos-kill}]
\noindent\textit{Step 1: choosing a random domain $U$.}
We first choose the domain $U$ to which we will apply Lemma~\ref{lem-cond-diam-pos}. 
The choice will depend on $\mcl B_{\tau_{\BB r}}^\bullet$ and $I$, which is why we need a lower bound for the probability of the intersection of all of the $E_{\BB r}^U$'s in~\eqref{eqn-pos-kill-prob}.

Since $\mcl B_{\ep \BB r}(\bdy\mcl B_{\tau_{\BB r}}^\bullet \setminus I ;  d^{\BB C\setminus \mcl B_{\tau_{\BB r}}^\bullet}  )$ does not disconnect $I$ from $\infty$ and $\mcl B_{\tau_{\BB r}}^\bullet\subset \ol{B_{\BB r}(0)}$, we can choose, in a manner depending only on $\mcl B_{\tau_{\BB r}}^\bullet$ and $I$, a path $\pi$ in $B_{(A-\ep) \BB r}(0)\setminus \mcl B_{\tau_{\BB r}}^\bullet$ which starts from a point of $I$, disconnects $\bdy B_{\BB r}(0)$ from $\bdy B_{(A-\ep) \BB r}(0)$, and lies at $d^{\BB C\setminus\mcl B_{\tau_{\BB r}}^\bullet}$-distance at least $\ep \BB r$ from $\bdy\mcl B_{\tau_{\BB r}}^\bullet\setminus I$. 
This path is shown in purple in Figure~\ref{fig-one-geo}. 
Let $U$ be the interior of the union of all of the $\delta \BB r\times \delta \BB r$ squares $S\in \mcl S_{\delta \BB r}(B_{A \BB r}(0))$ which intersect $\pi$ but do not intersect $\mcl B_{\tau_{\BB r}}^\bullet$.
Then $U\in\mcl U_{\BB r}$, as defined just above~\eqref{eqn-pos-kill-harmonic}. 

By definition, $U\cap \mcl B_{\tau_{\BB r}}^\bullet=\emptyset$. We claim that also $U\cap \mcl B^* =\emptyset$. Indeed, each of the $\delta \BB r\times\delta \BB r$ squares $S$ in the union defining $U$ is contained in $\BB C\setminus \mcl B_{\tau_{\BB r}}^\bullet$ and has Euclidean diameter at most $\sqrt 2 \delta \BB r < \ep \BB r / 4$. If one of these squares intersected $\mcl B^*$, then by the triangle inequality and the definition~\ref{eqn-d^U-def} of $d^{\BB C\setminus\mcl B_{\tau_{\BB r}}^\bullet}$, the $d^{\BB C\setminus\mcl B_{\tau_{\BB r}}^\bullet}$-distance from $\pi$ to $\bdy\mcl B_{\tau_{\BB r}}^\bullet\setminus I$ would be at most $\ep \BB r/2$, contrary to the definition of $\pi$. 

Hence $U\cap \left( \mcl B_{\tau_{\BB r}}^\bullet \cup \mcl B^* \right) = \emptyset$. 
Since $\mcl B_{\tau_{\BB r}}^\bullet$ is a local set for $h$ (Lemma~\ref{lem-ball-local}) and $\pi$ is determined by $(\mcl B_{\tau_{\BB r}}^\bullet, h|_{\mcl B_{\tau_{\BB r}}^\bullet})$, for $\frk U \in \mcl U_{\BB r}$, the event $\{U = \frk U\}$ is determined by $h|_{\BB C\setminus \frk U}$. 
Therefore, the bound~\eqref{eqn-cond-diam-pos'} of Lemma~\ref{lem-cond-diam-pos} holds a.s.\ for our (random) choice of $U$. 
\medskip

\noindent\textit{Step 2: bounding conditional probabilities.}
By~\eqref{eqn-pos-kill-prob}, we have $\BB P[E_{\BB r}^U] \geq p$, so Markov's inequality (applied to the random variable $\BB P[ (E_{\BB r}^U)^c \,\big|\, \mcl B_{\tau_{\BB r}}^\bullet, h|_{\mcl B_{\tau_{\BB r}}^\bullet}]$, which has mean at most $1-p$) implies that 
\eqb \label{eqn-use-pos-kill-prob}
\BB P\left[ \BB P\left[ E_{\BB r}^U \,\big|\, \mcl B_{\tau_{\BB r}}^\bullet, h|_{\mcl B_{\tau_{\BB r}}^\bullet} \right] \geq 1-(1-p)^{1/2} \right] \geq 1- (1-p)^{1/2} .
\eqe 
By the locality of $h|_{\mcl B_{\tau_{\BB r}}^\bullet}$ (Lemma~\ref{lem-ball-local}), $\sigma(\mcl B_{\tau_{\BB r}}^\bullet, h|_{\mcl B_{\tau_{\BB r}}^\bullet}) \subset \sigma(h|_{\BB C\setminus U})$.  
By Lemma~\ref{lem-cond-diam-pos} and the preceding sentence the conditional probability given $(\mcl B_{\tau_{\BB r}}^\bullet, h|_{\mcl B_{\tau_{\BB r}}^\bullet})$ and $E_{\BB r}^U$ that
\eqb \label{eqn-cond-diam-pos-event}
\max_{u,v\in \mcl V(U)} \sup_{u,v \in V} D_h\left( u , v \right) \leq \frac{c}{2} \frk c_{\BB r} e^{\xi h_{\BB r}(z)} 
\eqe 
is at least $\frk p$, where $\mcl V(U)$ is as in Lemma~\ref{lem-cond-diam-pos}. 
By this and~\eqref{eqn-pos-kill-prob} and since $p$ can be made arbitrarily close to 1, to conclude the proof of Lemma~\ref{lem-pos-kill} we only need to show that if $E_{\BB r}^U$ occurs and~\eqref{eqn-cond-diam-pos-event} holds, then every $D_h$-geodesic from 0 to a point of $\BB C\setminus B_{A \BB r}(0)$ passes through $I$.
This will be accomplished via a similar argument to the proof of Lemma~\ref{lem-geo-kill-pt}, as we now explain.
\medskip

\noindent\textit{Step 3: preventing $D_h$-geodesics from hitting $\bdy\mcl B_{\tau_{\BB r}}^\bullet\setminus I$.}
By the definitions of $\pi$ and $U$, each square $S\in \mcl S_{\delta {\BB r}}(B_{A \BB r}(0))$ hit by $\pi$ which is \emph{not} included in the union defining $U$ intersects $I$.
Note that such a square cannot intersect $\bdy\mcl B_{\tau_{\BB r}}^\bullet\setminus I $ without intersecting $I$ since then the $d^{\BB C\setminus \mcl B_{\tau_{\BB r}}^\bullet}$-distance from $\pi$ to $\bdy\mcl B_{\tau_{\BB r}}^\bullet\setminus I $ would be smaller than $2\sqrt 2\delta \BB r < \ep \BB r$, contrary to the definition of $\pi$.
Since $\pi$ is connected, it follows that each connected component of $U$ shares a boundary point with a square $S\in\mcl S_{\delta \BB r}(B_{A \BB r}(0))$ which intersects $I$.  
The event in~\eqref{eqn-pos-kill-ball} shows that this square has $D_h$-diameter at most $\frac{c}{100} \frk c_{\BB r} e^{\xi h_{\BB r}(0)} $.
This together with the event in~\eqref{eqn-cond-diam-pos-event} shows that  
\eqb \label{eqn-ball-complement-diam'}
\sup_{u \in U} D_h\left(u , I \right)  < c \frk c_{\BB r} e^{\xi h_{\BB r}(0)}  .
\eqe

By our choice of $U$, each path $P$ from a point outside of $B_{A \BB r}(0)$ to 0 which first hits $\bdy\mcl B_{\tau_{\BB r}}^\bullet$ at a point not in $I$ must pass through $U$ and then must subsequently cross from a point of $\BB C\setminus (\mcl B_*\cup \mcl B_{\tau_{\BB r}}^\bullet)$ to $\bdy\mcl B_{\tau_{\BB r}}^\bullet \setminus I$. 
By~\eqref{eqn-ball-complement-diam'}, the $D_h$-distance from the first point of $U$ hit by $P$ to 0 is strictly smaller than $\tau_{\BB r} + c \frk c_{\BB r} e^{\xi h_{\BB r}(0)} $. On the other hand,~\eqref{eqn-pos-kill-across} shows that the $D_h$-length of the segment of $P$ which crosses from $\BB C\setminus (\mcl B_*\cup \mcl B_{\tau_{\BB r}}^\bullet)$ to $\bdy\mcl B_{\tau_{\BB r}}^\bullet \setminus I$ is at least $ c \frk c_{\BB r} e^{\xi h_{\BB r}(0)} $, so the $D_h$-length of the segment of $P$ after it first hits $U$ is at least $\tau_{\BB r} +  c \frk c_{\BB r} e^{\xi h_{\BB r}(0)} $. 
Therefore, $P$ cannot be a $D_h$-geodesic. 
\end{proof}

\subsection{Killing off all but one geodesic almost surely}
\label{sec-one-geo-as}

By combining Theorem~\ref{thm-finite-geo} and Lemma~\ref{lem-pos-kill}, we get the following lemma. 

\begin{lem} \label{lem-one-point}
There exists $M > 1$ and $\frk q  > 0$, depending only on $\xi$, such that for each $\BB r> 0$, it holds with probability at least $\frk q$ that there is a single point $x\in\bdy\mcl B_{\tau_{\BB r}}^\bullet$ which is hit by every leftmost $D_h$-geodesic from 0 to $\bdy \mcl B_{\tau_{M \BB r}}^\bullet$.  
\end{lem}
\begin{proof}  
By Theorem~\ref{thm-finite-geo} applied with $\tau = \tau_{\BB r}$ (and Axiom~\ref{item-metric-coord} to lower-bound $\tau_{2\BB r} -\tau_{\BB r}$), there is an $N\in\BB N$ such that for each $\BB r > 0$, it holds with probability at least $7/8$ that there are only $N$ points of $\bdy\mcl B_{\tau_{\BB r}}^\bullet$ which are hit by leftmost $D_h$-geodesics from 0 to $\bdy\mcl B_{\tau_{2 \BB r}}^\bullet$.   
Let $X \subset\bdy\mcl B_{\tau_{\BB r}}^\bullet$ be the set of such points, and note that $X$ is determined by $(\mcl B_{\tau_{2\BB r} }^\bullet , h|_{\tau_{2\BB r}}^\bullet)$. 
For $x\in X$, let $I_x$ be the set of $y\in\bdy\mcl B_{\tau_{2\BB r}}^\bullet$ for which there is a leftmost $D_h$-geodesic from 0 to $y$ which passes through $x$.
By Lemma~\ref{lem-geo-arc}, the $I_x$'s are disjoint arcs of $\bdy\mcl B_{\tau_{2 \BB r}}^\bullet$ and by the definition of $X$ their union is all of $\bdy\mcl B_{\tau_{2 \BB r}}^\bullet$.  

By Axiom~\ref{item-metric-coord}, there is an $a  \in (0,1)$ such that with probability at least $7/8$, $B_{a \BB r}(0) \subset \mcl B_{\tau_{\BB r}}^\bullet$. 
If this is the case, then in the notation~\eqref{eqn-outrad-inrad} we have $ \op{outrad}(\mcl B_{\tau_{\BB r}}^\bullet) / \op{inrad}(\mcl B_{\tau_{\BB r}}^\bullet)  \leq 1/a$. 
Lemma~\ref{lem-exposed-arc-invert} applied with $K = \mcl B_{\tau_{\BB r}}^\bullet$ and $\mcl I = \{I_x\}_{x\in X}$ therefore shows that there exists $\ep = \ep(a , N) > 0$ such that whenever $\# X\leq N$ and $B_{a \BB r}(0) \subset \mcl B_{\tau_{2\BB r}}^\bullet$ (which happens with probability at least $3/4$), there exists $x\in X$ such that the following is true. 
The arc $I_x$ is not disconnected from $\infty$ in $\BB C\setminus \mcl B_{\tau_{2\BB r}}^\bullet$ by $\mcl B_{\ep \BB r}(\bdy\mcl B_{\tau_{\BB r}}^\bullet \setminus I ;  d^{\BB C\setminus \mcl B_{\tau_{\BB r}}^\bullet}  )$.
Choose such an $x$ in a manner depending only on $(\mcl B_{\tau_{2\BB r} }^\bullet , h|_{\tau_{2\BB r}}^\bullet)$. 
  
By Lemma~\ref{lem-pos-kill} (applied with $2\BB r$ in place of $\BB r$ and with $A=3/2$), we can find $\frk p = \frk p(\ep ) > 0$ such that with probability at least $3/4$, it holds with conditional probability at least $\frk p$ given $(\mcl B_{\tau_{2 \BB r}}^\bullet , h|_{\mcl B_{\tau_{2 \BB r}}}^\bullet)$ that every $D_h$-geodesic from 0 to a point of $\BB C\setminus B_{3 \BB r}(0)$ passes through $I_x$. 
By the definition of $I_x$ and the estimates for the probabilities of the events above, it holds probability at least $\frk p/4$ that every leftmost $D_h$-geodesic from 0 to a point of $\BB C\setminus B_{3 \BB r}(0)$ passes through $x$. 

By Axiom~\ref{item-metric-coord}, we can find $M > 4$ such that for every $\BB r>0$, it holds with probability at least $1-\frk p / 8$ that $\mcl B_{\tau_{M \BB r}}^\bullet \supset B_{4\BB r}(0)$. 
Then with probability at least $\frk p/4 - \frk p /8 = \frk p/8$, each leftmost $D_h$-geodesic from 0 to $\bdy\mcl B_{\tau_{M \BB r}}^\bullet$ passes through $x$. 
Hence the statement of the lemma is true with $\frk q = \frk p /8$. 
\end{proof}
 
We now upgrade from the statement that all \emph{leftmost} $D_h$-geodesics from 0 to $\bdy\mcl B_{\tau_{M \BB r}}^\bullet$ hit the same point of $\bdy\mcl B_{\tau_{\BB r}}^\bullet$ to the statement that \emph{all} $D_h$-geodesics from 0 to $\bdy\mcl B_{\tau_{M \BB r}}^\bullet$ coincide until they hit $\bdy\mcl B_{\tau_{\BB r}}^\bullet$. 

\begin{lem} \label{lem-one-geo}
There exists $M > 1$ and $\frk q  > 0$, depending only on $\xi$, such that for each $\BB r> 0$, it holds with probability at least $\frk q$ that any two $D_h$-geodesics from 0 to a point of $\BB C\setminus \mcl B_{\tau_{M \BB r}}^\bullet$ coincide on the time interval $[0,\tau_{\BB r}]$. 
\end{lem}
\begin{proof}
By Lemma~\ref{lem-one-point}, there exists $M$ and $\frk q$ as in the statement of the lemma such that for each $\BB r> 0$, it holds with probability at least $\frk q$ that there is a single point $x\in\bdy\mcl B_{\tau_{\BB r}}^\bullet$ which is hit by every leftmost $D_h$-geodesic from 0 to $\bdy \mcl B_{\tau_{M \BB r}}^\bullet$. 
Henceforth assume that this is the case. 
By the uniqueness of geodesics to points in $\BB Q^2$ (Lemma~\ref{lem-geo-unique}), every $D_h$-geodesic from 0 to a point in $\BB Q^2\setminus \mcl B_{\tau_{M \BB r}}^\bullet$ must pass through $x$. By Lemma~\ref{lem-leftmost-geodesic}, leftmost and rightmost geodesics can be approximated by geodesics to points in $\BB Q^2\setminus \mcl B_{\tau_{M \BB r}}^\bullet$, so it follows that every rightmost geodesic and every leftmost geodesic from 0 to $\bdy\mcl B_{\tau_{M \BB r}}^\bullet$ must pass through $x$. 
This implies that a.s.\ every $D_h$-geodesic from 0 to $\bdy \mcl B_{\tau_{M \BB r}}^\bullet$ must pass through $x$, so the restriction of any such geodesic to the time interval $[0,\tau_{\BB r}]$ is a $D_h$-geodesic from 0 to $x$. 

We now argue that there is only one $D_h$-geodesic from 0 to $x$, so that any two $D_h$-geodesics from 0 to $\bdy\mcl B_{\tau_{M \BB r}}^\bullet$ coincide on $[0,\tau_{\BB r}]$. To this end, choose a point $q \in \BB Q^2\setminus \mcl B_{\tau_{M \BB r}}^\bullet$.
By Lemma~\ref{lem-geo-unique}, the $D_h$-geodesic from 0 to $q$ is a.s.\ unique. If there were more than one $D_h$-geodesic from 0 to $x$, then by concatenating such geodesics with a fixed geodesic from $x$ to $q$, we would get multiple distinct geodesics from 0 to $q$, so there must be only one $D_h$-geodesic from 0 to $x$.
\end{proof}

We now conclude the proof via a zero-one law argument. 

\begin{proof}[Proof of Theorem~\ref{thm-clsce}]
By Axiom~\ref{item-metric-translate} and the translation invariance of the law of $h$, viewed modulo additive constant, we can assume without loss of generality that $z = 0$.
Fix $M$ and $\frk q$ as in Lemma~\ref{lem-one-geo} and for $\BB r >0$ let $F_{\BB r}$ be the event that any two $D_h$-geodesics from 0 to a point of $\bdy \mcl B_{\tau_{M \BB r}}^\bullet$ coincide on the time interval $[0,\tau_{\BB r}]$. Then $\BB P[F_{\BB r}] \geq \frk q$ and $F_{\BB r}$ is determined by $(\mcl B_{\tau_{M \BB r}}^\bullet, h|_{\mcl B_{\tau_{M \BB r}}^\bullet})$ and hence by $h|_{B_{M \BB r}(0)}$.  
Since the tail $\sigma$-algebra $\bigcap_{\BB r > 0} \sigma(h|_{B_{\BB r}(0)})$ is trivial, there a.s.\ exists arbitrarily small values of $\BB r> 0$ for which $F_{\BB r}$ occurs.

Since each LQG metric ball centered at 0 contains a Euclidean ball centered at 0, it is a.s.\ the case that for each $s> 0$, there exists $\BB r > 0$ for which $\mcl B_{\tau_{M \BB r}}^\bullet \subset\mcl B_s(0; D_h)$ and $F_{\BB r}$ occurs. Then any two $D_h$-geodesic from 0 to a point outside of $\mcl B_s(0;D_h)$ coincide on $[0,\tau_{\BB r}]$, so the theorem is true with $t = \tau_{\BB r}$. 
\end{proof}

\bibliography{cibiblong,cibib}
\bibliographystyle{hmralphaabbrv}

\end{document}